\documentclass[10pt]{amsart}
\usepackage{ amssymb} 
\usepackage{ mathrsfs} 
\usepackage{ amscd}

\allowdisplaybreaks[4]
\usepackage{ascmac}

\theoremstyle{definition}
\numberwithin{equation}{section}
\newtheorem{thm}{Theorem}[section]
\newtheorem{dfn}[thm]{Definition}

\newtheorem{prop}[thm]{Proposition}
\newtheorem{cor}[thm]{Corollary}
\newtheorem{lem}[thm]{Lemma}
\newtheorem{conj}[thm]{Conjecture}
\newtheorem{rem}[thm]{Remark}

\renewcommand{\hat}{\widehat}

\title{A construction of $p$-adic Asai $L$-functions for ${\rm GL}_2$ over CM fields}
\author[K. Namikawa]{Kenichi Namikawa}
\address{ Faculty of Mathematics, Kyushu University, 744 Motooka, Nishi-Ku, Fukuoka, 819-0395, Japan
}
\email{namikawa@math.kyushu-u.ac.jp}

\begin{document}

\maketitle

\begin{abstract}
We give a construction of $p$-adic Asai $L$-functions for cohomological cuspidal automorphic representations of ${\rm GL}_2$ over CM fields.
If the base field is imaginary quadratic, 
Loeffler-Williams recently constructed the $p$-adic Asai $L$-functions. 
We generalize their construction to the case that the base fields are general CM fields.
\end{abstract}

\tableofcontents

\section{Introduction}\label{sec:intro}

The study of $p$-adic $L$-functions have been considered in a lot of literatures  
for its importance in arithmetic applications.  
It is believed that $p$-adic $L$-functions should contain information of certain $p$-adic families of arithmetic objects
according to the philosophy of Iwasawa theory. 
However it is a difficult subject to construct $p$-adic $L$-functions, 
even though we have plenty results on complex $L$-functions in the theory of automorphic representations. 

In this paper, we present a construction of $p$-adic analogue of  Asai $L$-functions, which are also called twisted tensor $L$-functions, attached to 
irreducible cohomological cuspidal automorphic representations $\pi$ of ${\rm GL}_2$ over CM fields.  
If the base field is imaginary quadratic, we can find a recent construction by D. Loeffler and C. Williams (\cite{lw}).  
The method of their construction relies on the theory of Siegel units on modular curves, 
however its analogue for Hilbert modular settings is not yet known in the literature so far.   
Hence the generalization of the construction in \cite{lw} for general CM fields is not an obvious subject. 
In this paper, we will study a generalization of Eisenstein series associated with Siegel units 
which are considered in \cite{lw} (or more originally in \cite{bl94}, \cite{ka04}, \cite{llz})
 by using a technique which is developed in \cite{ha87}
 and we give a generalization of Loeffler-Williams's construction in the case of general CM extensions.

Let us write down the main result in this paper precisely. 
Fix an embedding $\overline{\mathbf Q} \to {\mathbf C}$ and fix an isomorphism ${\mathbf i}_p: {\mathbf C} \to {\mathbf C}_p$ for an odd prime $p$. 
Consider a CM field $E\subset \overline{\mathbf Q}$, that is, $E$ is a totally imaginary quadratic extension of a totally real field $F$. 
Suppose that $p$ is unramified in $E/{\mathbf Q}$.
Denote by $\Sigma_{E}$ (resp. $\Sigma_F$) the set of places of $E$ (resp. $F$).
Let $\pi=\otimes^\prime_{w\in \Sigma_E} \pi_w$ be an irreducible unitary cohomological cuspidal automorphic representation of ${\rm GL}_2(E_{\mathbf A})$
     with the central character $\omega_\pi$. 
Note that $\pi$ has the Langlands parameter 
\begin{align*}
  W_{E_\sigma}\cong {\mathbf C}^\times \to {\rm GL}_2({\mathbf C}); z \longmapsto   \begin{pmatrix} (z/z^c)^{\frac{n_\sigma+1}{2}} & 0 \\ 0 &  (z^c/ z)^{ \frac{n_\sigma+1}{2} } \end{pmatrix}
\end{align*}
for each infinite place $\sigma \in \Sigma_E$, where $W_{E_\sigma}$ is the Weil group of $E_\sigma$, 
$n_\sigma \in {\mathbf Z}, n_\sigma\geq 0$  
and $c:{\mathbf C} \to {\mathbf C}$ is the complex conjugation.
Assume that $p>{\rm max}\{ n_\sigma \mid \sigma \in \Sigma_E, \sigma \mid \infty  \}$.
For each ideal ${\mathfrak N} \subset \widehat{\mathcal O}_E$, define a subgroup ${\mathcal K}_1({\mathfrak N})$ of ${\rm GL}_2(\widehat{\mathcal O}_E)$ to be 
\begin{align*}
  {\mathcal K}_1({\mathfrak N})  
  = \left\{ \begin{pmatrix} a & b \\ c & d \end{pmatrix} \in {\rm GL}_2(\widehat{\mathcal O}_E)  \middle| c, d-1 \in {\mathfrak N}   \right\}.  
\end{align*}
Suppose that $\pi$ has the conductor ${\mathfrak N}$, that is, $\pi$ has a ${\mathcal K}_1({\mathfrak N})$-fixed vector and ${\mathfrak N}$ is minimal among such ideals. 
We also suppose that ${\mathfrak N}$ is prime to the discriminant of $E/F$.

Denote by $L(s, {\rm As}^+(\pi)) = \prod_{v\in \Sigma_F}  L(s, {\rm As}^+(\pi)_v)$ to be the Asai $L$-function which is associated with $\pi$. 
Recall that, if $v\in \Sigma_F$ is split in $E$ then the local factor $L(s, {\rm As}^+(\pi)_v)$ at $v$ of the Asai $L$-function 
  coincides with the Rankin-Selberg $L$-function $L(s, \pi_{w}\otimes \pi_{w_c})$ where $v=ww_c$ for $w, w_c\in \Sigma_E$.
Denote by ${\rm As}^+_{\mathcal M}(\pi)$ the conjectural Asai motive over $F$ whose $L$-function is given by $L(s+1, {\rm As}^+(\pi))$. 
See Section \ref{sec:AsaiMot}, where we briefly introduce a description of this Asai motive ${\rm As}^+_{\mathcal M}(\pi)$. 
Note that,  for each $0\leq \alpha \leq n_{\rm min}:={\rm min}\{ n_\sigma | \sigma \in \Sigma_E, \sigma\mid \infty \}$ and each finite-order Hecke character 
$\varphi: F^\times \backslash F^\times_{\mathbf A} \to {\mathbf C}^\times$ such that $(-1)^{n_{\rm min}-\alpha} \varphi(-1_\sigma)=1$ for each infinite place $\sigma \in \Sigma_F$, 
$L(s, {\rm As}^+_{\mathcal M}(\pi)\otimes |\cdot|^{n_{\rm min}-\alpha}_{F_{\mathbf A}} \varphi  )$ is critical at $s=0$ in the sense of Deligne (\cite{de79}).    

Let $\mu_{p^\infty}\subset \overline{\mathbf Q}$ be the group of $p$-power roots of $1$. 
Fix a finite extension $K_\pi$ of ${\mathbf Q}_p$ so that $K_\pi$ contains all conjugates of $E$, all Hecke eigenvalues of $\pi$ and the values of $\omega_\pi$. 
Denote by ${\mathcal O}_\pi$ the ring of integers of $K_\pi$.

The following statement is the main theorem in this paper:

\begin{thm}(Theorem \ref{thm:Main})
{\itshape 
Let ${\mathfrak N}_F={\mathfrak N}\cap \widehat{\mathcal O}_F$. 
Assume that 
\begin{itemize}
\item $\pi$ is nearly $p$-ordinary;
\item $\omega_{\pi,p}$ is unramified;
\item   if $\pi$ is conjugate self-dual, then $\alpha \neq n_{\rm min}$; 
 \item the conductor ${\mathfrak N}$ of $\pi$ is square-free;
\item for each $v\mid {\mathfrak N}_F$ with $v\nmid p$, suppose either one of the following conditions: 
           \begin{itemize}
            \item $\omega_{\pi, v}$ is ramified;  
            \item if $\omega_{\pi, v}$ is unramified, 
                     then $v=ww_c\mid {\mathfrak N}_F$ splits in $E/F$
                     and one of $\pi_w$ and $\pi_{w_c}$ is an unramified principal representation 
                    and the other is a special representation;
           \end{itemize}     
\end{itemize}
Then, 
for each $0\leq \alpha \leq n_{\rm min}$, 
there exists an element 
$      {\mathscr L}^\alpha_{p}({\rm As}^+_{\mathcal M}(\pi)) 
     \in  K_\pi \otimes_{{\mathcal O}_\pi}  {\mathcal O}_\pi  [[ {\rm Gal}( F(\mu_{p^\infty})/ F)   ]]$    
satisfying the following interpolation formula:
\begin{quote}
For each finite-order Hecke character $\varphi$ of a $p$-power conductor 
satisfying that $(-1)^{n_{\rm min} - \alpha }\varphi(-1_\sigma)=1$ for each infinite place $\sigma \in \Sigma_F$, 
we have
\begin{align*}
      \widehat{\phi} ( {\mathscr L}^\alpha_{p}({\rm As}^+_{\mathcal M}(\pi))  )   
    =&      {\mathcal E}_\infty({\rm As}^+_{\mathcal M} (\pi) ( \phi ))
        {\mathcal E}_p({\rm As}^+_{\mathcal M} (\pi) ( \phi )) 
        \frac{  L( 0, {\rm As}^+_{\mathcal M} (\pi)(\phi)) }{ \Omega( {\rm As}^+_{\mathcal M}(\pi)  )},   
\end{align*}
where $\phi=|\cdot|^{n_{\rm min} -\alpha}_{\mathbf A}\varphi$ and $\widehat{\phi}$ is the $p$-adic avatar of $\phi$, and 
\begin{itemize}
 \item ${\mathcal E}_\infty({\rm As}^+_{\mathcal M} (\pi) ( \phi ))$ and ${\mathcal E}_p({\rm As}^+_{\mathcal M} (\pi) ( \phi ))$ are modified Euler factor at $\infty$ and $p$ respectively; 
 \item $\Omega( {\rm As}^+_{\mathcal M}(\pi)  )$ 
           is the period of ${\rm As}^+_{\mathcal M}(\pi)$ due to Coates {\rm (\cite[page 107]{co89})}, 
          which is a product of Deligne's period $c^+( {\rm As}^+_{\mathcal M}(\pi) )$ of ${\rm As}^+_{\mathcal M}(\pi)$ and a power of the circular constant. 
\end{itemize}
   {\rm (}The definitions of ${\mathcal E}_\infty({\rm As}^+_{\mathcal M} (\pi) ( \phi )), {\mathcal E}_p({\rm As}^+_{\mathcal M} (\pi) ( \phi ))$ and $\Omega( {\rm As}^+_{\mathcal M}(\pi)  )$ 
               are recalled in Section \ref{sec:CPconj} according to \cite{co89}.{\rm )}
\end{quote}
}
\end{thm}

\begin{rem}\label{rem:mainthm}
\begin{enumerate}
\item The algebraicity of critical values of Asai $L$-functions $L(s, {\rm As}^+_{\mathcal M}(\pi))$ is proved by Ghate in \cite{gh99} and \cite{gh99b}. 
         By using the same method with Ghate's works, it can be proved the algebraicity of critical values of Asai $L$-functions $L(s, {\rm As}^+_{\mathcal M}( \pi ) \otimes \varphi)$ 
         twisted by Hecke characters $\varphi: F^\times \backslash {F}^\times_{\mathbf A} \to {\mathbf C}^\times$. 
\item By multiplying a denominator of an Eisenstein cohomology class, 
          the $p$-adic Asai $L$-functions can be considered as elements in ${\mathcal O}_\pi  [[ {\rm Gal}( F(\mu_{p^\infty})/ F)   ]]$  
          (see Remark \ref{rem:denom}). 
\item The basic strategy of the construction of $p$-adic Asai $L$-functions is based on 
          a study of a certain good family of Eisenstein cohomology classes which satisfies a distribution relation.   
          If the base field is the rational number field, the Eisenstein cohomology classes with a stabilization in \cite{lw} have an integral coefficient,  
          since they have a motivic construction due to \cite{bl94}.
         However it is not clear whether there exists a good family of Eisenstein cohomology classes with integral coefficients in Hilbert modular settings, 
         since a generalization of the construction in \cite{bl94} is not yet known so far.
         In the present paper, we will introduce a generalization of Eisenstein series coming from Siegel units and 
         we will prove that these Eisenstein series define cohomology classes with rational coefficients according to the method in \cite{ha87}.  
         Moreover we show that its denominators are bounded in a uniform way (see Section \ref{sec:bdddenom}).            
\item The proof of interpolation formula in \cite{lw} is exploited by using a global method, which is the same strategy with \cite{gh99}, 
         that is, which relies on a coordinate of the Poincar$\acute{\rm e}$ upper 3-space. 
         However, it is desired to give a proof in the adelic language for the further developments of this study. 
         In this paper, we give a proof of both the algebraicity of critical values and the interpolation formula of $p$-adic $L$-functions 
         in terms of the automorphic representation theory. 
         The method in this paper will show that  
          the arguments in \cite{gh99} and \cite{lw} can be explained in a more concise way.  
\item \label{rem:mainthmTw} 
        Define ${\rm Tw}_p:K_\pi \otimes_{\mathcal O_\pi}{\mathcal O}_\pi [[ {\rm Gal}(F(\mu_{p^\infty})/F)]] 
                                       \to K_\pi \otimes_{\mathcal O_\pi}{\mathcal O}_\pi [[ {\rm Gal}(F( \mu_{p^\infty})/F)]]$ to be 
        \begin{align*}
            {\rm Tw}_p (g) =  \varepsilon_{\rm cyc}(g) g,   
        \end{align*}  
        where $\varepsilon_{\rm cyc}: {\rm Gal}(F(\mu_{p^\infty})/F) \to {\mathbf C}^\times_p$ is the $p$-adic cyclotomic character.
        Then we expect that 
        \begin{align} \label{eq:mamin}
            {\rm Tw}^{\alpha^\prime -\alpha }_p ( {\mathscr L}^\alpha_p({\rm As}^+_{\mathcal M}(\pi) )  ) 
            = {\mathscr L}^{\alpha^\prime}_p({\rm As}^+_{\mathcal M}(\pi) )  
        \end{align}
        for each $0\leq \alpha, \alpha^\prime \leq n$.
        If the base field $F$ is the rational number field, the above identity is established in \cite[Proposition 5.6]{lw}.  
        In the Hilbert modular setting, 
           it seems to be difficult to prove the identity (\ref{eq:mamin}) so far.  
        Since the weight of Eisenstein series in the construction of ${\mathscr L}^\alpha_p({\rm As}^+_{\mathcal M}(\pi) )$ depends on $\alpha$,    
        we have to compare the denominator of the cohomology classes of Eisenstein series of different weights. 
        Since this might be necessary more subtle arguments, we postpone this property for the future work.   
\item There are another works \cite{ba17} and \cite{bgv} on the similar topics. 
        Their choices of Eisenstein series in \cite{ba17} and \cite{bgv} are different from ours, 
        and they do not deduce that their distributions are in an Iwasawa algebra over a certain integral coefficients.    
       The main theorem (Theorem \ref{thm:Main}) in this paper refines the results in \cite{ba17} and \cite{bgv}.
\end{enumerate}
\end{rem}

The organization of this paper is as follows. 
After fixing some basic notations in Section \ref{sec:not}, 
we introduce our settings on Asai representations in Section \ref{sec:AsaiSet}.
The conjecture on the existence of $p$-adic $L$-functions attached to motives due to Coates and Perrin-Riou is reviewed in Section \ref{sec:CPconj}. 
The construction of $p$-adic Asai $L$-functions is based on a study of Eisenstein series and its associated cohomology classes. 
In Section \ref{sec:Eis}, we introduce our distinguished Eisenstein series, which give a generalization of Eisenstein series coming form Siegel units. 
In Section \ref{sec:Coh}, we introduce certain Eisenstein cohomology classes according to the method in \cite{ha87}. 
In Section \ref{sec:AsaiInt}, we give an integral expression of Asai $L$-functions in terms of these cohomology classes.
In Section \ref{sec:pAsai}, we construct an element ${\mathscr L}^\alpha_p({\rm As}^+_{\mathcal M}(\pi))$ in an Iwasawa algebra 
      and we state its characterization in terms of critical values of Asai $L$-functions (see Theorem \ref{thm:Main}). 
Section \ref{sec:urInt}, \ref{sec:pInt} and \ref{sec:InfInt} are devoted to prove the interpolation formulas for $p$-adic Asai $L$-functions.

\section{Notation}\label{sec:not}

\subsection{Basic notation}

Denote by ${\mathbf Q}$ (resp. ${\mathbf Z}, {\mathbf R}, {\mathbf C}$) the rational number field 
(resp. the ring of the rational integers, the real number field, the complex number field).
The binomial coefficient $\binom{a}{b}$ for $a,b\in {\mathbf Q}$ is defined to be 
\begin{align*}
   \binom{a}{b} = \begin{cases} \frac{\Gamma(a+1)}{\Gamma(b+1)\Gamma(a-b+1)},  &  (a, b \in {\mathbf Z}),   \\
                                                        0, &  \text{(otherwise)}. \end{cases}
\end{align*}

For a number field $L$, we denote by ${\mathcal O}_L$ 
(resp.  $L_{\mathbf A}$, $\Sigma_L$, $I_L$) the ring of integers of $L$ 
(resp. the ring of adeles of $L$, the set of places of $L$, the set of embeddings from $L$ to ${\mathbf C}$).
Define $\Sigma_L({\mathbf R})$ (resp. $\Sigma_L({\mathbf C})$) to be the set of real (resp. complex) places of $L$.
Put $\Sigma_{L,\infty}=\Sigma_L({\mathbf R}) \cup \Sigma_L({\mathbf C})$.
By fixing a representatives, consider $\Sigma_{L,\infty}$ as a subset of $I_L$. 
For each place $v$ of $L$, define $L_v$ to be the completion of $L$ at $v$.  
Let $L_\infty = \prod_{\sigma \in \Sigma_{L,\infty}} L_\sigma$ 
and denote by $L_{\infty, +}$ 
the set of $x\in L_\infty$ such that $x_\sigma>0$ for each $\sigma \in \Sigma_{L,\infty}$. 
Put $L_{\rm fin}$ to be the set of finite adeles of $L_{\mathbf A}$. 
These notations are used for algebraic groups, that is,  
for an algebraic group $G$ over $L$ and an $L$-algebra $A$,  
   we denote by $G(A)$ (resp. $Z_G$) the $A$-rational points of $G$ (resp. the center of $G$).
For $x \in G(L_{\mathbf A})$, we write $x_{\rm fin}$ (resp. $x_\infty$) for the projection of $x$ to $G( L_{ {\mathbf A}, {\rm fin} } )$ (resp. $G(L_\infty)$).  

Let ${\rm Tr}_{L/{\mathbf Q}}$ (resp. ${\rm Nr}_{L/{\mathbf Q}}$) be the trace (resp. norm) map of $L/{\mathbf Q}$.  
Let $\psi_{\mathbf Q}=\prod_{v\in \Sigma_{\mathbf Q}}\psi_{\mathbf Q, v}:{\mathbf Q} \backslash {\mathbf Q}_{\mathbf A}\to {\mathbf C}^\times$ be the additive character 
satisfying $\psi_{\mathbf Q}(x_\infty)=\exp(2\pi\sqrt{-1}x_\infty)$ for $x_\infty\in {\mathbf R}={\mathbf Q}_\infty$.
Define an additive character $\psi_L: L \backslash L_{\mathbf A} \to {\mathbf C}^\times$ to be $\psi_L=\psi_{\mathbf Q}\circ {\rm Tr}_{L/{\mathbf Q}}$.

Put $\widehat{\mathbf Z}=\prod_{v\in \Sigma_{\mathbf Q}, v<\infty} {\mathbf Z}_v$   
    and $\widehat{\mathcal O}_L={\mathcal O}_L\otimes_{\mathbf Z} \widehat{\mathbf Z}$.  
For each finite place $v\in \Sigma_L$, denote by ${\mathcal O}_{L,v}$ (resp. $\varpi_v$) the ring of integers  (resp. a uniformizer) of $L_v$. 
Put $q_v=\sharp {\mathcal O}_{L,v}/\varpi_v {\mathcal O}_{L,v}$.

We frequently use the multi-index notation. 
For instance, let $x= (x_v)_{v\in \Sigma_L} \in L_{\mathbf A}$ and write  
\begin{align*}
   x_\infty = (x_\sigma)_{\sigma \in \Sigma_{L, \infty}},  \quad 
   x^n_\infty = \prod_{\sigma \in \Sigma_L({\mathbf R})} x^{n_\sigma}_\sigma \prod_{v\in \Sigma_L({\mathbf C})} x^{n_\sigma}_\sigma \overline{x}^{n_{c\sigma}}_\sigma,    
\end{align*}
where $n=\sum_{\sigma\in I_L} n_\sigma \sigma \in {\mathbf Z}[I_L]$ and $c:{\mathbf C} \to {\mathbf C}$ is the complex conjugation.
Similarly, we put
\begin{align*}
  L_{{\mathbf A}, \infty }=  \prod_{v\in \Sigma_{L, \infty}}  L_v  , \quad 
  L_{{\mathbf A}, {\rm fin}}
  =L^{(\infty)}_{\mathbf A}=  \widehat{\mathcal O}_L\otimes_{{\mathcal O}_L} L.   
\end{align*}
For each $\alpha = \sum_{\sigma\in I_L} \alpha_\sigma \sigma, 
                 \alpha^\prime = \sum_{\sigma\in I_L} \alpha^\prime_\sigma \sigma \in {\mathbf Z}[I_L]$, 
    we write $\alpha < \alpha^\prime$ if $\alpha_\sigma < \alpha^\prime_\sigma$ for each $\sigma \in I_L$.                 

In this paper, $L$-functions are always the complete ones. 
For instance, we define the Riemann zeta function $\zeta$ as follows:
\begin{align*}
   \zeta_\infty(s) = \Gamma_{\mathbf R}(s) = \pi^{-\frac{s}{2}} \Gamma\left(\frac{s}{2}\right), 
   \quad 
   \zeta_p(s) = \frac{1}{1-p^{-s}}, 
   \quad 
   \zeta(s) = \prod_{v\in \Sigma_{\mathbf Q}} \zeta_v(s), \quad 
   \zeta^{(S)}(s) = \prod_{v\in \Sigma_{\mathbf Q} \backslash S} \zeta_v(s) \ (S\subset \Sigma_{\mathbf Q}). 
\end{align*}

\subsection{Algebraic representation of ${\rm GL}_2$}\label{algrep}
Let $A$ be a ${\mathbf Z}$-algebra.
Let $A[X,Y]_n$ denote the space of two variable homogeneous polynomials of degree $n$ over $A$.
Suppose that $n!$ is invertible in $A$.
We define the perfect pairing $[ \cdot , \cdot ]_n:A[X,Y]_n\times A[X,Y]_n\to A$ by
\begin{align}\label{eq:pairing}
  [ X^iY^{n-i}, X^jY^{n-j} ]_n
 =\begin{cases}  (-1)^i\binom{n}{i}^{-1},  & ( i+j=n), \\
                       0, & ( i+j\neq n). \end{cases} 
\end{align}
Denote by $u^\vee$ the dual of $u\in A[X,Y]_n$ with respect the pairing  $[ \cdot, \cdot ]_n$.

Let $L$ be a number field.
For each $n, m \in {\mathbf Z}[I_L]$, define $L(n; {\mathbf C}) = \otimes_{\sigma \in I_L} L(n_\sigma; {\mathbf C})$. 
Define an action $\varrho^{(\infty)}_{n,m}(g)$ of $g\in {\rm GL}_2(L_\infty)$ by 
\begin{align}\label{eq:rhodfn} 
 \begin{aligned} 
      \varrho^{(\infty)}_{n,m}(g) P\left( \begin{pmatrix} X  \\  Y  \end{pmatrix}  \right) 
    =& P\left(  \prod_{\sigma \in I_L}  {}^\iota g_\sigma \begin{pmatrix} X_\sigma  \\  Y_\sigma  \end{pmatrix}     \right)  \cdot (\det g)^m,  \\
      & ( P\in L(n; {\mathbf C}),  
                \begin{pmatrix} X  \\  Y  \end{pmatrix} = \otimes_{\sigma\in I_L} \begin{pmatrix} X_\sigma  \\  Y_\sigma  \end{pmatrix}, 
                g = \begin{pmatrix} a & b \\ c & d\end{pmatrix}, {}^\iota g  = \begin{pmatrix} d & -b \\ -c & a \end{pmatrix} ).
 \end{aligned} 
\end{align} 
It is well-known that the pairing 
       $[ \cdot,\cdot ]_n$ on $L(n; {\mathbf C})$ satisfies
\begin{align*}
     [ \varrho^{(\infty)}_{n,m}(g)P, \varrho^{(\infty)}_{n,m}(g)Q ]_n
  = (\det g)^{n+2m } \cdot [ P, Q ]_n, 
   \quad (g\in {\rm GL}_2(L_\infty), P, Q \in L(n;{\mathbf C}) ).
\end{align*}

\subsection{$p$-adic avatar}

Throughout this paper, we fix an odd prime $p$ and 
      fix an embedding $\overline{\mathbf Q} \subset {\mathbf C}$ and an isomorphism ${\mathbf i}_p: {\mathbf C} \to {\mathbf C}_p$. 
For each number field $L$, 
define $\Sigma_{L,p}$ to be the set of places of $L$ above $p$
   and let $I_{L_v}$ be the set of the continuous embedding $L_v \to {\mathbf C}_p$ for $v\in \Sigma_{L,p}$.    
For each $x=(x_v)_{v\in \Sigma_L} \in L_{\mathbf A}$, put $x_p=(x_v)_{v\in \Sigma_{L,p}}$.      
For each $x_p=(x_v)_{v\mid p}\in L_p$ and $w\in {\mathbf Z}[I_L]$, define $x^w_p\in {\mathbf C}_p$ to be 
\begin{align*}
  x^w_p   =   \prod_{v\mid p} \prod_{\sigma_v\in I_{L_v}} \sigma_v(x_v)^{w_{\sigma(\sigma_v)}},      
\end{align*}
where $\sigma(\sigma_v) \in I_L$ is the composition of the natural inclusion $L \hookrightarrow L_v$  and $\sigma_v: L_v \to {\mathbf C}_p$.

Let $\phi:L^\times \backslash L^\times_{\mathbf A} \to {\mathbf C}^\times$ be a Hecke character of the infinity type $\phi(x_\infty)= x^w_\infty \ (x\in L^\times_{\infty, +})$
for some $w\in {\mathbf Z}[I_L]$. 
Define the $p$-adic avatar $\widehat{\phi}: L^\times \backslash L^\times_{\mathbf A} \to {\mathbf C}^\times_p$ to be 
\begin{align*}
  \widehat{\phi}(x) 
   =  x^{w}_p   {\mathbf i}_p \left(  x^{-w}_\infty  \phi(x)     \right),    
   \quad (x=(x_v)_{v\in \Sigma_L}).    
\end{align*}
By using the geometrically normalized reciprocity map, 
    we sometimes consider $\widehat{\phi}$ as a character on the absolute Galois group ${\rm Gal}(\overline{L}/L)$.   

We also define the $p$-adic avatar $\varrho^{(p)}_{n,m}$ of $\varrho^{(\infty)}_{n,m}$. 
For $g\in {\rm GL}_2(L_p)$ and $P\in L(n; {\mathbf C}_p) :=\otimes_{\sigma \in I_L} L(n_\sigma; {\mathbf C}_p)$, define 
\begin{align}\label{eq:pact}
  \varrho^{(p)}_{n,m}(g) P \left(  \begin{pmatrix} X \\ Y \end{pmatrix} \right)  
= \prod_{v\mid p}
   \prod_{\sigma_v \in I_{L_v}}    
           P \left(  \sigma_v \left( {}^\iota g_v  \right)  \begin{pmatrix} X_{ \sigma(\sigma_v) } \\ Y_{ \sigma(\sigma_v) } \end{pmatrix} \right)  
           \cdot (\det g)^m.
\end{align}   
Note that for each $g\in {\rm GL}_2(L)$, ${\mathbf i}_p( \varrho^{(\infty)}_{n,m}(g) ) =\varrho^{(p)}_{n,m}(g)$.
Hereafter we write both $\varrho^{(\infty)}_{n,m}$ and $\varrho^{(p)}_{n,m}$ as $\varrho_{n,m}$, since no confusion likely occurs.

\section{Settings on Asai representations}\label{sec:AsaiSet}

\subsection{Automorphic forms on ${\rm GL}_2$}\label{sec:defauto}

We recall the definition of cusp forms on ${\rm GL}_2$ over a number field $E$, 
  which gives a realization of cohomological cuspidal autormophic representation of  ${\rm GL}_2(E_{\mathbf A})$.
Following the description in \cite[Section 2, 3]{hi94}, 
   we briefly recall conventions and we fix some notations.  

Let $n=\sum_{\sigma\in I_E} n_\sigma\sigma $ be an element of ${\mathbf Z}_{\geq 0}[I_E]$ which satisfies the following two conditions:
\begin{itemize}
   \item For all $\sigma, \tau\in I_E$, $n_\sigma\equiv n_\tau$ mod $2$;
   \item For each complex place $\sigma\in \Sigma_E$, $n_\sigma = n_{c \sigma}$, where $c$ is the complex conjugate.
 \end{itemize}
We choose an element  $m=\sum_{\sigma\in I_E} m_\sigma\sigma$  of ${\mathbf Z}_{\geq 0}[I_E]$ which satisfies the following condition:
\begin{itemize}
   \item For all $\sigma, \tau\in I_E$, $n_\sigma+ 2m_\sigma = n_\tau+2m_\tau$.
\end{itemize}
We denote $\Sigma_{\sigma\in I_E}\sigma\in {\mathbf Z}[I_E]$ by $t$.
For $a\in {\mathbf Z}$, we sometimes identify $a$ with $a\cdot t \in {\mathbf Z}\cdot t$.  
We put $\kappa=n+2 m$ and $k=n+2t \in {\mathbf Z}[I_E]$.
We write $[\kappa]=n_\sigma+2m_\sigma$, where $\sigma\in I_E$. 
By the assumption on $m$, $[\kappa]$ is independent of the choice of $\sigma$.

For an ideal $\mathfrak M$ of $\widehat{\mathcal O}_E$, we define some subgroups of ${\rm GL}_2(\hat{\mathcal O}_E)$ by
\begin{align*}
  &  {\mathcal K}_0(\mathfrak M) = \left\{  \begin{pmatrix} a & b \\ c & d  \end{pmatrix} \in {\rm GL}_2(\hat{\mathcal O}_E) \middle| c \equiv 0\mod \mathfrak M    \right\}, \quad
    {\mathcal K}_1(\mathfrak M) = \left\{  \begin{pmatrix} a & b \\ c & d  \end{pmatrix} \in  {\mathcal K}_0(\mathfrak M) \middle| c, d-1 \equiv 0\mod \mathfrak M    \right\}, \\
  &  {\mathcal K}(\mathfrak M) = \left\{  \begin{pmatrix} a & b \\ c & d  \end{pmatrix} \in {\mathcal K}_1({\mathfrak M}) \middle| a-1, b \equiv 0\mod \mathfrak M    \right\}.  
\end{align*}
Let $\Sigma_E({\mathbf R})$ (resp. $\Sigma_E({\mathbf C})$) be the set of real (resp. complex) places of $E$ and 
\begin{align*}
  C_{\infty, +}  = \prod_{\sigma \in \Sigma_E(\mathbf R)} {\rm SO}_2({\mathbf R}) 
                          \prod_{\sigma \in \Sigma_E(\mathbf C)} {\rm SU}_2({\mathbf R}). 
\end{align*}

\begin{dfn}\label{adeliccusp}(\cite[Section 3]{hi94})
{
Let ${\mathcal K}$ be an open compact subgroup of ${\rm GL}_2(\hat{\mathcal O}_E)$. 
Let $J$ be a subset of the set of real places $\Sigma_E({\mathbf R})$ of $E$.
A cusp form on ${\rm GL}_2(E_{\mathbf A})$ of weight $k$, of type $J$ and of level ${\mathcal K}$ 
   is a $C^\infty$-function $f:{\rm GL}_2(E_{\mathbf A}) \to \otimes_{\sigma\in \Sigma_E({\mathbf C})} L(2n_\sigma +2; {\mathbf C})$ which satisfies the following conditions:
\begin{enumerate}
   \item For any  $\sigma\in I_E$,  
             $\displaystyle
                   D_\sigma f = \left( \frac{n^2_\sigma}{2} + n_\sigma \right) f,
             $ where $D_\sigma$ is the Casimir element. 
             (See \cite[Section 2.4]{hi94} for the definition of $D_\sigma$. 
               We use the same notation with \cite{hi94}.)
   \item\label{adeliccusp(ii)} For any $\gamma\in {\rm GL}_2(E), z_\infty\in E^\times_\infty, g\in {\rm GL}_2(E_{\mathbf A})$ and $u\in {\mathcal K}$, 
             we have $  f(\gamma z_\infty g u) = z^{-\kappa}_\infty f(g)$.
     \item\label{adeliccusp(iii)}  For each element $u=\left( \left(\begin{pmatrix} \cos\theta_\sigma & - \sin\theta_\sigma \\ 
                                                                                        \sin\theta_\sigma & \cos\theta_\sigma  \end{pmatrix}\right)_{\sigma\in \Sigma_E({\mathbf R})},
                               (u_\sigma)_{\sigma\in \Sigma_E({\mathbf C})}  \right) \in C_{\infty,+}$,
              we have
             \begin{align*}
                  f\left(gu: \otimes_{\sigma\in \Sigma_E({\mathbf C})} \begin{pmatrix}  S_\sigma \\ T_\sigma \end{pmatrix}  \right) 
               = e^{\sqrt{-1} (-\sum_{\sigma\in J} \theta_\sigma k_\sigma + \sum_{\sigma\in\Sigma_E({\mathbf R})\backslash J} \theta_\sigma k_\sigma  )  }
                        f\left(g: \otimes_{\sigma\in \Sigma_E({\mathbf C})}  u_\sigma\begin{pmatrix}  S_\sigma \\ T_\sigma \end{pmatrix}  \right),
             \end{align*}
             where $(S_\sigma, T_\sigma)$ is  the indeterminate of $L(2n_\sigma+2;{\mathbf C})$.             
      \item For any $g\in {\rm GL}_2(E_{\mathbf A})$, we have $\displaystyle \int_{ {\rm N}_2(E)\backslash {\rm N}_2(E_{\mathbf A}) } 
                         f\left(  u g \right) {\rm d}u =0$,  
                         where ${\rm N}_2$ is the group of upper unipotent matrices in ${\rm GL}_2$
                                and   ${\rm d}u$ is a Haar measure on ${\rm N}_2(E)\backslash {\rm N}_2(E_{\mathbf A})$.
\end{enumerate}
 Let us denote by $S_{\kappa,J}({\mathcal K})$ the space of cusp forms on ${\rm GL}_2(E_{\mathbf A})$ of weight $k$, of type $J$ and of level ${\mathcal K}$. 
 If $E$ does not have a real place, then we drop the subscript $J$ from the notation $S_{\kappa,J}({\mathcal K})$.
 
 Let $\mathfrak N$ be an ideal of ${\mathcal O}_E$ 
 and $\chi:E^\times \backslash E^\times_{\mathbf A} \to {\mathbf C}^\times$ a Hecke character of the conductor dividing ${\mathfrak N}$ 
 and of the infinity type $\chi(z_\infty)= z^{-\kappa}_\infty$.
 We define $S_{\kappa,J}({\mathcal K}_0({\mathfrak N}), \chi)$ to be the subspace of $S_{\kappa,J}( {\mathcal K}_1({\mathfrak N}) )$
 which consists from $f\in S_{\kappa,J}( {\mathcal K}_1({\mathfrak N}) )$ satisfying the following condition:
    \begin{itemize}
        \item for any $z\in E^\times_{\mathbf A}$ and $g\in {\rm GL}_2(E_{\mathbf A})$, $f(zg)= \chi(z)f(g)$. 
    \end{itemize}
              The Hecke character $\chi$ is called the central character of $f$ and we denote the central character of $f$ by $\omega_f$. 
   }
\end{dfn}   

We denote by $\pi_f$ the cuspidal automorphic representation associated with $f$ with the central character $\omega_f$.
We also denote by $\pi$ the unitarization of $\pi_f$, 
that is $\pi$ is defined to be $\pi=\pi_f\otimes |\cdot|^{\frac{[\kappa]}{2}}_{ E_{\mathbf A} }$.
We denote the central character of $\pi$ by $\omega_\pi$.

In this paper, we are interested in the case that the base field $E$ of $\pi$ is a CM field.    
Hence here, we recall some basic properties on $\pi$ assuming $E$ to be a CM field. 
Firstly, we remark that, for each $\sigma\in \Sigma_{E,\infty}$, the Langlands parameter of $\pi_\sigma$ is given by 
\begin{align*}
  \phi[\pi_\sigma]: W_{ E_\sigma } = {\mathbf C}^\times \to {\rm GL}_2({\mathbf C}) : 
        z \mapsto {\rm diag}( z^{\frac{ n_\sigma+1 }{2}} \overline{z}^{-\frac{n_\sigma+1}{2}} , z^{-\frac{ n_\sigma+1 }{2}} \overline{z}^{\frac{ n_\sigma+1 }{2}}  ),   
\end{align*} 
where $\overline{\ast}$ is the complex conjugate of $\ast$.
We recall a description of the Whittaker model ${\mathcal W}(\pi, \psi_E)\cong \otimes^\prime_{w\in \Sigma_E} {\mathcal W}(\pi_w, \psi_{E,w})$ of $\pi$.
The Whittaker function $W_f: {\rm GL}_2(E_{\mathbf A}) \to L(2n+2; {\mathbf C})$ associated with a newform $f\in S_{\kappa}({\mathcal K}_1({\mathfrak N}))$ is defined to be 
\begin{align*}
    W_f(g) = \int_{ E \backslash E_{\mathbf A} } f \left(\begin{pmatrix} 1 & x \\  0 & 1 \end{pmatrix} g \right) \psi_E(-x) {\rm d}x,   
\end{align*}
where ${\rm d}x$ is the self-dual Haar measure with respect to $\psi_E$. 
By the multiplicity one theorem on the Whittaker model of $\pi_w$, 
  we can normalize $f$ so that 
  the unitarization $W_\pi=|\cdot|^{\frac{[\kappa]}{2}}_{E_{\mathbf A}} W_f$ of $W_f$ satisfies the following identity:
\begin{align}\label{eq:WhittDef}
\begin{aligned}
    &    W_\pi  =  \prod_{w\in \Sigma_E} W_{\pi,w}, \quad  (W_{\pi, w} \in {\mathcal W}(\pi_w, \psi_{E,w})),  \\
    & \int_{E^\times_w}  W_{\pi,w} \left(\begin{pmatrix} t & 0 \\ 0 & 1 \end{pmatrix} \right)  | t |^{s-\frac{1}{2}}_{E_w} {\rm d}^\times t 
        = q^{(s-\frac{1}{2}) {\rm ord}_w(\delta_{E,w}) }_w L(s,\pi_w),  \quad (w\in \Sigma_{E}, w<\infty),   \\   
&     W_{\pi, \sigma} \left(\begin{pmatrix} t & 0 \\ 0 & 1  \end{pmatrix} \right) 
         = 2^3 \times \sum^{n+1}_{j=-n-1} \sqrt{-1}^j   t^{n+2} K_j(4\pi t) 
             (  S^{n+1-j} T^{n+1+j})^\vee,  \quad (0<t\in {\mathbf R}, \sigma \mid \infty).
\end{aligned}
\end{align} 
Here 
${\rm ord}_w(x)$ is defined to be $m$ if $x=\varpi^m_w u\in E_w$ for $u\in {\mathcal O}^\times_{E, w}$, 
$\delta_E\in E_{{\mathbf A}, {\rm fin}}$ is a generator of the different of $E/{\mathbf Q}$, 
  $(  S^{n+1-j} T^{n+1+j})^\vee$ is the dual element of $S^{n+1-j} T^{n+1+j} $ with respect to the pairing $[\cdot, \cdot]_{2n+2}$ in (\ref{eq:pairing})
and $K_s(z)$ is the modified Bessel function of the second kind which is defined to be 
\begin{align*}
   K_s(z) =  \frac{1}{2}  \int^\infty_0  \exp \left(  -\frac{z}{2} \left( t + \frac{1}{t} \right)    \right)  t^{s-1} {\rm d}t, \quad ({\rm Re}(z) > 0, s\in {\mathbf C}). 
\end{align*}

Hereafter we always fix an odd prime such that $p>n$, that is, $p>n_\sigma$ for each $\sigma \in \Sigma_{E,\infty}$.
Throughout this paper,  assume the following condition: 
\begin{itemize}
   \item  $\omega_{\pi, p}$ is unramified.    
\end{itemize}

\subsection{Asai motives}\label{sec:AsaiMot}

Let $F$ be a totally real field and $E/F$ a CM extension.
Consider an irreducible unitary cohomological cuspidal automorphic representation $\pi$ of ${\rm GL}_2(E_{\mathbf A})$, which is introduced in Section \ref{sec:defauto}.
The existence of $p$-adic $L$-functions is conjectured by Coates and Perrin-Riou in terms of motives, which we will recall in Section \ref{sec:CPconj} in the Asai motives case. 
In this subsection, we briefly describe some notion of conjectural motives associated with $\pi$. 

The conjectural motive ${\mathcal M}[\pi]$ over $E$ attached to $\pi$  (\cite[Conjecture 4.5]{cl90}, \cite[Section 8]{hi94}), 
which has a pure weight $n+2m+1$, 
 satisfies 
\begin{itemize}
\item $\displaystyle  
          H_{\rm B}({\mathcal M}[\pi]_\sigma) \otimes {\mathbf C} 
      = H^{ n_\sigma+m_\sigma+1, m_{\sigma} }({\mathcal M}[\pi])
             \oplus H^{m_\sigma, n_\sigma+m_\sigma+1 }({\mathcal M}[\pi]),  \quad (\sigma \in I_E)$; 
\item $\displaystyle 
            L(s, {\mathcal M}[\pi] ) 
       = L(s+\frac{[\kappa]+1}{2}, \pi)$,      
\end{itemize}
where $H_{\rm B}(\ast)$ is the Betti realization of $(\ast)$. 
Denote by ${\rm As}^+({\mathcal M}[\pi])$ the Asai motive over $F$ attached to $\pi$, 
that is, ${\rm As}^+({\mathcal M}[\pi])$ is the descent of the motive ${\mathcal M}[\pi]\otimes{\mathcal M}[\pi]^{c}$ over $E$  to $F$ 
with the descent datum $v\otimes w \mapsto w\otimes v$ (\cite[Section 4]{gh99}).
We put 
\begin{align*}
   {\rm As}^+_{\mathcal M}(\pi) = {\rm As}^+({\mathcal M}[\pi])(n+2m+2), 
\end{align*}
which is a pure motive of the weight $-2$.
Let $h(i,j)={\rm dim}_{\mathbf C} H^{i,j}({\rm As}^+( {\mathcal M}[\pi]) )$.
Note that ${\rm As}^+_{\mathcal M}(\pi)$ has the following properties: 
\begin{itemize}
\item $\displaystyle  
  H_{\rm B}( {\rm As}^+_{\mathcal M}(\pi)_\sigma )\otimes {\mathbf C}
  = H^{ -n_\sigma-2, n_\sigma} (  {\rm As}^+_{\mathcal M}(\pi)_\sigma )
     \oplus H^{ -1, -1} ( {\rm As}^+_{\mathcal M}(\pi)_\sigma )  
     \oplus H^{n_\sigma, -n_\sigma-2} ( {\rm As}^+_{\mathcal M}(\pi)_\sigma )$;  
\item $\displaystyle  
   h( -n_\sigma-2, n_\sigma) = h(n_\sigma, -n_\sigma-2) = 1,  
     \quad h(-1, -1)=2$;   
\item $\displaystyle  
     \text{ the complex conjugate acts on   $H^{ -1, -1} (  {\rm As}^+_{\mathcal M}(\pi)  )$ as $+1$} $.  
\end{itemize} 
We see that ${\rm As}^+_{\mathcal M}(\pi)$ is critical at $s=j$ if and only if  
\begin{align}\label{eq:critrange}
  \left\{ j \ \middle| \  \text{odd}, -n_{\rm min}-1 \leq j \leq -1  \right\}\cup \left\{ j \ \middle| \  \text{even},  0 \leq j \leq n_{\rm min}  \right\}, 
  \quad (n_{\rm min}:= \min \left\{ n_\sigma \middle|  \sigma \in I_E \right\}).
\end{align} 
In particular, ${\rm As}^+_{\mathcal M}(\pi)$ is critical at $s=0$.   
Consider $n$ as an element in ${\mathbf Z}[I_F]$, since $n_\sigma = n_{c\sigma}$ for each $\sigma \in I_E$. 
Let $\alpha = \sum_{\sigma \in I_F}  \alpha_\sigma \sigma \in {\mathbf Z}[I_F]$ satisfying the following conditions:
\begin{itemize}
\item[(Alp1)] for each $\sigma\in I_F$, $0\leq \alpha_\sigma \leq n_\sigma$;
\item[(Alp2)]  for each $\sigma, \tau \in I_F$, we have $n_\sigma-\alpha_\sigma = n_\tau - \alpha_\tau$.
\end{itemize}
The following proposition is easily deduced from the above description and \cite[Section 5.3]{de79}:  

\begin{prop}
{\itshape    
Let $\varphi: F^\times \backslash F^\times_{\mathbf A} \to {\mathbf C}^\times$ be a Hecke character of finite-order   
      and $\alpha \in {\mathbf Z}[I_F]$ as above. 
Put $\phi=|\cdot|^{n-\alpha}_{ F_{\mathbf A} }\varphi$. 
Then ${\rm As}^+_{\mathcal M}(\pi)(\phi)$ is critical at $s=0$ if and only if 
\begin{itemize}
   \item $n-\alpha$ is even and $\varphi_\sigma(-1)=1$ for each $\sigma \in \Sigma_{F,\infty}$. In this case, we find that 
            \begin{align*}
               L_\infty(s, {\rm As}^+_{\mathcal M}(\pi)(\phi)) = \Gamma_{\mathbf C}(s+2n-\alpha+2t) \Gamma_{\mathbf R}(s+n-\alpha+2t)^2, \quad   
               \epsilon_\infty({\rm As}^+_{\mathcal M}(\pi)(\phi), \psi_{F, \infty} ) = \sqrt{-1}^{ (2n+3t) + 2t }.   
            \end{align*}    
   \item $n-\alpha$ is odd and $\varphi_\sigma(-1)= -1$ for each $\sigma \in \Sigma_{F,\infty}$.  In this case, we find that 
            \begin{align*}
               L_\infty(s, {\rm As}^+_{\mathcal M}(\pi)(\phi)) = \Gamma_{\mathbf C}(s+2n-\alpha+2t) \Gamma_{\mathbf R}(s+n-\alpha+t)^2, \quad 
               \epsilon_\infty({\rm As}^+_{\mathcal M}(\pi)(\phi), \psi_{F, \infty} ) = \sqrt{-1}^{ (2n+3t) },     
            \end{align*}   
\end{itemize}
where $\Gamma_{\mathbf R}(s) =\pi^{-\frac{s}{2}} \Gamma(\frac{s}{2})$ and $\Gamma_{\mathbf C}(s) = 2 (2\pi)^{-s} \Gamma(s)$.

}
\end{prop}

We also summarize the Deligne's period $c^+({\rm As}^+_{\mathcal M} (\pi))$ of the Asai motive ${\rm As}^+_{\mathcal M} (\pi)$ briefly
    according to \cite{gh99} and \cite{gh99b}. 
Let ${\mathcal N}[\pi]={\rm Res}_{E/{\mathbf Q}}( {\mathcal M}[\pi])$ be the Weil restriction of ${\mathcal M}[\pi]$. 
Define $\Omega_{\pi,p}$ to be the Hida's canonical period attached to $\pi$, 
which we will precisely recall the definition around the identity(\ref{eq:CanPer}) in Section \ref{sec:ESHmap}. 
Note that the period attached to ${\mathcal N}[\pi]$ is originally defined in \cite[Section 8]{hi94}. 
Then we find that (see \cite[(8.7b)]{hi94} for instance):
  \begin{align*}
    c^+({\mathcal N}[\pi]) = (2\pi\sqrt{-1})^{-2m} \Omega_{\pi,p}, \quad \delta({\mathcal N}[\pi]) = (2\pi\sqrt{-1})^{-2(n+t+2m)},   
  \end{align*}  
  where we use notations $c^+(\ast)$ and $\delta(\ast)$ as the same ones in \cite{de79}. 
Furthermore 
 \cite[page 613, Proposition 3, page 637, Remark 3]{gh99} shows that 
   \begin{align*}
     c^+( {\rm Res}_{F/{\mathbf Q}}( {\rm As}^+({\mathcal M}[\pi]) ) )
        =& c^+({\mathcal N}) \delta( {\mathcal N} )   \\
       =&  (2\pi\sqrt{-1})^{-2m} \times (2\pi\sqrt{-1})^{-2n-2t-4m} \Omega_{\pi,p}
        = (2\pi\sqrt{-1})^{-3(n+2t+2m)} \times (2\pi\sqrt{-1})^{n+4t} \Omega_{\pi,p}.     
   \end{align*}
Hence by taking the Tate twist, we find that 
   \begin{align*}
     c^+( {\rm Res}_{F/{\mathbf Q}}( {\rm As}^+_{\mathcal M}( \pi )   ) 
    = c^+( {\rm Res}_{F/{\mathbf Q}}( {\rm As}^+({\mathcal M}[\pi]) ) ) 
     \times (2\pi\sqrt{-1})^{-3(n+2m+2t)} 
    =   (2\pi\sqrt{-1})^{n+4t} \Omega_{\pi,p}.     
   \end{align*}
Hereafter we write $c^+( {\rm Res}_{F/{\mathbf Q}}( {\rm As}^+_{\mathcal M}( \pi )   )$ 
  as $c^+(  {\rm As}^+_{\mathcal M}( \pi )   ) = (2\pi\sqrt{-1})^{n+4t} \Omega_{\pi,p}$ for the sake of the simplicity.

\begin{rem}
The motive ${\rm As}^-({\mathcal M}[\pi])$ is also defined to be 
the descent of the motive ${\mathcal M}[\pi]\otimes {\mathcal M}[\pi]^c$ with the descent datum $v\otimes w \mapsto -w\otimes v$. 
Let $\tau_{E/F}$ be the quadratic character associated with $E/F$. 
Then ${\rm As}^-({\mathcal M}[\pi])$ is isomorphic to the twist of ${\rm As}^+({\mathcal M}[\pi])$ by $\tau_{E/F}$ (see \cite[Section 4.1]{gh99}). 
The $p$-adic $L$-functions for ${\rm As}^-({\mathcal M}[\pi])$ should be also constructed in the similar way with the case of ${\rm As}^+({\mathcal M}[\pi])$.  
However in the present paper, we concentrate on the construction for ${\rm As}^+({\mathcal M}[\pi])$ to avoid the complications. 
\end{rem}

\subsection{Asai transfer}

We summarize and fix notation on Asai $L$-functions. 
Let $v\in \Sigma_F$ and $W^\prime_{F_v}$ the Weil-Deligne group of $F_v$. 
We identify ${\mathcal G} = ( {\rm GL}_2({\mathbf C}) \times {\rm GL}_2({\mathbf C}) )  \rtimes {\rm Gal}(\overline{F}_v/ F_v) $ 
    with  the Langlands dual group of ${\rm Res}_{E_v/F_v}({\rm GL}_{2/ E_v})$. 
For an irreducible cohomological cuspidal automorphic representation $\pi=\otimes^\prime_{w\in \Sigma_{E}} \pi_w$ 
    of ${\rm GL}_2(E_{\mathbf A})$, let $\phi_{\pi_v}: W^\prime_{F_v} \to {\mathcal G}$ be the Langlands parameter corresponding to $\pi_v$.  
Define $r^\pm:  {\mathcal G}  \to   {\rm GL}({\mathbf C}^2 \otimes {\mathbf C}^2)   \rtimes {\rm Gal}(\overline{F}_v/ F_v)$ to be 
\begin{align*}
   r^\pm( g_1, g_2, \sigma ) 
   = \begin{cases}  (g_1\otimes g_2, \sigma),  &    (\sigma|_E\text{ is trivial}), \\ 
                              (\pm g_2\otimes g_1, \sigma),  &    (\sigma|_E\text{ is not trivial}).  \end{cases}
\end{align*}
Let $\phi_{\chi_v}: W^\prime_F \to {\rm GL}_1(\mathbf C)\rtimes {\rm Gal}(\overline{F}_v/ F_v)$ 
           be the Langlands parameter of a continuous character $\chi_v:F^\times_v \to {\mathbf C}^\times$.
Define ${\rm As}^\pm(\pi_v) \otimes \chi_v$ to be the admissible representation of ${\rm GL}_4(F_v)$ corresponding to the Langlands parameter 
  $r^\pm \circ \phi_{\pi_v} \otimes \phi_{\chi_v}$ 
 via the local Langlands correspondence 
and 
denote by $L(s, {\rm As}^\pm(\pi_v) \otimes \chi_v )$ (resp. $\epsilon(s, {\rm As}^\pm(\pi_v) \otimes \chi_v, \psi_{F,v})$) 
    the $L$-factor $L(s,  r^\pm \circ \phi_{\pi_v}\otimes \phi_{\chi_v} )$ 
    (resp. $\epsilon$-factor $\epsilon(s,  r^\pm \circ \phi_{\pi_v} \otimes \phi_{\chi_v}, \psi_{F,v})$ with respect to $\psi_{F,v}$) (\cite[(4.1.6)]{ta79}).  
Define the $\gamma$-factor of ${\rm As}^\pm(\pi_v)\otimes \chi_v$ to be 
\begin{align}\label{eq:gfactor}
  \gamma(s, {\rm As}^\pm(\pi_v)\otimes \chi_v, \psi_{F, v}) 
  = \epsilon(s, {\rm As}^\pm(\pi_v) \otimes \chi_v, \psi_{F,v}) 
       \frac{  L(1-s, {\rm As}^\pm(\pi_v)^\vee \otimes \chi^{-1}_v )  
                 }{   L(s, {\rm As}^\pm(\pi_v) \otimes \chi_v )   }, 
\end{align}
where ${\rm As}^\pm(\pi_v)^\vee$ is the contragradient of ${\rm As}^\pm(\pi_v)$.

For $\pi$, we have the Asai lift ${\rm As}^+(\pi)$ of $\pi$, which is an isobaric automorphic representation of ${\rm GL}_4(F_{\mathbf A})$ due to \cite[Theorem 6.3]{kr03}.
Recall that the Asai $L$-functions $L(s, {\rm As}^+(\pi))=\prod_{v\in \Sigma_F}  L(s, {\rm As}^+(\pi_v) )$ is meromorphically continued to the whole ${\mathbf C}$-plane
 with possible pole at $s=0$ or $1$ 
satisfying
\begin{align*}
  L(s, {\rm As}^+_{\mathcal M}(\pi) )   =  L(s+1,  {\rm As}^+(\pi)). 
\end{align*}
Assuming the contragradient $\pi^\vee$ of $\pi$ is not isomorphic to the complex conjugate $\pi^c$ of $\pi$, 
it is known that $L(s, {\rm As}^+(\pi))$ is entire (see \cite[Theorem 4.3]{gs15}).
See also \cite[Section 3.2]{hn18} for some explicit formulas for local $L$-factors in certain cases.

\section{Coates-Perrin-Riou's conjecture}\label{sec:CPconj}

A conjecture on the existence of $p$-adic $L$-functions for motives is given in \cite{cpr89} and \cite{co89}.
In this section, we recall this conjecture in the Asai motives case. 
Note that we still do not have a motive attached to an automorphic representation $\pi$ of ${\rm GL}_2$ over CM fields.  
However, the conjecture on the existence of $p$-adic $L$-functions can be written down without assuming the existence of the motive, 
since several constants on the conjectural motive are already conjectured in \cite{cl90} in terms of the automorphic representation theory. 
Our exposition is close to the language in \cite{co89}.  
Although it is assumed that the motives are defined over the rational number field in \cite{co89}, 
   we give a description of the conjectures in \cite{co89} in the case that the motives are defined over totally real fields $F$
   to clarify the relation between conjectures in \cite{co89} and the main theorem of the present paper. 

Let $\phi:F^\times \backslash F^\times_{\mathbf A} \to {\mathbf C}^\times$  
   with the infinity type $\phi_\infty(x) = x^{n-\alpha}$ for some $0\leq \alpha \leq n$.
    Write $\phi=|\cdot|^{n-\alpha}_{\mathbf A} \varphi$ for some finite-order character $\varphi$. 
    Hence  $\widehat{\phi}({\mathscr L}_p({\rm As}^+_{\mathcal M}(\pi) ))$ should interpolate the values
    $ L(0, {\rm As}^+_{\mathcal M}(\pi)(\phi))= L(n-\alpha+1, {\rm As}^+(\pi)(\varphi))$.  
    
Let $p$ be a prime number. 
Throughout this paper, we assume that $p$ is odd and that $p$ is unramified in $E/{\mathbf Q}$.     
We give a characterization of the modified Euler factor at $p$ and $\infty$ and periods according to \cite{co89} as follows:    

    \noindent
    {\bf (nearly $p$-ordinarity of $\pi$)}    
    Let $w\in \Sigma_E, w\mid p$ and ${\mathcal K}$ an open compact subgroup of ${\rm GL}_2(\widehat{\mathcal O}_E)$ such that ${\mathcal K}_w = {\mathcal K}_0(\varpi_w)$. 
    Then define the $U(w)$-operator $U(w)\in {\rm End}( S_{\kappa}({\mathcal K}) )$ to be 
    \begin{align*}
        U(w) f(g) = \sum_\gamma f\left( g \gamma\right), 
        \quad \left( f\in S_{\kappa}({\mathcal K}), 
                          {\mathcal K}_0(\varpi_w)\begin{pmatrix} \varpi_w & 0 \\ 0 & 1 \end{pmatrix} {\mathcal K}_0(\varpi_w) = \cup_\gamma  \gamma{\mathcal K}_0(\varpi_w) (\text{disjoint}) \right). 
    \end{align*}
    Moreover, we introduce another Hecke operator $U_0(w)$ according to \cite[Section 4]{hi94}.
    Recall that ${\mathcal O}_\pi$ is the subring of ${\mathbf C}_p$ as introduced in Section \ref{sec:intro}.  
    Define $\{ \varpi^m_w\} \in {\mathcal O}_\pi$ to be 
    \begin{align*}
         \{\varpi^m_w\} = \prod_{\sigma_p}  \sigma_p(\varpi_w)^{m_\sigma},  
    \end{align*}
    where $\sigma_p$ runs over a set $\left\{  \sigma_p: E_w \to {\mathbf C}_p\middle| \sigma_p: \text{continuous} \right\}$ and we put $\sigma={\mathbf i}^{-1}_p\circ\sigma_p|_E$. 
    Define 
    \begin{align*}
      U_0(w) = \{ \varpi^m_w\}^{-1} U(w).    
    \end{align*}
    Since we suppose that $p$ is unramified in $E/{\mathbf Q}$, we may assume that $\varpi_w=p$ for each $w\mid p$
    and hence we have $p^m= \prod_{w\mid p} \{\varpi^m_w\}$.

    Let $f$ be an element in $S_{\kappa}({\mathcal K})$.   
    We call $f$ nearly $p$-ordinary, if $f$ is an eigenform for $U_0(w)$-operator with $p$-adic unit eigenvalue $\lambda_{w,0}$ for each $w\mid p$.  
    We say $\pi$ to be nearly $p$-ordinary, if there exists a nearly $p$-ordinary $p$-stabilized newform $f\in \pi\otimes|\cdot|^{-\frac{[\kappa]}{2}}_{E_{\mathbf A} }$.
    
    Hereafter throughout this paper, we always assume that 
    $\pi$ is nearly $p$-ordinary and that $f\in \pi\otimes|\cdot|^{-\frac{[\kappa]}{2}}_{E_{\mathbf A} }$ is a nearly $p$-ordinary $p$-stabilized newform.
    Write the eigenvalue $\lambda_{w,0}  = \{ \varpi^m_w\}^{-1} \beta_{\pi_w} q^{\frac{[\kappa]+1}{2}}_w$ 
    of $U_0(w)$ with respect to $f$ for some $\beta_{\pi_w}\in {\mathbf C}$.   
    Then $\beta_{\pi_w}$ is described as follows (see \cite[Corollary 2.2]{hi89}): 
    \begin{itemize}
      \item[(\rm Sat1)] 
              If $\pi_w=\pi(\mu, \nu)$ is a principal series representation for some characters $\mu, \nu:E^\times_w \to {\mathbf C}^\times$. 
               Since $\pi$ is  nearly $p$-ordinarity, we may assume that $\nu$ is unramified and that $\beta_{\pi_w}  = \nu(\varpi_w)$. 
               Let $\alpha_{\pi_w}=\mu(\varpi_w)$. 
      \item[(\rm Sat2)]
                If $\pi_w= {\rm Sp}(\eta) 
                                \subset \pi(\eta|\cdot|^\frac{1}{2}_{E, w}, \eta|\cdot|^{-\frac{1}{2} }_{E, w})$ is a special representation. 
               Then the nearly $p$-ordinarity implies that $\eta$ is an unramified quadratic character. 
               In this case we have $\beta_{\pi_w} = \eta (\varpi_w) q^{-\frac{1}{2}}_w$.
               Let $\alpha_{\pi_w}=\eta (\varpi_w)q^{\frac{1}{2}}_w$. 
    \end{itemize} 
     We also note that 
    \begin{align}\label{eq:lam0}
      \lambda_{p,0}:= \prod_{w\mid p} \lambda_{w,0} =  p^{-m} \prod_{w\mid p} \beta_{\pi_w} q^\frac{[\kappa]+1}{2}_w.  
    \end{align}
     We put $\lambda_w =  \beta_{\pi_w} q^\frac{[\kappa]+1}{2}_w$ for each $w\mid p$, which is an eigenvalue of $U(w)$. 
     Let $\lambda_p=\prod_{w\mid p} \lambda_w$,
     which is an eigenvalue of $U(p) :=\prod_{w\mid p}U(w)$-operator.
    We also recall that the Whittaker function $W_{\pi, w}$ at $w\mid p$ associated with a $p$-stabilized newform is characterized by the following identity:
    \begin{align}\label{eq:WhittStab}
            W_{\pi, w}\left(  \begin{pmatrix} a  & 0 \\ 0 & 1  \end{pmatrix}     \right)
        =\chi_{\beta_{\pi_w}}(a)  |a|^\frac{1}{2}_{E, w} {\mathbb I}_{\mathcal O_{E,w}}(a), 
    \end{align}
    where $\chi_{\beta_{\pi_w}}: E^\times_w \to {\mathbf C}^\times$ is the unramified character such that $\chi_{\beta_{\pi_w}} (\varpi_w)  = \beta_{\pi_w}$ 
               and ${\mathbb I}_X(\ast)$ is the characteristic function of $X$.

    \noindent
    {\bf (Modifeid Euler factor ${\mathcal E}_v( {\rm As}^+_{\mathcal M}(\pi)( \phi ) )$ at $v\mid p$)}
    Let $v\in \Sigma_F, v\mid p$, which is unramifed in $E/F$. 
   Let $w, w_c\in \Sigma_E$ so that $v = ww_c$ if $v$ is split in $E/F$.
   If $v$ is inert in $E/F$, write $w=v\in \Sigma_E$. 
   Suppose that $\pi$ is nearly $p$-ordinary. For each $v\mid p$, we write $\pi_v=(\pi_w, \pi_{w_c})$ (resp. $\pi_w$) if $v$ is split (resp. inert).  
   Let $\alpha_{\pi_w}, \beta_{\pi_w} (\alpha_{\pi_{w_c}}, \beta_{\pi_{w_c}} \text{if $v$ is split})$ be the constants which is defined in (Sat1) and (Sat2).
   Write $\alpha_{\pi_v}=\alpha_{\pi_w} \alpha_{{\pi_{w_c}}}, \beta_v = \beta_{\pi_w}\beta_{\pi_{w_c}}$ if $v$ is split,
   and $\alpha_{\pi_v}=\alpha_{\pi_w}, \beta_{\pi_v}=\beta_{\pi_w}$ if $v$ is inert.
   For $\gamma  \in {\mathbf C}^\times$, define an unramified character $\chi_\gamma:F^\times_v \to {\mathbf C}^\times$ so that $\chi_\gamma(\varpi_v) = \gamma$.

Let $\gamma(s, \chi_{\alpha_{\pi_v}} \varphi_v, \psi_{F,v})$ is the $\gamma$-factor of $\chi_{\alpha_{\pi_v}} \varphi_v$ 
      associated with the fixed additive character $\psi_{F,v}$.  
Define the modified Euler factor ${\mathcal E}_v( {\rm As}^+_{\mathcal M}(\pi)( \phi ) )$  to be   
   \begin{align*}
      {\mathcal E}_v( {\rm As}^+_{\mathcal M}(\pi)( \phi ) ) L_v(0, {\rm As}^+_{\mathcal M}(\pi) (\phi)  )
      = \frac{ \gamma(n-\alpha+1, \chi_{\alpha_{\pi_v}} \varphi_v, \psi_{F,v}   ) }{\gamma(n-\alpha+1, {\rm As}^+(\pi_v) \otimes \varphi_v  , \psi_{F,v}  )},  
   \end{align*} 
  where $\gamma(s, {\rm As}^+(\pi_v) \otimes \varphi_v  , \psi_{F,v})$ 
                 is the $\gamma$-factor in (\ref{eq:gfactor}).
Suppose that $\pi_v$ is unramified and the base field $F$ is the rational number field.  
 Then this definition is compatible with the definition of the modified Euler factor in \cite[Section 2, (18)]{co89}. 
   In particular, by using the notation ${\mathcal L}^{(\sqrt{-1})}_v(  {\rm As}^+_{\mathcal M}(\pi)( \phi ) )$ in \cite[Section 2, (18)]{co89}, 
   we find that 
   \begin{align*}
      {\mathcal E}_v( {\rm As}^+_{\mathcal M}(\pi)( \phi ) ) L_v(0, {\rm As}^+_{\mathcal M}(\pi) (\phi)  )
      = {\mathcal L}^{(\sqrt{-1})}_v(  {\rm As}^+_{\mathcal M}(\pi)( \phi ) ). 
   \end{align*}

\begin{rem}
We write down the explicit form of ${\mathcal E}_v( {\rm As}^+_{\mathcal M}(\pi)( \phi ) )$ as follows. 
This explicit form is not used in this paper, however it might be useful for the readers and the future study. 
We use the same notations as above. 
Denote by $c(\varphi_v)$ the exponent of the conductor of $\varphi_v$. 
We also recall that the local Gauss sum is defined to be 
\begin{align*}
     \tau( \varphi_v, \psi_{F, v})
      = \sum_{u \in ({\mathcal O}_{F,v} / \varpi^{c(\varphi_v)}_v{\mathcal O}_{F,v} )^\times}  
              \varphi_v(u \varpi^{-c(\varphi_v )}_v ) \psi_{F, v}( u \varpi^{-c(\varphi_v)}_v ).
\end{align*}
Note that the $\epsilon$-factor $\epsilon(s, \varphi_v, \psi_{F,v})$ satisfies  
\begin{align*}
\epsilon(s, \varphi_v, \psi_{F,v}) 
= q^{-s c(\varphi_v)}_v \tau( \varphi^{-1}_v, \psi_{F, v}), 
\end{align*}
since $v\mid p$ is unramified in $F/{\mathbf Q}$ (\cite[(3.2.6.2)]{ta79}).
We also note that $\pi_v$ is an irreducible subquotient of $\pi(\mu_v, \nu_v)$ for some unramified characters $\mu_v, \nu_v: E^\times_v \to {\mathbf C}^\times$ 
   since we assume that $\pi$ is nearly $p$-ordinary and that $\omega_{\pi, p}$ is unramified. 
Hence we have the following formulas: 
\begin{enumerate}
 \item Assume that $v\in \Sigma_{F,p}$ is split in $E/F$. Write $v=ww_c$ for $w, w_c\in \Sigma_{E}$. 
          Then we have 
          \begin{align*}
       & {\mathcal E}_v( {\rm As}^+_{\mathcal M}(\pi)( \phi ) )   \\  
    =& L_v(n-\alpha+1, {\rm As}^+(\pi)(\varphi) )^{-1}  \\ 
      & \times   \begin{cases} 
                                \displaystyle 
                                      \frac{  ( 1 - \alpha^{-1}_{w_c} \beta^{-1}_w \varphi^{-1}_v( \varpi_v ) q^{n-\alpha }_v)  
                                                 ( 1 - \alpha^{-1}_{w} \beta^{-1}_{w_c}  \varphi^{-1}_v( \varpi_v ) q^{ n-\alpha }_v) 
                                                 ( 1 - \beta^{-1}_{w} \beta^{-1}_{w_c} \varphi^{-1}_v( \varpi_v ) q^{ n-\alpha }_v) 
                                             }{  ( 1 - \alpha_{w_c} \beta_w \varphi_v( \varpi_v ) q^{-(n-\alpha+1) }_v)  
                                                 ( 1 - \alpha_{w} \beta_{w_c}  \varphi_v( \varpi_v ) q^{ -(n-\alpha+1) }_v) 
                                                 ( 1 - \beta_{w} \beta_{w_c} \varphi_v( \varpi_v ) q^{ -(n-\alpha+1) }_v)  
                                                    },             
                                             &     (c(\varphi_v)=0),   \\
                              \displaystyle 
                               \frac{  q^{ 3 c(\varphi_v) (n-\alpha+1)  }_v     }{  ( \alpha_w\beta_w \alpha_{w_c} \beta_{w_c} \cdot  \beta_w \beta_{w_c})^{c(\varphi_v)}  } 
                                         \tau( \varphi^{-1}_v, \psi_{F, v} )^{-3},    
                                          &     (c(\varphi_v)>0).
                             \end{cases} 
\end{align*}
 \item Assume that $v\in \Sigma_{F,p}$ is inert in $E/F$. Consider $v$ as an element $w$ in $\Sigma_{E,p}$. 
                     Then we have 
          \begin{align*}
    {\mathcal E}_v( {\rm As}^+_{\mathcal M}(\pi)( \phi ) )
    =& L_v(n-\alpha+1, {\rm As}^+(\pi)(\varphi) )^{-1}  \\ 
      & \times   \begin{cases} 
                                 \displaystyle 
                                      \frac{  ( 1 - \alpha^{-1}_{w} \beta^{-1}_w \varphi^{-2}_v( \varpi_v ) q^{2(n-\alpha) }_v)  
                                                 ( 1 - \beta^{-1}_{w}  \varphi^{-1}_v( \varpi_v ) q^{ n-\alpha }_v) 
                                             }{  ( 1 - \alpha_{w} \beta_w \varphi_v( \varpi_v ) q^{-2(n-\alpha+1) }_v)  
                                                 ( 1 - \beta_{w}  \varphi_v( \varpi_v ) q^{ -(n-\alpha+1) }_v)  
                                                    },             
                                             &     (c(\varphi_v)=0),   \\
                              \displaystyle 
                               \frac{  q^{ 3 c(\varphi_v) (n-\alpha+1)  }_v     }{  ( - \alpha_w\beta_w  \cdot  \beta_w  )^{c(\varphi_v)}  } 
                                         \tau( \varphi^{-1}_v, \psi_{F, v} )^{-3},    
                                          &     (c(\varphi_v)>0).
                             \end{cases} 
\end{align*}
\end{enumerate}
\end{rem}

    \noindent
    {\bf (Modifeid Euler factor ${\mathcal E}_v( {\rm As}^+_{\mathcal M}(\pi)( \phi ) )$ at $v\mid \infty$)}
    Let $v\in \Sigma_{F, \infty}$. 
     Define the modified Euler factor ${\mathcal E}_v( {\rm As}^+_{\mathcal M}(\pi)( \phi ) )$ to be 
     \begin{align*}
        & {\mathcal E}_v( {\rm As}^+_{\mathcal M}(\pi)( \phi ) ) 
            L_v( 0, {\rm As}^+_{\mathcal M}(\pi)( \phi ) )  \\ 
  =& \begin{cases} 
        \displaystyle  \frac{  \sqrt{-1}^{-(2n_v-\alpha_v+2)} 
                                       }{ (-1) \times    \Gamma_{\mathbf R} (1-(n_v-\alpha_v))^2    } 
                                        \times  \left( \Gamma_{\mathbf C}(2n_v-\alpha_v+2) \Gamma_{\mathbf R}(n_v -\alpha_v +2)^2 \right), 
                            & (n_v - \alpha_v :\text{even}, \varphi_v(-1)=1),   \\
       \displaystyle  \frac{ \sqrt{-1}^{-(2n_v-\alpha_v+2)} 
                                       }{    \Gamma_{\mathbf R} (-(n_v-\alpha_v))^2    }  
                              \times  \left( \Gamma_{\mathbf C}(2n_v-\alpha_v+2) \Gamma_{\mathbf R}(n_v -\alpha_v +1)^2 \right),  
                            & (n_v-\alpha_v:\text{odd}, \varphi_v(-1)=-1)).   
       \end{cases} 
     \end{align*}
     This definition coincides with the definition of modified Euler factor in \cite[Section 1, (4)]{co89}. 
   In particular, by using the notation ${\mathcal L}^{(\sqrt{-1})}_v(  {\rm As}^+_{\mathcal M}(\pi)( \phi ) )$ in \cite[Section 1, (4)]{co89}, 
   we find that 
   \begin{align*}
      {\mathcal E}_v( {\rm As}^+_{\mathcal M}(\pi)( \phi ) ) L_v(0, {\rm As}^+_{\mathcal M}(\pi) (\phi)  )
      = {\mathcal L}^{(\sqrt{-1})}_v(  {\rm As}^+_{\mathcal M}(\pi)( \phi ) ). 
   \end{align*}
     This number is the same (up to a simple factor) with the number $G_\infty(0, f)$ in \cite[page 629, Conjecture 1]{gh99}.
    We will briefly check this compatibility in Section \ref{sec:InfInt}. 
    In \cite[Appendix]{lw}, Loeffler-Williams also check this compatibility in the classical language, 
    so we will check the compatibility with the classical language and the adelic language.

    \noindent
    {\bf (Period $\Omega({\rm As}^+_{\mathcal M}(\pi) )$)}
    Let $c^+({\rm As}^+_{\mathcal M}(\pi))$ be the Deligne's period attached to the Asai motive ${\rm As}^+_{\mathcal M}(\pi)$, 
             which we recalled in Section \ref{sec:AsaiMot}.    
    Coates introduced a constant $\tau ({\rm As}^+_{\mathcal M}(\pi))$ in \cite[(12)]{co89} 
    and a period $\Omega({\rm As}^+_{\mathcal M}(\pi) )$ to describe Deligne's conjecture on critical values (\cite[Conjecture 1.8]{de79})
    in terms of the complete $L$-functions. We write down it explicitly in our setting as follows:         
     \begin{align*}
          \tau ({\rm As}^+_{\mathcal M}(\pi)) =  -n-4t, \quad 
           \Omega({\rm As}^+_{\mathcal M}(\pi)) 
             :=   (2\pi\sqrt{-1})^{ \tau ({\rm As}^+_{\mathcal M}(\pi)) }
                    \times  c^+({\rm As}^+_{\mathcal M}(\pi))  
              = \Omega_{\pi,p}.
     \end{align*}
      Assume that ${\rm As}^+_{\mathcal M} (\pi) (\phi)$ is critical at $s=0$. 
     Then \cite[page 107, Period Conjecture]{co89} predicts the following value 
     \begin{align}\label{eq:GhaAlg}
        {\mathcal E}_p({\rm As}^+_{\mathcal M} (\pi) (\phi))
        {\mathcal E}_\infty({\rm As}^+_{\mathcal M} (\pi) (\phi)) 
        \frac{ L_\infty(0, {\rm As}^+_{\mathcal M} (\pi) (\phi)) L^{(\infty)}(0, {\rm As}^+_{\mathcal M} (\pi)(\phi)) }{\Omega( {\rm As}^+_{\mathcal M}(\pi)  )}  
        \sim_{\overline{\mathbf Q}^\times} \frac{ L^{(\infty)}(0, {\rm As}^+_{\mathcal M} (\pi) (\phi)) }{ \pi^{ 4n-3\alpha+4t } \Omega_{\pi,p}}
     \end{align}
     gives an element in $\overline{\mathbf Q}$, 
     where we write $a\sim_{\overline{\mathbf Q}^\times} b$ for $a, b\in {\mathbf C}$, if $a=bc$ for some $c\in \overline{\mathbf Q}^\times$. 
     The algebraicity of (\ref{eq:GhaAlg}) is proved by Ghate (\cite[page 635, Theorem 1]{gh99}, \cite[page 106, Theorem 1]{gh99b}), 
     if $\phi$ is of the form $|\cdot|^{n-\alpha}_{F_{\mathbf A} }$ for even $n-\alpha$ ($0\leq \alpha \leq n$).

By using the above definitions and constants, the conjecture on the existence of $p$-adic Asai $L$-functions is stated as follows:

\begin{conj}\label{conj:pL}
{\itshape 
There exists an element ${\mathscr L}_p( {\rm As}^+_{\mathcal M} (\pi) ) \in {\mathcal O}_\pi[[ {\rm Gal}(F (\mu_{p^\infty} ) /  F ) ]]$ such that,  
for each Hecke character $\phi: F^\times \backslash F^\times_{\mathbf A} \to {\mathbf C}^\times$ satisfying the following conditions: 
\begin{itemize}      
   \item the infinity type satisfies that $\phi_\infty(x) = x^{n-\alpha}$ for some $0\leq \alpha \leq n$;
   \item $\phi$ has a $p$-power conductor;
   \item ${\rm As}^+_{\mathcal M} (\pi) (\phi)$ is critical at $s=0$,
\end{itemize}      
we have 
\begin{align*}
   \widehat{\phi}( {\mathscr L}_p( {\rm As}^+_{\mathcal M} (\pi) ) ) 
   = {\mathcal E}_\infty(  {\rm As}^+_{\mathcal M} (\pi) (\phi) )
            {\mathcal E}_p(  {\rm As}^+_{\mathcal M} (\pi) (\phi) ) 
   \frac{ L(0,  {\rm As}^+_{\mathcal M} (\pi) (\phi)  )    }{  \Omega( {\rm As}^+_{\mathcal M} (\pi)  )  }. 
\end{align*}
}
\end{conj}

\begin{rem}
\begin{enumerate}
\item Conjecture \ref{conj:pL} coincides with \cite[page 111, Principal Conjecture]{co89} if $F={\mathbf Q}$. 
\item In the present paper, we prove a weaker version of Conjecture \ref{conj:pL}. 
        See Theorem \ref{thm:Main} for the result.  
\end{enumerate}
\end{rem}

\section{Eisenstein series}\label{sec:Eis}

In this section, we define distinguished Eisenstein series which are one of important ingredients for the construction of $p$-adic Asai $L$-functions. 
This Eisenstein series are essentially the Eisenstein series which are constructed from Siegel units in \cite{bl94} and \cite{ka04},  
and which are used by Loeffler-Williams in \cite{lw} if the base field is the rational number field. 
However,  
  since the motivic construction of these Eisenstein series is a difficult problem in the Hilbert modular setting so far, 
  the method in \cite{lw} is also not easy to generalize.   
In this section, we introduce and study a certain generalization of the Eisenstein series
    in the Hilbert modular setting by using the framework of the automorphic representation theory. 
We will introduce a method to construct $p$-adic Asai $L$-functions by using these Eisenstein series in the subsequent sections. 

In Section \ref{sec:eisdef}, we recall notion of Eisenstein series according to \cite{gj72}, \cite{ja72} and \cite{sc98}.
The Eisenstein series depend on the choice of test functions. 
In Section \ref{sec:choice}, we fix a distinguished test function which gives a generalization of Eisenstein series associated with Siegel units. 
The basic properties of these Eisenstein series are studied in Section \ref{sec:propeis}.

Hereafter to the end of this paper, let $F$ be a totally real field. 
We sometimes abbreviate $|\cdot|_{F_{\mathbf A}}$ to $|\cdot|_{\mathbf A}$ for the sake of the simplicity.

\subsection{Generality on Eisenstein series}\label{sec:eisdef}

In this subsection, we recall some basic definitions and facts on Eisenstein series on ${\rm GL}_2(F_{\mathbf A})$ according to \cite[Section 19]{ja72} and \cite[Section 11]{gj72}.  
The Eisenstein series which come from Siegel units are reviewed in \cite[Section 4]{sc98} in terms of the semi-adelic language.
We introduce a description of these Eisenstein series in terms of the adelic language and 
 we compare this description with the description in \cite{ja72} (see Remark \ref{rem:EisSum}).

Let ${\mathcal S}(\ast)$ be the space of Schwartz functions on $\ast$, where $\ast=F^{\oplus 2}_{\mathbf A}$ or $F^{\oplus 2}_v$. 
Consider two Hecke characters $\mu_1, \mu_2:F^\times \backslash F^\times_{\mathbf A} \to {\mathbf C}^\times$. 
For each $v\in \Sigma_F, \Phi_v\in {\mathcal S}(F^{\oplus 2}_v)$ and $s\in {\mathbf C}$, 
we also define ${\mathcal F}_v(-, s)={\mathcal F}_v(-, s ; \mu_{1,v}, \mu_{2,v}, \Phi_v): {\rm GL}_2(F_v) \to {\mathbf C}$ to be 
\begin{align}\label{eq:FDef}
  {\mathcal F}_v(g, s; \mu_{1,v}, \mu_{2,v}, \Phi_v) = \mu_{1, v}(\det g) |\det g|^{s+\frac{1}{2}}_v  
                                     \int_{F^\times_v} \Phi_v((0,t) g) \mu_{1,v}\mu^{-1}_{2,v}(t) |t|^{2s+1}_v  {\rm d}^\times t.   
\end{align}
Let $B_2$ be the Borel subgroup of ${\rm GL}_2$ of upper triangular matrices. 
Note that, for $\begin{pmatrix} a & b \\ 0 & d \end{pmatrix}   \in {\rm B}_2(F_v)$ and $g\in {\rm GL}_2(F_v)$, we have 
\begin{align}\label{eq:ind}
  {\mathcal F}_v\left( \begin{pmatrix} a & b \\ 0 & d \end{pmatrix} g, s ; \mu_{1,v}, \mu_{2,v}, \Phi_v\right) 
  = \mu_{1,v}(a)\mu_{2,v}(d)|\frac{a}{d}|^{s+\frac{1}{2}}_v 
      {\mathcal F}_v\left(  g, s ; \mu_{1,v}, \mu_{2,v}, \Phi_v \right).
\end{align}
Let $\Phi = \otimes_v \Phi_v$ and 
\begin{align*}
     {\mathcal F}(g, s; \mu_{1}, \mu_{2}, \Phi)
  =   \prod_{v\in \Sigma_F} {\mathcal F}_v(g_v, s; \mu_{1,v}, \mu_{2,v}, \Phi_v). 
\end{align*}
Then define the Eisenstein series 
$E(-,s; {\mathcal D})$ associated with the datum ${\mathcal D}:=(\mu_1, \mu_2, \Phi)$
to be 
\begin{align}\label{eq:EisDef}
E(g,s; {\mathcal D}) = \sum_{\gamma \in {\rm B}_2(F) \backslash {\rm GL}_2(F)} {\mathcal F}( \gamma g, s; \mu_1, \mu_2, \Phi).  
\end{align}

Recall that the Fourier transform $\widehat{\Phi}:=\otimes_v \widehat{\Phi}_v$ of $\Phi$ is defined by 
\begin{align*}
   \widehat{\Phi}_v(x,y) = \int_{F^{\oplus 2}_v}  \Phi_v(s,t) \psi_v(sy-tx) {\rm d}s {\rm d}t.    
\end{align*} 

\begin{prop}\label{prop:EisConv}
{\itshape 
The Eisenstein series $E(g,s; {\mathcal D})$ is absolutely convergent for ${\rm Re}(s)\gg 0$
 and it is meromorphically continued to the whole ${\mathbf C}$-plane. 
The Eisenstein series $E(g,s; {\mathcal D})$ has a possible pole at $s\in {\mathbf C}$ with 
\begin{align*}
\mu_1\mu^{-1}_2 |\cdot|^{2s}_{\mathbf A} = |\cdot|^{\pm 1}_{\mathbf A}. 
\end{align*}
If $\widehat{\Phi}(0,0)=0$ {\rm (}resp. $\Phi(0,0)=0${\rm )}, then $E(g,s; {\mathcal D})$ does not have a pole  
even if  $\mu_1\mu^{-1}_2|\cdot|^{2s}_{\mathbf A} = |\cdot|_{\mathbf A}$
{\rm (}resp.  $\mu_1\mu^{-1}_2|\cdot|^{2s}_{\mathbf A} = |\cdot|^{- 1}_{\mathbf A}${\rm )}. 
}
\end{prop}
\begin{proof}
The statement on the convergence on ${\rm Re}(s)\gg 0$ and the meromorphic continuation immediately follows from \cite[Proposition 19.3, (19.6)]{ja72}. 
See \cite[(19.6)]{ja72} and \cite[Lemma 11.2, Lemma 11.5]{gj72} for the statement on the poles.
\end{proof}

For later use, we recall a notion of the M\"{o}bius transform. 
For each place $v\in \Sigma_F$, we define the M\"{o}bius transform ${\mathcal M}_v{\mathcal F}_v\in {\mathcal S}(F^{\oplus 2}_v)$ of ${\mathcal F}_v\in {\mathcal S}(F^{\oplus 2}_v)$ to be 
\begin{align*}
  {\mathcal M}_v{\mathcal F}_v(g,s) 
  =  \int_{ F_v} {\mathcal F}_v \left( w_2\begin{pmatrix} 1 & u \\ 0 & 1 \end{pmatrix} g , s\right)  {\rm d}u, 
  \quad \left(w_2 = \begin{pmatrix} 0 & 1 \\ -1 & 0 \end{pmatrix} \right). 
\end{align*}
We note that   
\begin{align*}
      {\mathcal M}_v{\mathcal F}_v\left(\begin{pmatrix} a & b \\ 0 & d  \end{pmatrix}  g, s\right) 
=& \int_{ F_v} {\mathcal F}_v \left(  \begin{pmatrix} d & 0 \\ 0 & a \end{pmatrix}  w_2\begin{pmatrix} 1 & \frac{du+b}{a} \\ 0 & 1 \end{pmatrix} 
                                                         g , s\right)  {\rm d}u  \\
=&  \mu_{1,v}(d)\mu_{2,v}(a) |\frac{d}{a}|^{s-\frac{1}{2}}_v
  {\mathcal M}_v{\mathcal F}_v(g,s) .
\end{align*}
For ${\mathcal F} =\otimes_v {\mathcal F}_v\in {\mathcal S}(F^{\oplus 2}_{\mathbf A})$,  
define 
\begin{align*}
  {\mathcal M} {\mathcal F}  = \prod_{v\in \Sigma_F} {\mathcal M}_v {\mathcal F}_v. 
\end{align*}

We introduce another type of Eisenstein series, which is essentially introduced in \cite[Section 4.3]{sc98}.  
Define 
\begin{align*}
  {\mathcal F}_\Phi(g, s; \mu_{1,\infty}, \mu_{2,\infty})  
  = \sum_{x\in F^\times}   {\mathcal F}_\infty(xg, s; \mu_{1,\infty}, \mu_{2,\infty},  \Phi_\infty)  \Phi_{\rm fin}( (0,1) xg ) \mu_{1,{\rm fin}}(\det xg) |\det xg |^{s+\frac{1}{2}}_{\rm fin}.   
\end{align*}
For each Hecke character $\chi: F^\times \backslash F^\times_{\mathbf A} \to {\mathbf C}^\times$ with $\chi(t) = \mu_1\mu_2(t)$ for $t\in F^\times_\infty$, define 
\begin{align*}
  {\mathcal F}^\chi_\Phi(g,s; \mu_{1,\infty}, \mu_{2,\infty} ) 
      = \int_{ F^\times \backslash F^\times_{\mathbf A}/ F^\times_\infty} 
                  \chi(t) {\mathcal F}_\Phi(t^{-1} g, s; \mu_{1,\infty}, \mu_{2,\infty} )  {\rm d}^\times t.     
\end{align*}
This is well-defined, since 
\begin{align*}
  {\mathcal F}_\Phi(tg, s; \mu_{1,\infty}, \mu_{2,\infty}) = \mu_{1,\infty}\mu_{2,\infty}(t) {\mathcal F}_\Phi(g, s; \mu_{1,\infty}, \mu_{2,\infty}).  
\end{align*}
for each $t\in F^\times_\infty$.
The following lemma immediately follows from the definition of ${\mathcal F}_\Phi(g, s; \mu_{1,\infty}, \mu_{2,\infty})$:

\begin{lem}\label{lem:SchEis}
{\itshape 
\begin{enumerate}
  \item \label{lem:SchEis(i)}
           For $\gamma \in {\rm B}_2(F)$, we have 
           \begin{align*}
              {\mathcal F}_\Phi( \gamma g,s; \mu_{1,\infty}, \mu_{2,\infty}) = {\mathcal F}_\Phi(g,s; \mu_{1,\infty}, \mu_{2,\infty}).   
           \end{align*}
  \item \label{lem:SchEis(ii)}
            Define $\mu_\chi: F^\times_{\mathbf A} \to {\mathbf C}^\times$ to be $\mu_\chi = \mu^{-1}_1 \chi$.
           Then we have 
            \begin{align*}
              {\mathcal F}^\chi_\Phi (g,s; \mu_{1,\infty}, \mu_{2,\infty}) 
               = \prod_{v\in \Sigma_F} {\mathcal F}_v(g,s ; \mu_{1,v}, \mu_{\chi, v}, \Phi_v).
            \end{align*}
\end{enumerate}
}
\end{lem}

\begin{dfn}\label{dfn:SchEis}
We define 
\begin{align*}
  E_\Phi(g,s; \mu_{1,\infty}, \mu_{2,\infty})  
     = \sum_{\gamma \in {\rm B}_2(F)\backslash {\rm GL}_2(F)} {\mathcal F}_\Phi(\gamma g, s; \mu_{1,\infty}, \mu_{2,\infty}).   
\end{align*}
This is well-defined by Lemma \ref{lem:SchEis} \ref{lem:SchEis(i)}.
\end{dfn}

\begin{rem}\label{rem:EisSum}
As an analogue to \cite[Section 4.3, page 434]{sc98}, 
the isotypic decomposition and Lemma \ref{lem:SchEis} \ref{lem:SchEis(ii)} show that 
\begin{align*}
E_\Phi(g,s; \mu_{1,\infty}, \mu_{2,\infty})
   = \sum_\chi E(g,s; \mu_1, \mu_\chi, \Phi),
\end{align*}
where $\chi$ runs over a finite set of Hecke characters satisfying $\chi_\infty=\mu_{1,\infty}\mu_{2,\infty}$.
Hence the analytic properties of $E_\Phi(g,s; \mu_{1,\infty}, \mu_{2,\infty})$
   can be deduced from the properties of $E(g,s; \mu_1, \mu_\chi, \Phi)$.
\end{rem}

\subsection{Choice of sections}\label{sec:choice}

In this subsection, we introduce a certain test function $\Phi^{(r)} = \otimes_{v\in \Sigma_F} \Phi^{(r)}_v\in {\mathcal S}(F^{\oplus 2}_{\mathbf A})$ for the construction of Eisenstein series. 
See Definition \ref{dfn;BSfn} for our choice of distinguished test functions. 

Consider $\alpha = \sum_{\sigma \in I_F}  \alpha_\sigma \sigma \in {\mathbf Z}[I_F]$ satisfying 
  conditions (Alp1), (Alp2) in Section \ref{sec:AsaiMot}.  
 Let $\phi_1: F^\times\backslash F^\times_{\mathbf A} \to {\mathbf C}^\times$ be $   \phi_1 = |\cdot|^{-(\alpha+2m)+\frac{1}{2}}_{\mathbf A}$
 and $\phi_{2} : F^\times\backslash  F^\times_{\mathbf A} \to {\mathbf C}^\times$ a Hecke character such that 
 \begin{align*}
     \phi_{2, \infty}(x) =         x^{ -2(n-\alpha) -(\alpha+2m) - \frac{1}{2}},  
                                           \quad \left(x\in F^\times_{\infty, +} \right). 
 \end{align*} 
Note that $\alpha_\sigma+2m_\sigma$ also does not depend on $\sigma\in I_F$ by the assumption on $m$ and $\alpha$.  
Put $n_\alpha=2n-2\alpha$ and $k_\alpha=n_\alpha+2t$. 
We also define 
\begin{align}\label{eq:UnitChar1}
   \omega_\phi =  \frac{\phi_1\phi_2}{ |\cdot |^{-n_\alpha -2(\alpha+2m)}_{\mathbf A} },   \quad 
   \omega^\prime_\phi = \frac{\phi_1\phi^{-1}_2}{  |\cdot |^{n_\alpha+1}_{\mathbf A}  }, 
\end{align}
then both $\omega_\phi$ and $\omega^\prime_\phi$ are finite-order.
Furthermore we suppose that 
\begin{align}\label{eq:UnitChar2}
\omega_\phi =  \omega^{-1}_{\pi}|_{F^\times_{\mathbf A}}.  
\end{align}
Since $\omega_\pi|_{F^\times_\infty}$ is the trivial character, 
the assumption (\ref{eq:UnitChar2}) implies that  $\phi_{1,\infty} \phi_{2,\infty} = |\cdot|^{-n_\alpha -2(\alpha+2m) }_{\infty}$ and $\phi_{1,\infty} \phi^{-1}_{2,\infty} = |\cdot|^{n_\alpha+1}_{\infty}$.

Let ${\mathfrak N}\subset \widehat{\mathcal O}_E$ be the conductor of $\pi$, 
and put ${\mathfrak N}_F= {\mathfrak N}\cap \widehat{\mathcal O}_F$.

\begin{dfn}\label{def:aux}

Let $v_0$ be a prime of $F$  with $v_0\nmid p{\mathfrak N}_F{\rm Disc}(E/F)$, 
     where ${\rm Disc}(E/F)$ is the discriminant of $E/F$.  
We call $v_0$ an auxiliary prime if $v_0$ satisfies the following conditions:
\begin{itemize}
\item[(Aux1)] 
$\phi_{2, v_0}$ and $\pi_{v_0}$ are unramified. 
\item[(Aux2)]  $q^2_{v_0} \not \equiv 1 \ {\rm mod} \ p$. 
\end{itemize}
\end{dfn}

\begin{dfn}\label{dfn;BSfn}
Let $r\in {\mathbf Z}[\Sigma_{F,p}]$ with $r_{v_1}=r_{v_2}\geq 0$ for each $v_1, v_2\mid p$.
We define a Bruhat-Schwartz function $\Phi^{(r)} = \otimes_v \Phi^{(r)}_{v} \in {\mathcal S}(F^{\oplus 2}_{\mathbf A})$ as follows: 
\begin{itemize}
  \item for each infinite place $\sigma \in \Sigma_F$, 
             $\Phi^{(r)}_\sigma(x,y) = 2^{-k_{\alpha,\sigma}} (x+\sqrt{-1}y)^{k_{\alpha, \sigma}} e^{-\pi(x^2+y^2)}$; 
  \item for each $p$-adic place $v\in \Sigma_F$ ($v\mid p$), 
             $\Phi^{(r)}_v(x,y) = \psi_{F, v}\left( \frac{ x}{  p^{r_v}  } \right) {\mathbb I}_{\mathcal O^{\oplus 2}_{F,v}}(x,y)$;
  \item for each place $v\in \Sigma_F$ ($v\nmid \infty p {\mathfrak N}_F $),  
             $\Phi^{(r)}_v(x,y) =  {\mathbb I}_{   {\mathcal O}^{\oplus 2}_{F,v}    }(x,y)$;
  \item for each finite place $v$ with $v\mid   {\mathfrak N}_F$ and $v\mid p$,  
             $\Phi^{(r)}_v(x,y) = \omega^{-1}_{\pi, v}(y) {\mathbb I}_{{\mathcal R}_v}(x,y)$, 
             where $\omega_\pi$ is the central character of $\pi$ and 
             \begin{align*}
               {\mathcal R}_v = \{ ( x,  y) \mid  x \in  {\mathfrak N}_{F,v}, y \in {\mathcal O}^\times_{F,v}  \}. 
             \end{align*} 
   \item If the following two conditions hold:
            \begin{itemize}
              \item $n-\alpha=0$; 
              \item ${\rm ord}_v({\mathfrak N}_F) = 0$ for each $v\in \Sigma_F$ with $v\nmid p$, 
            \end{itemize}    
            fix an auxiliary prime $v_0\in \Sigma_F$. Then define 
          \begin{align*} 
             \Phi^{(r)}_{v_0}(x,y) 
                  = {\mathbb I}_{ {\mathcal O}^\times_{F,v_0} \oplus {\mathcal O}_{F,v_0}  }(x,y).
            \end{align*} 

\end{itemize}
\end{dfn}

Note that $\Phi^{(r)}_v$ does not depend on $r$ for each $v\nmid p$. 
Hence we sometimes write $\Phi_v$ for $\Phi^{(r)}_v$ if $v\in \Sigma_F$ does not divide $p$.
Let $\Phi^{ (r) } = \prod_{v\in \Sigma_F} \Phi^{ (r)}_v$ and consider the datum 
\begin{align*}
   {\mathcal D}_r := (\phi_1, \phi_2, \Phi^{(r)}). 
\end{align*}
Denote by ${\mathcal F}_v(g_v, s; {\mathcal D}_r)$ the section in (\ref{eq:FDef}) associated with  the datum ${\mathcal D}_r$. 
Let ${\mathcal F}(g, s; {\mathcal D}_r) = \prod_v {\mathcal F}_v(g_v, s; {\mathcal D}_r)$ and define  
the Eisenstein series $E(g,s; {\mathcal D}_r) $ associated with ${\mathcal D}_r$ as in (\ref{eq:EisDef}).
We also define $E_{\Phi^{(r)}}(g,s;\phi_{1,\infty}, \phi_{2,\infty})$ as in Definition \ref{dfn:SchEis}. 
Since hereafter we always fix $\Phi^{(r)}$, $\phi_{1, \infty}$ and $\phi_{2,\infty}$, 
we sometimes abbreviate $E_{\Phi^{(r)}}(g,s;\phi_{1,\infty}, \phi_{2,\infty})$ to $E_{r,s}(g)$ for the sake of the simplicity

\begin{rem}
For each infinite place $\sigma\in \Sigma_F$, we find that 
\begin{align*}
  {\mathcal F}_\sigma\left( \begin{pmatrix} \cos\theta & \sin\theta \\ -\sin\theta & \cos\theta \end{pmatrix} , s; {\mathcal D}_r  \right) 
  = e^{\sqrt{-1}k_{\alpha,\sigma} \theta} \cdot 2^{-k_{\alpha,\sigma} } \sqrt{-1}^{k_{\alpha,\sigma} } 
     \pi^{- \left(s+\frac{k_{\alpha,\sigma} +1}{2} \right)} \Gamma\left( s+\frac{k_{\alpha,\sigma} +1}{2} \right) .
\end{align*}
Hence the Eisenstein series in the above discussion are holomorphic Hilbert modular forms of a parallel weight $k_\alpha$.
\end{rem}

\subsection{Properties of distinguished Eisenstein series}\label{sec:propeis}

In this subsection, 
we introduce some properties on Eisenstein series 
$E(g,s; {\mathcal D}_r)$ and $E_{r,s}(g)$
associated with the distinguished section $\Phi^{(r)}$ in Definition \ref{dfn;BSfn}.

\subsubsection{Holomorphy}

We discuss the holomorphy of the Eisenstein series which we fixed in Section \ref{sec:choice} by using Proposition \ref{prop:EisConv}.

\begin{lem}\label{lem:PhiZero}
{\itshape 
For each $v\mid {\mathfrak N}_F$ and $v\nmid p$, we have $\Phi_v (0,0) = \widehat{\Phi}_v (0,0) = 0$.   
Furthermore, for the fixed auxiliary prime $v_0$, we have $\Phi_{v_0} (0,0) = \widehat{\Phi}_{v_0} (0,0) = 0$.   
}
\end{lem}
\begin{proof}
For each $v\mid {\mathfrak N}_F$ and $v\nmid p$, $\Phi_v(0,0)=0$ by the definition of $\Phi_v$. 
Note that 
   for each character $\chi:F^\times_v\to {\mathbf C}^\times$ of the conductor $\varpi^{c(\chi)}_v{\mathcal O}_{F,v}$, 
   the Fourier transform of $\phi_\chi(x):= \chi(x) {\mathbb I}_{{\mathcal O}^\times_{F,v}}(x)$ is given by 
   \begin{align*} 
     \widehat{\phi}_{\chi}(x) = \chi^{-1}(x)  {\mathbb I}_{\varpi^{-c(\chi)}_v {\mathcal O}^\times_{F,v}} (x) \epsilon(1, \chi^{-1}, \psi_{F,v}),  
   \end{align*} 
   where $\epsilon(s, \chi, \psi_{F,v})$ is the $\epsilon$-factor for $\chi$.
   This yields that 
\begin{align*}
     \widehat{\Phi}_v (x,y)  = \widehat{\phi}_{\omega_{\pi, v}}(-x)  \times {\mathbb I}_{  \frac{1}{\mathfrak N_{F,v}} {\mathcal O}_{F,v} }(y). 
\end{align*}
Hence $\widehat{\Phi}_v (0,0) = 0$.

For the auxiliary prime $v_0$, we have $\Phi_{v_0}(0,0)=0$ by the definition. 
We also find that 
  \begin{align*}
    \widehat{\Phi}_{v_0}(x,y) 
        =  {\mathbb I}_{  {\mathcal O}_{F,v_0} \oplus {\mathcal O}^\times_{F, v_0} } (x,y).
  \end{align*}   
In particular, we have $\widehat{\Phi}_{v_0}(0,0)=0$.
This shows the statement. 
\end{proof}

\begin{prop}\label{prop:EisHol}
{\itshape 
As functions on $s\in {\mathbf C}$, the Eisenstein series $E(g,s; {\mathcal D}_r)$ and $E_{r,s}(g)$ are absolutely convergent for ${\rm Re}(s)\gg 0$
and it is holomorphically continued to the whole ${\mathbf C}$-plane.
}
\end{prop}
\begin{proof}
Apply Proposition \ref{prop:EisConv} for $\mu_i=\phi_i$ ($i=1,2$). 
Proposition \ref{prop:EisConv} implies that 
the poles of $E(g,s; {\mathcal D}_r)$ might appear 
    at $s+n-\alpha=0$ or $s+n-\alpha+1=0$ 
    if the character $\omega^\prime_\phi$ in (\ref{eq:UnitChar1}) is trivial.
However even in these cases, $E(g, s; {\mathcal D}_r)$ does not have a pole by Proposition \ref{prop:EisConv} and Lemma \ref{lem:PhiZero}.

The statement for $E_{r,s}(g)$ follows from Remark \ref{rem:EisSum} and the statement for $E(g, s; {\mathcal D}_r)$.
\end{proof}

\begin{rem}
In the later discussion, we need Eisenstein series $E_{r,s}$ at $s=0$.   
By the proof of Lemma \ref{lem:PhiZero} and Proposition \ref{prop:EisHol}, 
$E_{r,s}$ is holomorphic at $s=0$ without using an auxiliary prime $v_0$
if the either one of the following conditions holds:
\begin{itemize}
   \item $n-\alpha\neq 0$.
   \item for some $v\in \Sigma_F$ with $v\nmid p$, ${\rm ord}_v({\mathfrak N}_F)>0$.  
\end{itemize}
Hence, only if both two conditions that $n-\alpha=0$ 
    and that ${\rm ord}_v({\mathfrak N}_F) = 0$ for each $v\nmid p$ are satisfied, 
we use an auxiliary prime $v_0$ for the convergence of Eisenstein series. 
This explains the reason of the condition in the definition of $\Phi_{v_0}$ in Definition \ref{dfn;BSfn}. 
\end{rem}

\subsubsection{Constant term}

We compute the constant term of the Eisenstein series $E_{0, s}$ in the following Lemma \ref{lem:FVal}. 
This result is necessary to show the Eisenstein cohomology class associated with $E_{0,0}$ is rational in Proposition \ref{prop:RatEis}.

\begin{lem}\label{lem:FVal} 
{\itshape 
We abbreviate ${\mathcal F}_v(g_v, s; {\mathcal D}_r)$ to ${\mathcal F}_v(g_v, s)$  in this lemma.
Recall that we denote by $v_0$ a fixed auxiliary prime. 
\begin{enumerate}
\item Let $g=\begin{pmatrix} a & b \\ 0 & d \end{pmatrix} \begin{pmatrix} \cos\theta & \sin \theta  \\ - \sin\theta & \cos \theta \end{pmatrix} \in {\rm GL}_2(F_\sigma)$
          for each $\sigma \in \Sigma_{F, \infty}$. 
         Then we find that 
          \begin{align*}
                {\mathcal F}_\sigma(g, s) 
                =& \phi_1(a)\phi_2(d) |\frac{a}{d}|^{s+\frac{1}{2}}_\sigma e^{\sqrt{-1}k_{\alpha, \sigma} \theta }
                    \times 2^{s-1} \sqrt{-1}^{ k_{\alpha, \sigma} }   \times \Gamma_{\mathbf C}(s+k_{\alpha, \sigma}),  \\
               {\mathcal M}_\sigma{\mathcal F}_\sigma(g,s)
               =&  \phi_1(d)\phi_2(a) |\frac{d}{a}|^{s-\frac{1}{2}}_v  e^{\sqrt{-1}k_{\alpha, \sigma} \theta }
                    \times (-1)2^{-( s + n_{\alpha,\sigma} + 1)}  
        \times  \frac{  \Gamma_{\mathbf C}(1+n_{\alpha,\sigma}+2s)  
               }{  \Gamma_{\mathbf C}\left( s  \right) }.  
          \end{align*}         
\item Let $v\in \Sigma_F$ be a finite place with $v\nmid p{\mathfrak N}_F v_0$ and 
         write $g=\begin{pmatrix} a & b \\ 0 & d \end{pmatrix} k$ for some $k \in {\rm GL}_2({\mathcal O}_{F,v})$. 
         Then we find that 
          \begin{align*}
                {\mathcal F}_v(g, s) 
                =& \phi_1(a)\phi_2(d) |\frac{a}{d}|^{s+\frac{1}{2}}_v 
                    \times L_v(2s+1, \phi_{1,v}\phi^{-1}_{2,v})  \\
                =& \phi_1(a)\phi_2(d) |\frac{a}{d}|^{s+\frac{1}{2}}_v 
                    \times L_v(2(s+n-\alpha+1), \omega^\prime_{\phi, v} ),     \\     
               {\mathcal M}_v{\mathcal F}_v(g,s)
               =&  \phi_1(d)\phi_2(a) |\frac{d}{a}|^{s-\frac{1}{2}}_v 
                    \times L_v(2s, \phi_{1,v}\phi^{-1}_{2,v})  \\
                =& \phi_2(a) \phi_1(d) |\frac{a}{d}|^{-s+\frac{1}{2}}_v 
                    \times L_v(2(s+n-\alpha)+1, \omega^\prime_{\phi, v} )
          \end{align*}
\item Let $v\in \Sigma_F$ be a finite place with $v\mid  {\mathfrak N}_F $ and $v\nmid p$. 
         Then we find that 
          \begin{align*}
                {\mathcal F}_v(g, s) 
                =& \phi_1(a)\phi_2(d) |\frac{a}{d}|^{s+\frac{1}{2}}_v
                    \times   \begin{cases}    
                                            1,    &  \left(g\in \begin{pmatrix} a & b \\ 0 & d \end{pmatrix} {\mathcal K}_0({\mathfrak N}_{F,v})  
                                                             \right),    \\
                                            0, & (\text{\rm otherwise}).     \end{cases}   
          \end{align*}
         Let $w_2 = \begin{pmatrix} 0 & 1 \\ -1 & 0 \end{pmatrix}$. Then we also have 
          \begin{align*}
                {\mathcal M}_v{\mathcal F}_v(g,s)
                =&  \phi_1(d)\phi_2(a) |\frac{d}{a}|^{s-\frac{1}{2}}_v   
                       \times \begin{cases}  
                                   \omega_{\pi, v}(-1) {\rm vol}( {\mathfrak N}_{F,v} , {\rm d}u),  
                                     &  \left(g\in \begin{pmatrix} a & b \\ 0 & d \end{pmatrix} w_2 {\mathcal K}_0({\mathfrak N}_{F,v} )   
                                                 \right),    \\
                                      0, & (\text{\rm otherwise}).     \end{cases}  
          \end{align*}
  \item Let $v$ be a fixed auxiliary prime $v_0$.    
           Then we find that 
          \begin{align*}
                {\mathcal F}_{v_0}(g, s) 
                =& \phi_1(a)\phi_2(d) |\frac{a}{d}|^{s+\frac{1}{2}}_{v_0}
                    \times   \begin{cases}    
                                            1, &  \left(g\in \begin{pmatrix} a & b \\ 0 & d \end{pmatrix} w_2 {\mathcal K}_0( \varpi_{v_0} )  \right),    \\
                                            0, & (\text{\rm otherwise}).     \end{cases}     \\
                {\mathcal M}_{v_0}{\mathcal F}_{v_0}(g,s)
                =&  \phi_1(d)\phi_2(a) |\frac{d}{a}|^{s-\frac{1}{2}}_{v_0}  
                    \times   \begin{cases} 
                         1,    &  \left( g\in \begin{pmatrix} a & b \\ 0 & d \end{pmatrix} {\mathcal K}_0(  \varpi_{v_0} ) \right),    \\
                          0, & (\text{\rm otherwise}).     \end{cases}  
          \end{align*}
\end{enumerate}
}
\end{lem}
\begin{proof}
We prove the first statement. Let $\sigma\in \Sigma_F$ be an infinite place of $F$. 
Since $k_\alpha$ is even, 
the condition (\ref{eq:UnitChar2}) shows that $\phi^{-1}_{1,\sigma} \phi_{2, \sigma}(-1)=1$. 
Hence we find that 
\begin{align*}
       {\mathcal F}_\sigma(1_2, s) 
   =&   \int_{F^\times_\sigma} \Phi^{ (r)}_\sigma(0, t) \phi_{1,\sigma}\phi^{-1}_{2,\sigma}(t) |t|^{2s+1}_\sigma  {\rm d}^\times t   \\
   =& 2 \int^\infty_{0} 
                2^{-k_{\alpha, \sigma}}     ( \sqrt{-1} t )^{ k_{\alpha, \sigma} } 
                e^{-\pi t^2 }
                t^{k_{\alpha, \sigma}-1}    |t|^{2s+1}_v  {\rm d}^\times t   \\
    =&   2^{-k_{\alpha, \sigma}}  \sqrt{-1}^{ k_{\alpha, \sigma} }  
            \pi^{-s-k_{\alpha, \sigma}}\Gamma(s+k_{\alpha, \sigma})   \\
    =& 2^{s-1} \sqrt{-1}^{ k_{\alpha, \sigma} }   \times
            2(2\pi)^{-s-k_{\alpha, \sigma}}\Gamma(s+k_{\alpha, \sigma})   \\  
    =& 2^{s-1} \sqrt{-1}^{ k_{\alpha, \sigma} }   \times \Gamma_{\mathbf C}(s+k_{\alpha, \sigma}).               
\end{align*}
Compute $\displaystyle \int_{ F_\sigma} {\mathcal F}\left(   w_2 \begin{pmatrix} 1 & u  \\ 0 & 1 \end{pmatrix}, s \right) {\rm d}u$. 
Note that 
\begin{align*}
  w_2 \begin{pmatrix} 1 & u  \\ 0 & 1 \end{pmatrix}
  =\begin{pmatrix} 0 & 1  \\  -1 & 0  \end{pmatrix} \begin{pmatrix} 1 & u  \\ 0 & 1 \end{pmatrix}
  = \begin{pmatrix} (1+u^2)^{-\frac{1}{2}} & 0  \\ 0 &  (1+u^2)^{ \frac{1}{2}}  \end{pmatrix}
     \begin{pmatrix} 1 & -u  \\ 0 & 1 \end{pmatrix} 
     \begin{pmatrix} \frac{-u}{(1+u^2)^{\frac{1}{2}}} & \frac{1}{(1+u^2)^{\frac{1}{2}}}  \\ -\frac{1}{(1+u^2)^{\frac{1}{2}}} & \frac{-u}{(1+u^2)^{\frac{1}{2}}} \end{pmatrix}.           
\end{align*}
Put 
\begin{align}\label{eq:cossin}
  \cos\theta = \frac{-u}{(1+u^2)^{\frac{1}{2}}},  \quad 
   \sin\theta = \frac{1}{(1+u^2)^{\frac{1}{2}}}. 
\end{align}
Then we find that 
\begin{align*}
     & \int_{ F_\sigma} {\mathcal F}\left(   w_2 \begin{pmatrix} 1 & u  \\ 0 & 1 \end{pmatrix}, s \right) {\rm d}u   \\ 
  =& \int_{ F_\sigma} {\rm d}u \int_{ F^\times_\sigma} {\rm d}^\times t 
            \Phi^{(r)}_\sigma \left( 0, t(1+u^2)^{\frac{1}{2}} \right) 
             e^{\sqrt{-1} k_{\alpha, \sigma} \theta } 
             \phi_{1,\sigma} \phi^{-1}_{2,\sigma} (t) |t|^{2s+1}_\sigma  \\
  =& \int_{ F_\sigma}  \phi_{1,\sigma}  \phi^{-1}_{2,\sigma}   \left(  (1+u^2)^{-\frac{1}{2}}  \right) 
                                     |  (1+u^2)^{-\frac{1}{2}}  |^{2s+1}_\sigma  
                                 e^{\sqrt{-1} k_{\alpha, \sigma} \theta } {\rm d}u 
       \int_{ F^\times_\sigma} {\rm d}^\times t 
            \Phi^{(r)}_\sigma \left( 0, t \right)               
             \phi_{1,\sigma}  \phi^{-1}_{2,\sigma} (t) |t|^{2s+1}_\sigma.
\end{align*}
The equation (\ref{eq:cossin}) shows that 
\begin{align*}
  &  \int_{ F_\sigma}  \phi_{1,\sigma}  \phi^{-1}_{2,\sigma}   \left(  (1+u^2)^{-\frac{1}{2}}  \right) 
                                     |  (1+u^2)^{-\frac{1}{2}}  |^{2s+1}_\sigma  
                                 e^{\sqrt{-1} k_{\alpha, \sigma} \theta } {\rm d}u    \\
=&  \int^\pi_{ 0}  \sin^{k_{\alpha, \sigma}+2s}\theta 
                                 e^{\sqrt{-1} k_{\alpha, \sigma} \theta }  \cdot  \frac{ -{\rm d}\theta }{\sin^2\theta}.
\end{align*}
Apply a formula (see \cite[page 8]{mos66})
\begin{align*}
    \int^\pi_{ 0} \sin^a\theta e^{\sqrt{-1}b \theta}  {\rm d} \theta  
  = \frac{2^{-a} \pi \Gamma(1+a) e^{\frac{\sqrt{-1}\pi b }{2}} 
               }{\Gamma\left( \frac{a}{2} + \frac{b}{2} + 1 \right) 
                 \Gamma\left( \frac{a}{2} - \frac{b}{2} + 1 \right) }, 
      \quad (a, b \in {\mathbf C}, {\rm Re}(a)>-1)
\end{align*}
for $a=n_{\alpha,\sigma}+2s, b=k_{\alpha, \sigma}$, 
then we have 
\begin{align*}
  & - \int^\pi_{ 0}  \sin^{n_{\alpha, \sigma}+2s}\theta 
                                 e^{\sqrt{-1} k_{\alpha, \sigma} \theta }  {\rm d}\theta  \\
= & \frac{-2^{-(n_{\alpha,\sigma}+2s)} \pi \Gamma(1+n_{\alpha,\sigma}+2s) e^{\frac{\sqrt{-1}\pi k_{\alpha, \sigma}}{2}} 
               }{\Gamma\left( \frac{n_{\alpha,\sigma}+2s}{2} + \frac{k_{\alpha, \sigma}}{2} + 1 \right) 
                 \Gamma\left( \frac{n_{\alpha,\sigma}+2s}{2} - \frac{k_{\alpha, \sigma}}{2} + 1 \right) }      \\                           
= & -2^{-(n_{\alpha,\sigma}+2s)} \pi   \sqrt{-1}^{k_{\alpha, \sigma} }
        \times  \frac{  2 (2\pi)^{-(2s+n_{\alpha, \sigma} +1)} \Gamma(1+n_{\alpha,\sigma}+2s)  
               }{ 2(2\pi)^{-(s+k_{\alpha, \sigma}) } \Gamma\left(  s  + k_{\alpha, \sigma}  \right) 
                  \times 2(2\pi)^{- s   } \Gamma\left( s  \right) }   
         \times 2 (2\pi)^{-1}     \\                           
= & -2^{-(n_{\alpha,\sigma}+2s)}  \sqrt{-1}^{k_{\alpha, \sigma} }
        \times  \frac{  \Gamma_{\mathbf C}(1+n_{\alpha,\sigma}+2s)  
               }{ \Gamma_{\mathbf C}\left(  s  + k_{\alpha, \sigma}  \right) 
                  \times  \Gamma_{\mathbf C}\left( s  \right) }.    
\end{align*}
Hence we find that 
\begin{align*}
   & \int_{ F_\sigma} {\mathcal F}\left(   w_2 \begin{pmatrix} 1 & u  \\ 0 & 1 \end{pmatrix}, s \right) {\rm d}u   \\  
  =&-2^{-(n_{\alpha,\sigma}+2s)}  \sqrt{-1}^{k_{\alpha, \sigma} }
        \times  \frac{  \Gamma_{\mathbf C}(1+n_{\alpha,\sigma}+2s)  
               }{ \Gamma_{\mathbf C}\left(  s  + k_{\alpha, \sigma}  \right) 
                  \times  \Gamma_{\mathbf C}\left( s  \right) }    
     \times 2^{s-1} \sqrt{-1}^{ k_{\alpha, \sigma} }   \times \Gamma_{\mathbf C}(s+k_{\alpha, \sigma})   \\  
  =& -2^{-( s + n_{\alpha,\sigma} + 1)}  
        \times  \frac{  \Gamma_{\mathbf C}(1+n_{\alpha,\sigma}+2s)  
               }{  \Gamma_{\mathbf C}\left( s  \right) }.    
\end{align*}

We prove the second statement.  
We may assume that $g=1_2$. Then we find that  
\begin{align*}
        \int_{F^\times_v}  \Phi^{ (r) }_v(0, t) \phi_{1,v}\phi^{-1}_{2,v}(t) |t|^{2s+1}_v  {\rm d}^\times t   
   =& \sum^\infty_{n=0}    
             \phi_{1,v}\phi^{-1}_{2,v}( \varpi^{n}_v ) q^{-n(2s+1)}_v   
   = \frac{ 1 }{  1- \phi_{1,v}\phi^{-1}_{2,v}( \varpi_v ) q^{-(2s+1)}_v }. 
\end{align*}
The second identity is also proved in a similar way (see \cite[Section 3.7, (7.27)]{bu98}).
This proves the statement. 

We prove the third statement. 
Recall a decomposition:
\begin{align*}
   {\rm GL}_2(F_v) = \bigcup^{{\rm ord}_v({\mathfrak N}_F)}_{e=0}  {\rm B}_2(F_v) \begin{pmatrix}  1 & 0 \\  \varpi^e_v & 1 \end{pmatrix} {\mathcal K}_0( {\mathfrak N}_{F,v} ). 
\end{align*}

We prove the first identity. 
If $g= \begin{pmatrix} 1 & 0 \\ \varpi^{e}_v & 1 \end{pmatrix} (e=0,\ldots, {\rm ord}_v({\mathfrak N}_F)-1  )$, then we have 
\begin{align*}
 \Phi^{ (r) }_v \left( (0,1) g  \right)
  = \Phi^{ (r) }_v \left( \varpi^{e}_v , 1 \right) =0. 
\end{align*}
Hence ${\mathcal F}_v(g,s)$ is supported on ${\rm B}_2(F_v)  {\mathcal K}_0( {\mathfrak N}_{F, v})$.
If $g= 1_2$, then the condition (\ref{eq:UnitChar2}) shows that    
\begin{align*}
        \int_{F^\times_v}  \Phi^{ (r) }_v(0, t) \phi_{1,v}\phi^{-1}_{2,v}(t) |t|^{2s+1}_v  {\rm d}^\times t   =1.
\end{align*} 

We prove the second identity.
If $g= \begin{pmatrix} 1 & 0 \\ \varpi^{e}_v & 1 \end{pmatrix} (e=1,\ldots, {\rm ord}_v({\mathfrak N}_{F,v}))$, then we have 
\begin{align*}
\Phi^{ (r) }_v \left( (0,1)w_2\begin{pmatrix} 1 & u \\ 0 & 1 \end{pmatrix}  g  \right)  
    = \Phi^{ (r) }_v \left( (0,1) \begin{pmatrix} 0 & 1 \\ -1 & -u \end{pmatrix} g  \right)
  = \Phi^{ (r) }_v \left( -1-u\varpi^e_v,-u \right) =0,  
\end{align*}
for each $u\in F_v$.
Hence ${\mathcal M}_v{\mathcal F}_v(g,s)$ is supported on ${\rm B}_2(F_v)  \begin{pmatrix} 1 & 0 \\ 1 & 1 \end{pmatrix}    {\mathcal K}_0( {\mathfrak N}_{F, v})$.
Suppose that $g=\begin{pmatrix} 1 & 0 \\ 1 & 1 \end{pmatrix}$. 
Let $e_0={\rm ord}_v({\mathfrak N}_F)$. 
Note that 
\begin{align*}
 w_2= \begin{pmatrix} 0 & 1 \\ -1 & 0 \end{pmatrix} 
  =  \begin{pmatrix} -1 & 1+\varpi^{e_0}_v \\ 0 & -1 \end{pmatrix} 
       \begin{pmatrix} 1 & 0 \\ 1 & 1 \end{pmatrix} 
       \begin{pmatrix} 1+\varpi^{e_0}_v & -1 \\ -\varpi^{e_0}_v & 1 \end{pmatrix}.  
\end{align*}
Hence we may assume that $g=w_2$.
Then we find that 
\begin{align*}
  \Phi^{ (r) }_v \left( (0,1) \begin{pmatrix} 0 & 1 \\ -1 & -u \end{pmatrix} g  \right)
  = \Phi^{ (r) }_v \left( u,-1 \right). 
\end{align*}
Hence the condition (\ref{eq:UnitChar2}) proves 
          \begin{align*}
                \int_{ F_v} {\mathcal F}_v \left( w_2\begin{pmatrix} 1 & u \\ 0 & 1 \end{pmatrix}w_2 , s\right)  {\rm d}u
                =& \int_{ F_v}  {\rm d}u  \int_{ F^\times_v}  {\rm d}^\times t 
                      \Phi^{ (r) }_v \left( ut,-t \right)\phi_{1,v}\phi^{-1}_{2,v}(t) |t|^{2s+1}_v  \\
               =& \omega_{\pi, v}(-1) {\rm vol}( {\mathfrak N}_{F,v}, {\rm d}u).
          \end{align*}
This proves the third statement.

We prove the fourth statement. 
By a similar reason with the third statement, we find  that ${\mathcal F}_{v_0}$ 
is right ${\mathcal K}_0(\varpi_{v_0})$-invariant 
and it is supported on ${\rm B}_2(F_{v_0})w_2 {\mathcal K}_0(\varpi_{v_0})$.  
Let $g=w_2$, then we find that 
\begin{align*}
  {\mathcal F}_v(g,s) 
  =& \int_{F^\times_{v_0}}   \Phi^{(r)}_{v_0} (t, 0) \phi_{1,v_0}\phi^{-1}_{2,v_0}(t) |t|^{2s+1}_{v_0} {\rm d}^\times t 
  =1. 
\end{align*}
Hence we obtain the first identity. 

We prove the second identity.
It is proved that ${\mathcal M}_{v_0}{\mathcal F}_{v_0}$ is right ${\mathcal K}_0(\varpi_{v_0})$-invariant 
and it is supported on ${\rm B}_2(F_{v_0}) {\mathcal K}_0(\varpi_{v_0})$.  
For $g=1_2$, we find that
          \begin{align*}
                \int_{ F_{v_0}} {\mathcal F}_{v_0} \left( w_2\begin{pmatrix} 1 & u \\ 0 & 1 \end{pmatrix} , s\right)  {\rm d}u
                =& \int_{ F_{v_0}}  {\rm d}u  \int_{ F^\times_{v_0} }  {\rm d}^\times t 
                      \Phi^{ (r) }_{v_0} \left( t,ut \right)\phi_{1,{v_0}}\phi^{-1}_{2,{v_0}}(t) |t|^{2s+1}_v  
                 = 1.     
          \end{align*}
This prove the second identity.
\end{proof}

\subsubsection{Distribution property}

We prove the distribution property of Eisenstein series $\{ E_{r,s} \}_{r\geq 1}$ in the following Lemma \ref{lem:distsec} and Proposition \ref{prop:distKE2}.
In the construction of $p$-adic Asai $L$-functions, 
   the distribution property of Eisenstein series is used to prove the distribution property of certain partial zeta integrals for Asai $L$-functions
   in Proposition \ref{prop:dist}. 

For each function ${\mathcal F}$ on ${\rm GL}_2(F_{\mathbf A})$ and $h \in {\rm GL}_2(F_{\mathbf A})$, 
 define $\varrho(h) {\mathcal F}$ to be the right translation of ${\mathcal F}$ by $h$, that is, $\varrho(h) {\mathcal F}(g) = {\mathcal F}(gh)$ for each $g\in {\rm GL}_2(F_{\mathbf A})$.

\begin{lem}\label{lem:distsec}
{\itshape 
We have 
\begin{align*}
   \sum_{x, y \in {\mathcal O}_{F,p} / p {\mathcal O}_{F,p} } 
    \varrho\left(  \begin{pmatrix} 1+x p^r & 0 \\ y p^r & 1 \end{pmatrix} \right) {\mathcal F}_{\Phi^{(r+1)}}
    =  \omega_{\pi, p}( p )  \prod_{v\mid p} q^{-2(s+n-\alpha)}_v \times  {\mathcal F}_{\Phi^{(r)} }.   
\end{align*}
}
\end{lem}
\begin{proof}
Since no confusion likely occurs, we abbreviate ${\mathcal F}_\infty(g,s; \Phi^{(r)}_\infty, \phi_{1,\infty}, \phi_{2,\infty})$ to ${\mathcal F}_\infty(g)$ in this proof. 

Write $(m,n)=(0,1) a g$ and then we find that, for each $v\mid p$,   
\begin{align*}
\sum_{x, y \in {\mathcal O}_{F,v} / p {\mathcal O}_{F,v} } 
        \Phi^{(r+1)}_{ v}  \left( (m, n)  \begin{pmatrix} 1+x p^r & 0 \\ y p^r & 1 \end{pmatrix} \right) 
   =& \sum_{x, y  } 
         \Phi^{(r+1)}_{ v}  \left( (m+mx p^r + nyp^r, n)   \right)  \\
   =& \sum_{x, y  }           
          \psi_{F,v}\left( \frac{ (m+mx p^r +ny p^r )}{ p^{r+1} } \right) 
          {\mathbb I}_{\mathcal O^{\oplus 2}_{F,v}}(m+mx p^r +ny p^r, n )  \\
   =& \sum_{x, y  }           
          \psi_{F,v}\left( \frac{ m p^{-1} }{ p^r } \right)
          \psi_{F,v}\left( \frac{ mx }{ p } \right) 
          \psi_{F,v}\left( \frac{ ny}{ p } \right)
          {\mathbb I}_{\mathcal O^{\oplus 2}_{F,v}}(m(1+ x p^r ), n )  \\
   =& q^2_v 
          \psi_{F,v}\left( \frac{ m p^{-1}}{  p^r } \right)        
          {\mathbb I}_{ p \mathcal O^{\oplus 2}_{F,v}}(m, n )     \\
   =& q^2_v 
          \Phi^{(r)}_{v} ((0,1) a p^{-1} g).
\end{align*}
Let $\Phi^{(r)(p)}_{\rm fin} = \otimes_{v\in \Sigma_F, v\nmid p \infty} \Phi^{(r)}_v$.   
Note that 
\begin{align*}
  \Phi^{(r)(p)}_{\rm fin}(px, py) 
  = \Phi^{(r)(p)}_{\rm fin}(x, y)  \prod_{v\in \Sigma_F, v\nmid p \infty} \omega^{-1}_{\pi, v}(p)   
  =  \omega_{\pi,p}(p) \Phi^{(r)(p)}_{\rm fin}(x, y).
\end{align*}
Hence by changing variable $a$ to $a p$, we obtain 
\begin{align*}
  & \sum_{x, y \in {\mathcal O}_{F,p} / p {\mathcal O}_{F,p} }    
          \sum_{a \in F^\times}   {\mathcal F}_\infty(a g)  
           \Phi^{(r+1)}_{ {\rm fin}} \left( (0,1) a g \begin{pmatrix} 1+x p^r & 0 \\ y p^r & 1 \end{pmatrix}  \right) 
           \phi_{1,{\rm fin}}(\det a g) |\det a g |^{s+\frac{1}{2}}_{\rm fin}  \\
\stackrel{ a\mapsto ap }{=}
& \prod_{v\mid p} q^2_v  
  \times  
       \sum_{a \in F^\times}   {\mathcal F}_\infty(a p g)
        \Phi^{(r)}_{p}((0,1) a g) 
        \Phi^{(r)(p)}_{{\rm fin}}((0,1) a pg)    
            \phi_{1,{\rm fin}}(\det a p g) |\det a p g |^{s+\frac{1}{2}}_{\rm fin}  \\
=& \prod_{v\mid p} q^2_v 
      \times   \phi_{1,\infty}(p)\phi_{2,\infty}(p)  
      \times \omega_{\pi, p}(  p  ) 
      \times \phi^2_{1, {\rm fin}}(p) |p|^{2s+1}_{\rm fin}   \\
 & \quad   \times  
       \sum_{a \in F^\times}   {\mathcal F}_\infty(a g)
        \Phi^{(r)}_{ {\rm fin}}((0,1) a g)
            \phi_{1,{\rm fin}}(\det a g) |\det a g |^{s+\frac{1}{2}}_{\rm fin}  \\
=& \prod_{v\mid p} q^{1-2s}_v \times \phi^{-1}_{1,\infty}\phi_{2,\infty}(p)\omega_{\pi, p}( p )  
      \times  {\mathcal F}_{\Phi^{(r)}}(g).  
\end{align*}
Since $\phi^{-1}_{1, \infty}\phi_{2,\infty}=|\cdot|^{-2(n-\alpha)-1}_\infty$, we have $\phi^{-1}_{1, \infty}\phi_{2,\infty}(p)=\prod_v q^{-2(n-\alpha)-1}_v$. 
This proves the statement. 
\end{proof}

Let 
\begin{align}\label{eq:cars}
  c_{r, \alpha, s} =  \omega_{\pi, p}( p^{-r} )  \prod_{v\mid p} q^{2(s+n-\alpha)r}_v.  
\end{align}
The following proposition is an immediate corollary of Lemma \ref{lem:distsec}:

\begin{prop}\label{prop:distKE2}(Distribution property of Eisenstein series)
{\itshape 
We have  
\begin{align*}
   c_{r+1, \alpha, s} \sum_{x, y \in {\mathcal O}_{F,p} / p {\mathcal O}_{F,p} } 
    \varrho\left(  \begin{pmatrix} 1+x p^r & 0 \\ y p^r & 1 \end{pmatrix} \right) 
    E_{r+1,s}
   = c_{r, \alpha, s} E_{r,s}. 
\end{align*}
}
\end{prop}

\begin{rem}
Proposition \ref{prop:distKE2} is an analogue of \cite[Lemma 1.7 (2), Proposition 2.3]{ka04}.
\end{rem}

\section{Cohomology classes}\label{sec:Coh}

In this section, we describe automorphic forms in terms of cohomology groups of certain symmetric spaces. 
This is necessary for the description of the algebraicity of critical values of Asai $L$-functions, 
which is worked out by Ghate in \cite{gh99} and \cite{gh99b}. 
We will follow Ghate's descriptions in these papers. 
However Ghate uses a classical language depending on a coordinate of the Poincar$\acute{\rm e}$ upper $3$-space, 
which is not convenient for the further computation in the present paper. 
Hence we introduce a use of the adelic language in this section.  

The local systems which are necessary for the algebraic description of Asai $L$-functions 
are recalled in Section \ref{sec:locsys}. 
By using these local systems, we give definitions of differential forms and periods associated with automorphic forms on ${\rm GL}_2$ over CM fields
in Section \ref{sec:ESHmap}. 
Section \ref{sec:bdcoh} and \ref{sec:eiscohclass} are devoted to describe the rational properties of Eisenstein series, which are introduced in Section \ref{sec:Eis}.   
In Section \ref{sec:cup}, a certain cup product of differential forms is introduced, 
which will be necessary to give a cohomological interpretation of the critical values of Asai $L$-functions in Section \ref{sec:AsaiInt}.

\subsection{Local systems}\label{sec:locsys}

In this subsection, we recall certain local systems on symmetric spaces. 
We follow the descriptions in \cite[Section 2]{hi94} and \cite[Section 5]{gh99}, which are given in a language of the Poincar$\acute{\rm e}$ upper 3-space. 
However we prefer to use the adelic language for later use. 
(See also \cite[Section 3.7]{na17} for this adelic description.)

Let $E$ be a number field. 
For an open compact subgroup ${\mathcal K}\subset {\rm GL}_2(\widehat{\mathcal O}_E)$, put 
\begin{align*}
 Y^E_{\mathcal K}= {\rm GL}_2(E) \backslash {\rm GL}_2(E_\mathbf A)/C_{\infty,+}E^\times_\infty {\mathcal K}, 
\end{align*}
where $C_{\infty, +}  =  \prod_{\sigma\in \Sigma_E({\mathbf R})}  {\rm SO}_2({\mathbf R}) 
                                            \prod_{\sigma\in \Sigma_E({\mathbf C})}  {\rm SU}_2({\mathbf R})$.
We sometimes abbreviate $Y^E_{\mathcal K}$ to $Y_{\mathcal K}$. 
In this subsection, we recall definitions of local systems on $Y_{\mathcal K}$.

Define an action of ${\rm GL}_2(E)$ on $L(n ; {\mathbf C})$ by $\varrho_{n,m}$, where $\varrho_{n,m}$ is introduced in (\ref{eq:rhodfn})
and define an action of ${\mathcal K}$ on $L(n; {\mathbf C})$ to be the trivial action.
Then, a local system ${\mathcal L}(n ;{\mathbf C})$ on $Y_{\mathcal K}$ 
  is defined to be the sheaf of locally constant sections of the following map: 
\begin{align*}
     {\rm GL}_2(E) \backslash \left(  {\rm GL}_2(E_{\mathbf A})/C_{\infty,+}E^\times_\infty {\mathcal K}   \times L(n;{\mathbf C})  \right) 
     \stackrel{{\rm pr}_1}{\to} {\rm GL}_2(E) \backslash {\rm GL}_2(E_{\mathbf A})/C_{\infty,+}E^\times_\infty {\mathcal K},
\end{align*}
where ${\rm pr}_1$ is the projection to the first component and we give the discrete topology on $L(n;{\mathbf C})$.
For a Hecke character $\chi:E^\times \backslash E^\times_{\mathbf A} \to {\mathbf C}^\times$ of the infinity type $\chi_\infty(z)=z^{-\kappa}$
and an ideal ${\mathfrak N}\subset \widehat{\mathcal O}_E$ such that the conductor of $\chi$ divides ${\mathfrak N}$, 
we also introduce another local system ${\mathcal L}(n,\chi; A)$  on $Y_{{\mathcal K}_0({\mathfrak N})}$ as follows. 
For $u \in {\mathcal K}_0({\mathfrak N})$, choose $z_u\in \widehat{\mathcal O}^\times_E$ so that $z^{-1}_uu \in {\mathcal K}_1({\mathfrak N})$. 
Then define a right action of ${\mathcal K}_0({\mathfrak N})$ on $L(n;{\mathbf C})$ by 
\begin{align*}
   ( P\cdot u)(\begin{pmatrix} X \\ Y \end{pmatrix}) = \chi(z_u) P(\begin{pmatrix}  X \\ Y \end{pmatrix}).
\end{align*}
This definition does not depend on the choice of $z_u$, since the conductor of $\chi$ divides ${\mathfrak N}$. 
Then we define ${\mathcal L}(n, \chi; {\mathbf C})$ to be the local system of the locally constant sections of the following map: 
\begin{align*}
     {\rm GL}_2(E) \backslash \left( {\rm GL}_2(E_{\mathbf A})/C_{\infty,+}E^\times_\infty   \times L(n;{\mathbf C}) \right) / {\mathcal K}_0({\mathfrak N}) 
     \stackrel{{\rm pr}_1}{\to} {\rm GL}_2(E) \backslash {\rm GL}_2(E_{\mathbf A})/C_{\infty,+}E^\times_\infty {\mathcal K}_0({\mathfrak N}), 
\end{align*}
where we denote the projection to the first component by ${\rm pr}_1$. 
Note that if $\chi$ and $\chi^\prime$ are two Hecke characters $E^\times \backslash E^\times_{\mathbf A}\to {\mathbf C}^\times$ such that 
$\chi_\infty(z) = \chi^\prime_\infty(z)=z^{-\infty}$ and that $\chi|_{  \widehat{\mathcal O}^\times_E } = \chi^\prime|_{  \widehat{\mathcal O}^\times_E }$, 
then ${\mathcal L}(n, \chi; {\mathbf C})$ is isomorphic to ${\mathcal L}(n, \chi^\prime; {\mathbf C})$. 
Hence we write ${\mathcal L}(n, \chi_0; {\mathbf C})$ for ${\mathcal L}(n, \chi; {\mathbf C})$, where $\chi_0=\chi|_{  \widehat{\mathcal O}^\times_E } $.

Let $p$ be an odd prime and ${\mathcal O}_{E,p}=\prod_{v\mid p} {\mathcal O}_{E,v}$.  
Hereafter we always assume that 
\begin{itemize}
   \item $p>n$, that is, $p>{\rm max} \{  n_\sigma \mid \sigma \in \Sigma_{E,\infty} \}$.
\end{itemize}
We define another local system ${\mathcal L}(n ; A)$ on $Y_{\mathcal K}$ for an ${\mathcal O}_{E,p}$-subalgebra $A$ of ${\mathbf C}_p$  
in the following way.
For $u\in {\mathcal K}$ and $P\in L(n;A)$, we define
\begin{align}\label{eq:pavloc}
    P\cdot u =    \varrho_{n,m}(u^{-1}_p)  P, 
\end{align}
where the action of ${\mathcal K}$ on $L(n; A)$ is defined in (\ref{eq:pact}).
Then, we define a local system ${\mathcal L}(n ; A)$ on $Y_{\mathcal K}$ 
    to be the sheaf of locally constant sections of the following map: 
\begin{align*}
     \left(  {\rm GL}_2(E) \backslash{\rm GL}_2(E_{\mathbf A})/C_{\infty,+} E^\times_\infty  \times L(n;A) \right) / {\mathcal K} 
     \stackrel{{\rm pr}_1}{\to} {\rm GL}_2(E) \backslash {\rm GL}_2(E_{\mathbf A})/C_{\infty,+} E^\times_\infty {\mathcal K},
\end{align*}
where we give the discrete topology on $L(n;A)$.
In a similar way with ${\mathcal L}(n, \chi_0; {\mathbf C})$, 
  we define a local system ${\mathcal L}(n, \chi_0; A)$ on $Y_{{\mathcal K}_0({\mathfrak N})}$ by using the following right action of ${\mathcal K}_0({\mathfrak N})$: 
\begin{align}
    P\cdot u =  \chi(z_u) 
                        \varrho_{n,m}(u^{-1}_p)  P, \quad (u\in {\mathcal K}_0({\mathfrak N})).   
\end{align}

By the above arguments, we have two local systems 
${\mathcal L}(n; {\mathbf C})$ and ${\mathcal L}(n; {\mathbf C}_p)$ on $Y_{\mathcal K}$. 
The following proposition checks that these two are isomorphic to each other 
via fixed isomorphism ${\mathbf C} \cong {\mathbf C}_p$:

\begin{prop}\label{CCp}(\cite[Proposition 3.4]{na17})
{\itshape 
We identify $L(n;{\mathbf C})$ with $L(n; {\mathbf C}_p)$ via fixed isomorphism ${\mathbf i}_p:{\mathbf C}\to {\mathbf C}_p$.
Then, the following homomorphism 
\begin{align*}
{\mathcal L}(n;{\mathbf C}) \to {\mathcal L}(n; {\mathbf C}_p):&\ (g, P_g) 
   \mapsto \left(g,  \varrho_{n,m} (g^{-1}_p) \cdot ({\mathbf i}_p(P_g) \right) 
\end{align*} 
is an isomorphism.  In a similar way, ${\mathcal L}(n, \chi_0; {\mathbf C})$ is isomorphic to ${\mathcal L}(n, \chi_0;{\mathbf C}_p)$.
}
\end{prop}

\subsection{Eichler-Shimura-Harder maps}\label{sec:ESHmap}

Hereafter to the end of this paper, we fix a CM extension $E/F$, where $F$ is a totally real field. 
Let ${\mathcal K}$ be an open compact subgroup ${\rm GL}_2(\widehat{\mathcal O}_E)$.   
For the cohomological interpretation of Asai $L$-functions, 
     we need an explicit description of the cuspidal cohomology group $H^\ast_{\rm cusp}(Y^E_{\mathcal K}, {\mathcal L}(n; {\mathbf C}))$ 
     in terms of differential forms associated with cusp forms. 
This kind of descriptions is known as the Eichler-Shimura-Harder isomorphism, which is explicitly written down in \cite[Proposition 3.1]{hi94}. 
We briefly recall it to fix some notation for the later computations.

For a non-negative integer $n$, we define three polynomials $P_{-2}, P_0, P_2\in {\mathbf Z}[X, Y, X_c, Y_c, U, V]$ by the following identity:
   \begin{align*}
                      &\det\begin{pmatrix} X & U \\
                                Y & V \end{pmatrix}^n
                      \det\begin{pmatrix} Y_c  & U \\
                               -X_c & V  \end{pmatrix}^n
                       \det\begin{pmatrix}  A & U \\
                                 B & V\end{pmatrix}^2\\
                              &=P_{-2}(X, Y, X_c, Y_c, U, V)A^2
                                   +P_{0}(X, Y, X_c, Y_c, U, V)AB
                                                 +P_{2}(X, Y, X_c, Y_c, U, V)B^2. 
   \end{align*}
We note that $P_{-2}, P_0$ and $P_2$ are homogeneous polynomials of degree $n$ (resp. $n$, $2n+2$) with respect to $(X, Y)$ (resp. $(X_c, Y_c), (U, V)$).   
We define polynomials $v_i(-2), v_i(0)$ and $v_i(2)$ for $i=-n-1,\ldots, n,  n+1$ by the following identity:
\begin{align}\label{eq:v(j)def}
 \begin{aligned} 
       & (P_{-2}(X, Y, X_c, Y_c, U, V) ,P_{0}(X, Y, X_c, Y_c, U, V) ,P_2(X, Y, X_c, Y_c, U, V)) \\ 
     =&(V^{2n+2}, \ldots, ( U^{n+1-i}V^{n+1+i} )^\vee ,\ldots, U^{2n+2}) 
           \left( \begin{smallmatrix} 
                  v_{n+1}(-2) & v_{n+1}(0) & v_{n+1}(2)  \\
                  v_{n}(-2) & v_{n}(0) & v_{n}(2)  \\
                  v_{n-1}(-2) & v_{n-1}(0) & v_{n-1}(2)  \\
                                & \vdots & \\
                  v_{-n-1}(-2) & v_{-n-1}(0) & v_{-n-1}(2)  \\
           \end{smallmatrix}\right).
\end{aligned}
 \end{align}
Define elements $H_0, E_1$ and $E_2$ in $\mathfrak{sl}_2(E_\sigma)\otimes_{\mathbf R}{\mathbf C}$ for each infinite place $\sigma$ of $E$ by 
\begin{align*}
     H_0 &=\frac{1}{2} \begin{pmatrix} 1 & 0 \\ 0 & -1 \end{pmatrix}, \\
     E_1&=\frac{1}{4}  \left( \begin{pmatrix} 0 & 1 \\
                                                       1 & 0 \end{pmatrix} 
                                  + \begin{pmatrix} 0 & -\sqrt{-1} \\
                                                          \sqrt{-1} & 0 \end{pmatrix}  \otimes \sqrt{-1} \right) ,  \quad
     E_2	=\frac{1}{4}  \left( \begin{pmatrix} 0 & 1 \\
                                                       1 & 0 \end{pmatrix} 
                                  - \begin{pmatrix} 0 & -\sqrt{-1} \\
                                                          \sqrt{-1} & 0 \end{pmatrix} \otimes \sqrt{-1} \right).
\end{align*}
Let $r_E$ be the order of the set of archimedean places of $E$.   
For each $f\in S_\kappa({\mathcal K}_0({\mathfrak N}), \chi)$, 
 the following element defines 
an element in a cuspidal cohomology group $H^{ r_E }_{\rm cusp}(Y^E_{{\mathcal K}_0({\mathfrak N})}, {\mathcal L}(n, \chi_0; {\mathbf C}))$:
\begin{align} \label{eq:ESH}
\begin{aligned}
   \delta(f)(g)(X_{g}) 
=& \sum_{-n-1\leq  i  \leq n+1}  f^i( g )     
           \left(   \varrho_{n,m}(g_\infty ) \cdot v_{i}(-2) {\rm d}E_1(L^{-1}_{g} X_g)  \right. \\    
    & \quad \quad  \left.  - \varrho_{n,m}(g_\infty ) \cdot v_{i}(0) {\rm d}H_0(L^{-1}_{g} X_g)    
                                      - \varrho_{n,m}(g_\infty ) \cdot v_{i}(2) {\rm d}E_2(L^{-1}_{g} X_g) \right), 
\end{aligned}
\end{align} 
where we write 
\begin{align*}
   f(g) = \sum^{n+1}_{i=-n-1}  f^i(g) (S^{n+1-i}T^{n+1+i})^\vee.
\end{align*}

\begin{thm}(\cite[Section 3.6]{ha87}, \cite[Proposition 3.1]{hi94})
{\itshape 
The map $\delta$ in (\ref{eq:ESH}) gives a Hecke equivariant isomorphism:
\begin{align*}
  \delta: \begin{CD}  \oplus_{\chi^\prime} S_{\kappa}( {\mathcal K}_0({\mathfrak N}), \chi^\prime ) 
            @>\sim>> H^{ r_E}_{\rm cusp} ( Y^E_{ {\mathcal K}_0({\mathfrak N}) }, {\mathcal L}(n, \chi_0; {\mathbf C}) ),  \end{CD}    
\end{align*}
where $\chi^\prime:E^\times \backslash E^\times_{\mathbf A}\to {\mathbf C}^\times$ runs over 
the set of Hecke characters of the infinity type $\chi^\prime_\infty(z) = z^{-\kappa}$ and $\chi^\prime|_{ \widehat{\mathcal O}^\times_K } = \chi_0$.
}
\end{thm}

For later use, we briefly recall the definition of Hida's canonical period $\Omega_{\pi, p}$ in \cite[Section 8]{hi94}. 
(It is also briefly explained in \cite[Section 3]{na17}). 
In \cite[Section 2.1]{hi99}, it is showed that 
there exists a Hecke equivariant section:
\begin{align*}
 i: \begin{CD} H^{ r_E }_{\rm cusp}(Y^E_{ {\mathcal K}_0({\mathfrak N}) }, {\mathcal L}(n, \chi_0; {\mathbf C})) 
        @>>> H^{ r_E }_{\rm c}(Y^E_{ {\mathcal K}_0({\mathfrak N}) }, {\mathcal L}(n, \chi_0; {\mathbf C})),   \end{CD}  
\end{align*}
where $H^\ast_{\rm c}$ is the cohomology group with compact support.  
Then Proposition \ref{CCp} induces a map 
\begin{align*}
\begin{CD}
 H^{ r_E }_{\rm cusp}(Y^E_{ {\mathcal K}_0({\mathfrak N}) }, {\mathcal L}(n, \chi_0; {\mathbf C})) 
     @>>> H^{ r_E }_{\rm c}(Y^E_{ {\mathcal K}_0({\mathfrak N}) }, {\mathcal L}(n, \chi_0; {\mathbf C}_p)), 
\end{CD}
\end{align*}
which we denote by $i$ again. 
Fix a finite extension $K_\pi$ of ${\mathbf Q}_p$ so that $K_\pi$ contains all conjugates of $E$, all Hecke eigenvalues of $\pi$ and the values of $\omega_\pi$. 
Denote by ${\mathcal O}_\pi$ the ring of integers of $K_\pi$.
The natural embedding ${\mathcal O}_\pi \to {\mathbf C}_p$ induces a map 
\begin{align*}
 j: \begin{CD} H^{ r_E }_{\rm c}(Y^E_{ {\mathcal K}_0({\mathfrak N}) }, {\mathcal L}(n, \chi_0; {\mathcal O}_\pi   )) 
                      @>>> H^{ r_E }_{\rm c}(Y^E_{ {\mathcal K}_0({\mathfrak N}) }, {\mathcal L}(n, \chi_0; {\mathbf C}_p )).   \end{CD}  
\end{align*}
We define $\eta_f$ to be a generator of 
    $\langle i\circ \delta(f)  \rangle_{ {\mathbf C}_p }
       \cap j(  H^{r_E}_{\rm c}(Y^E_{ {\mathcal K}_0({\mathfrak N}) }, {\mathcal L}(n, \chi_0; {\mathcal O}_\pi  ) )  )$
as an ${\mathcal O}_\pi$-module.
Then define Hida's canonical period $\Omega_{\pi,p} \in {\mathbf C}^\times_p$ so that 
\begin{align}
  \Omega_{\pi,p} \eta_f = i \circ \delta(f).     \label{eq:CanPer}
\end{align}

\begin{rem}\label{rem:per}
\begin{enumerate}
\item Hida's canonical period $\Omega_{\pi,p}$ is  unique up to multiplications of elements in ${\mathcal O}^\times_\pi$.
\item Instead of the use of ${\mathcal O}_\pi$, smaller coefficient rings work well in a similar way. 
         In fact let $K^\prime_\pi$ be the finite extension of $E$ such that    
         $K^\prime_\pi$ contains all conjugates of $E$, all Hecke eigenvalues of $\pi$ and the values of $\omega_\pi$. 
        Denote by ${\mathcal O}^\prime_{\pi,p}$ the localization of the ring of integers of $K^\prime_\pi$ 
             at the prime corresponding to the fixed embedding $\overline{\mathbf Q} \to {\mathbf C}_p$.
         Then we can define $\Omega_{\pi, p}$ in the same way as above. 
         This coefficient ring ${\mathcal O}^\prime_{\pi,p}$ is used in \cite{gh99} and \cite{gh99b} 
         for the algebraicity of  the critical values of Asai $L$-functions. 
         In a similar way, we can deduce the algebraicity of the critical values of Asai $L$-functions twisted by finite-order characters.  
         However we will concentrate to use ${\mathcal O}_\pi$, 
         since the coefficient ring must be $p$-adically complete for the construction of $p$-adic $L$-functions.
\end{enumerate}
\end{rem}

For later use, we introduce a right translation of $\delta(f)$ and $\eta_f$ by a certain matrix $h^{(r)}_\xi$,     
    which is necessary to introduce an integral expression of partial Asai $L$-functions.    
Let $v_0$ be a fixed auxiliary prime, 
${\mathcal K}^\prime={\mathcal K}_0({\mathfrak N}\varpi_{v_0}) \cap {\mathcal K}_0(p)$
${\mathcal K}= {\mathcal K}^\prime \cap {\mathcal K}(p^r)$
and write $Y_{\mathcal K}$ as $Y_{\mathcal K}(p^r)$.
Let $\xi\in E$ such that 
\begin{itemize}
 \item ${\rm Tr}_{E/F}(\xi)=0$.
 \item ${\mathcal O}_{E, p}=\langle 1, \xi \rangle_{{\mathcal O}_{F,p}}$.  
\end{itemize}
Note that we can choose such $\xi$ assuming $p\nmid 2D_{E/F}$.
Let $r=\sum_{v\mid p}r_v v \in {\mathbf Z}[\Sigma_{F,p}]$ with $r_{v_1}=r_{v_2}>0$ for each $v_1, v_2\mid p$.
Put $h^{(r)}_{\xi} = \begin{pmatrix} p^r &  \xi  \\ 0 & 1 \end{pmatrix}\in {\rm GL}_2(E_p)$. 
Since we have 
\begin{align*}
   \left( h^{(r)}_{\xi} \right)^{-1}  u   h^{(r)}_{\xi} \in  {\mathcal K}^\prime 
\end{align*}
for each $u \in {\mathcal K}(p^r)$, 
 the following map between local systems on $Y^E_{{\mathcal K}^\prime }$  
 and $Y^E_{\mathcal K}(p^r)$ is well-defined:
\begin{align*}
  \varrho(h^{(r)}_{\xi} ): {\mathcal L}(n, \chi_0; {\mathbf C}_p)_{/Y^E_{{\mathcal K}^\prime }} 
  \to {\mathcal L}(n, \chi_0; {\mathbf C}_p)_{/ Y^E_{\mathcal K}(p^r)}; 
   (g, P_g) \to (g  ,  p^{-mr}  \varrho^{(p)}_{n,m}(h^{(r)}_{\xi})  P_{g h^{(r)}_{\xi}}).  
\end{align*}
Since $p^{-mr}  \varrho^{(p)}_{n,m}(h^{(r)}_{\xi}) P\in L(n; {\mathcal O}_\pi )$ for each $P\in L(n; {\mathcal O}_\pi )$, 
the above map also defines 
\begin{align*}
   {\mathcal L} (n, \chi_0; {\mathcal O}_\pi )_{/Y^E_{{\mathcal K}^\prime }} 
    \to {\mathcal L}(n, \chi_0; {\mathcal O}_\pi )_{/ Y^E_{\mathcal K}(p^r)},
\end{align*}
which we also denote by $\varrho(h^{(r)}_{\xi} )$.
Define an element $\varrho(h^{(r)}_{\xi} ) \delta(f)$ (resp. $\varrho(h^{(r)}_{\xi} ) \eta_f$)  
         in $H^{ r_E }_{\rm cusp}(Y^E_{\mathcal K}(p^r), {\mathcal L}(n, \chi_0; {\mathbf C}))$
(resp. $H^{ r_E }_{\rm cusp}(Y^E_{\mathcal K}(p^r), {\mathcal L}(n, \chi_0; {\mathcal O}_\pi))$)
to be the image of $\delta(f)$ (resp. $\eta_f$) via $\varrho(h^{(r)}_{\xi} )$.
The following gives an explicit form of $\varrho(h^{(r)}_{\xi} ) \eta_f$:
          \begin{align} \label{eq:etaexp}
          \begin{aligned}
            \varrho( h^{(r)}_{\xi} )\eta_f(g)(X_{g})  
          =& \frac{p^{-mr}}{ \Omega_{\pi,p}}  
                \sum_{-n-1\leq  i  \leq n+1}  f^i( g h^{(r)}_{\xi} )     
                  \left(  \varrho_{n,m}(g_\infty g^{-1}_p) \cdot v_i(-2) {\rm d}E_1(L^{-1}_g X_g)   \right.   \\ 
                &  \quad \quad   \left.    - \varrho_{n,m}(g_\infty g^{-1}_p) \cdot v_i(0) {\rm d}H_0(L^{-1}_g X_g)    
                                - \varrho_{n,m}(g_\infty g^{-1}_p) \cdot v_i(2) {\rm d}E_2(L^{-1}_g X_g) \right).
           \end{aligned}    
            \end{align}

\subsection{Boundary cohomology}\label{sec:bdcoh}

We briefly describe a certain class in the boundary cohomology group of $Y^F_{\mathcal K}$. 
To begin with, we recall the description of the boundary cohomology groups.
The basic reference of this subsection is \cite[Section II, IV]{ha87},  
which we recall  in an explicit manner by using notations of previous subsections.

Let ${\mathcal K}$ be an open compact subgroup of ${\rm GL}_2(\widehat{\mathcal O}_F)$. 
We denote by $Y^{\rm BS}_{\mathcal K}$ the Borel-Serre compactification of $Y_{\mathcal K}:=Y^F_{\mathcal K}$
and denote by $\partial Y^{\rm BS}_{\mathcal K}$ the boundary of $Y^{\rm BS}_{\mathcal K}$. 
In this subsection, we describe the boundary cohomology groups of $r_F$-th degree: 
\begin{align*}
     H^{r_F}(\partial Y^{\rm BS}_{\mathcal K}, {\mathcal L}(n, \chi_0; {\mathbf C} ) )
\end{align*}
to discuss the rationality of Eisenstein series, which we studied in Section \ref{sec:Eis}. 
Since Eisenstein series are defined by using Hecke characters, 
   we recall a notion of distinguished infinity types of Hecke characters according to \cite[Section 2.3]{ha87}: 

\begin{dfn}\label{dfn:deg}(See \cite[(2.3.4), (2.8.2)]{ha87})
Let $\sigma$ be an infinite place of $F$.
For $\phi_{i, \sigma}:F^\times_\sigma\to  {\mathbf C}^\times (i=1,2)$, we put $\phi_\sigma=\{\phi_{1, \sigma}, \phi_{2, \sigma}\}$. 
We define $\phi_\sigma$ to be of degree ${\rm deg}(\phi_\sigma)=$0 (resp. 1),
 if 
\begin{align*}
& \phi_{1,\sigma}(x) = x^{-n_\sigma-m_\sigma-\frac{1}{2}}, \quad    
  \phi_{2,\sigma}(x) = x^{ -m_\sigma +\frac{1}{2} },    \\
& \quad  ({\rm resp.}\  \phi_{1,\sigma}(x) = x^{\frac{-n_\sigma-2m_\sigma}{2}+\frac{k_\sigma-1}{2}}  = x^{-m_\sigma+\frac{1}{2}} , \quad    
                                 \phi_{2,\sigma}(x) = x^{\frac{-n_\sigma-2m_\sigma}{2}-\frac{k_\sigma-1}{2}}= x^{-n_\sigma-m_\sigma-\frac{1}{2}}). 
\end{align*}   
for each $x\in F^\times_{\sigma, +}$.
 We put
 \begin{align*}
   {\rm deg}(\phi) = \sum_\sigma  {\rm deg} (\phi_\sigma).  
 \end{align*}
We define ${\rm Coh}_{n,m}$ to be the set of pairs $\phi=\{\phi_1, \phi_2\}$ of Hecke characters of ${\rm deg}(\phi_\sigma)=0$ or $1$ for each $\sigma\in I_F$. 
\end{dfn}

\begin{rem}
\begin{enumerate}
\item For the reader's convenience, we make a brief explanation for the difference of conventions in \cite{ha87} and Definition \ref{dfn:deg}. 
We apply $d(\sigma), \nu(\sigma)$ in \cite[page 45]{ha87} for $n_\sigma, m_{\sigma}$. 
Note that the type $\gamma$ of Hecke character $\chi$ in the sense of \cite[(2.5)]{ha87} coincides with  
the infinity type $\chi^{-1}_\infty$ of $\chi$.
We also note that the terminology of the inducued Harish-Chandra module (\cite[page 67]{ha87}) is differ from ours. 
(Compare the last line in \cite[page 67]{ha87} with  (\ref{eq:ind}).)
Hence we find that 
$\phi=\{\phi_1, \phi_2\}$ is in Coh(${\mathscr L}$) if and only if 
$(\phi_{1,\sigma}(x_1)\phi_{2,\sigma}(x_2) |\frac{x_1}{x_2}|^{\frac{1}{2}} )^{-1}$ 
is the character given in \cite[(2.3.4)]{ha87}. 
Explicitly, $\phi=\{\phi_1, \phi_2\}$ has ${\rm deg}(\phi)=0$ or $1$ if and only if 
\begin{align*}
    \phi_{1,\sigma}(x_1)\phi_{2,\sigma}(x_2) |\frac{x_1}{x_2}|^{\frac{1}{2}} 
      =\begin{cases} x^{-n_\sigma}_1(x_1x_2)^{-m_\sigma} = x^{-n_\sigma-m_\sigma}_1 x^{-m_\sigma}_2,  &   ({\rm deg}(\phi)=0),  \\
                               x^{-n_\sigma-m_\sigma-1}_2 x^{-m_\sigma+1}_1 = x^{-m_\sigma+1}_1 x^{-n_\sigma - m_\sigma-1}_2, &  ({\rm deg}(\phi)=1).   \end{cases}
\end{align*}
This checks  that the convention in Definition \ref{dfn:deg} is the same with \cite[Section 2.3]{ha87}. 
\item Let $r_F$ be the order of the set of archimedean places of $F$.  
         Then the pair of Hecke characters $\phi_1$ and $\phi_2$, which we introduced in Section \ref{sec:choice}, is of  
         ${\rm deg}(\{\phi_1, \phi_2\}) = r_F$. 
\end{enumerate}
\end{rem}

For $\phi\in {\rm Coh}_{n,m}$, we put $\phi_0=\phi_1\phi_2|_{ {\mathcal O}^\times_{F} } $. 
Let ${\mathfrak M}\subset \widehat{\mathcal O}_F$ be an ideal such that the conductor of $\phi_0$ divides ${\mathfrak M}$.  
Consider an open compact subgroup ${\mathcal K}$ of ${\rm GL}_2(\widehat{\mathcal O}_F)$ such that 
${\mathcal K}_v= {\mathcal K}_0({\mathfrak M})_v$ for each place $v\mid {\mathfrak M}$.      
Let $C^\infty( {\rm B}_2(F)\backslash {\rm GL}_2(F_{\mathbf A})  )$ be the set of $C^\infty$-functions 
on ${\rm GL}_2(F_{\mathbf A})$ which are invariant under the left translation by ${\rm B}_2(F)$ and 
\begin{align*}
   C^\infty( {\rm B}_2(F)\backslash {\rm GL}_2(F_{\mathbf A})  )(\phi_0) 
   = \left\{  {\mathcal F} \in C^\infty( {\rm B}_2(F)\backslash {\rm GL}_2(F_{\mathbf A})  )
                \middle|    {\mathcal F}(gu) = \phi_0(u) {\mathcal F}(g), \ (u\in {\mathcal K})    \right\}. 
\end{align*}
Combining \cite[Section 2.1]{ha87} with \cite[VII, 2.2 Theorem, 2.7 Corollary]{bw80}, we find that 
\begin{align} \label{eq:boundary}
             H^\ast(\partial Y^{\rm BS}_{\mathcal K}, {\mathcal L}(n, \phi_0; {\mathbf C} ) ) 
  \cong  H^\ast({\mathfrak g}_\infty, K_\infty, C^\infty( {\rm B}_2(F)\backslash {\rm GL}_2(F_{\mathbf A})  )(\phi_0) \otimes  L(n, \chi_0;  {\mathbf C} )  ), 
\end{align}
where we put 
   ${\mathfrak g}_\infty = \prod_{\sigma \in I_F } \mathfrak{gl}(F_\sigma)_{\mathbf C}$,  
   and $K_\infty =  \prod_{\sigma \in I_F } F^\times_\sigma {\rm SO}_2( F_\sigma )$.
We introduce certain elements in the right-hand side of (\ref{eq:boundary}) in an explicit manner. 

Define 
\begin{align*}
 V_{\phi, s} 
    =\left\{  {\mathcal F}: {\rm GL}_2(F_{\mathbf A}) \times {\mathbf C} \to {\mathbf C} 
       \ \middle| \
       \substack{   {\mathcal F} \left(   \left(  \begin{smallmatrix} a & b \\ 0 & d \end{smallmatrix} \right) g u, s \right) 
           = \phi_1(a) \phi_2(d) |\frac{a}{d}|^{s+\frac{1}{2}}_{\mathbf A} {\mathcal F}(g, s), (u\in {\mathcal K}_1({\mathfrak N})).   \\
              \text{ ${\mathcal F}$ is $K_\sigma$-finite for each $\sigma\in I_{F}$.}
             } 
            \right\}.   
\end{align*}
We decompose $V_{\phi, s} = V_{\phi, s, \infty} \otimes V_{\phi, s, {\rm fin}}$ 
 and put $V_\phi = V_{\phi, 0}$.  
Since we have an embedding 
\begin{align*}
   I_\phi: V_{\phi, s} \to C^\infty( {\rm B}_2 (F)\backslash {\rm GL}_2(F_{\mathbf A}))(\phi_0) 
\end{align*}
for each $s\in {\mathbf C}$, 
we obtain a map (see also a discussion in \cite[Section 4.2, page 79]{ha87})
\begin{align*}
        H^\ast({\mathfrak g}_\infty, K_\infty, V_{\phi, \infty} \otimes L(n; {\mathbf C})  ) \otimes V_{\phi, {\rm fin}} 
  \to H^\ast({\mathfrak g}_\infty, K_\infty, C^\infty( {\rm B}_2 (F)\backslash {\rm GL}_2(F_{\mathbf A}))(\phi_0) \otimes L(n; {\mathbf C})  ).
\end{align*}
By the K\"unneth formula, we decompose  
\begin{align*}
   H^\ast({\mathfrak g}_\infty, K_\infty, V_{\phi, \infty} \otimes L(n; {\mathbf C})  )
 = \otimes_{\sigma\in I_F} H^\ast({\mathfrak g}_\sigma, K_\sigma, V_{\phi, \sigma} \otimes L(n_\sigma; {\mathbf C})  ). 
\end{align*} 
We put ${\mathscr B}_\sigma = V_{\phi, \sigma} \otimes L(n_\sigma; {\mathbf C})$ for the simplicity.
Hence, to give an element in the left-hand side of (\ref{eq:boundary}), 
   it suffices to give an explicit element in $H^\ast({\mathfrak g}_\sigma, K_\sigma,  {\mathscr B}_\sigma )$.   
Since we are interested in the Eisenstein cohomology class which has degree $r_F$, 
   we concentrate on $H^1({\mathfrak g}_\sigma, K_\sigma,  {\mathscr B}_\sigma )$. 
The infinity type of $\phi$, which we fixed in Definition \ref{dfn:deg} implies that   
 the dimension of $H^1({\mathfrak g}_\sigma, K_\sigma,  {\mathscr B}_\sigma )$ over ${\mathbf C}$ is $1$ (see \cite[page 69]{ha87}).

We introduce an element in $H^\ast({\mathfrak g}_\sigma, K_\sigma,  {\mathscr B}_\sigma )$. 
Define ${\mathcal P}_\sigma= \left\{X\in {\mathfrak g}_\sigma \middle| \theta(X)=-X \right\}$, where $\theta$ is the Cartan involution.
Then, by using \cite[II. Proposition 3.1]{bw80}, as explained in \cite[Section 1.5]{hi93}, we have 
\begin{align*}
H^\ast({\mathfrak g}_\sigma, K_\sigma, {\mathscr B}_\sigma) = {\rm Hom}_{ C_{\sigma,+}}(\wedge^\ast {\mathcal P}_{\sigma, {\mathbf C}}, {\mathscr B}_\sigma), 
\end{align*} 
where we put ${\mathcal P}_{\sigma, {\mathbf C}}={\mathcal P}_\sigma\otimes_{\mathbf R} {\mathbf C}$.
We introduce a distinguished element in ${\rm Hom}_{ C_{\sigma,+}}(\wedge^\ast {\mathcal P}_{\sigma, {\mathbf C}}, {\mathscr B}_\sigma)$.
Define elements $H, E$ in $\mathfrak{sl}_2(F_\sigma)_{\mathbf C}$ for each infinite place $\sigma$ of $F$ to be 
\begin{align*}
     H &=\frac{1}{2} \begin{pmatrix} 1 & 0 \\ 0 & -1 \end{pmatrix}, \quad E=\frac{1}{2}\begin{pmatrix} 0 & 1 \\ 1 & 0  \end{pmatrix}. 
\end{align*}
Let $\sigma \in I_F$ and assume that ${\rm deg}(\phi_\sigma)=1$.  
Recall that $\Phi_\sigma \in {\mathcal S}(F^{\oplus 2}_\sigma)$ is defined to be 
\begin{align*}
\Phi_\sigma(x,y) = 2^{-  (n_{\sigma}+2)} (x + \sqrt{-1}y)^{n_{\sigma}+2} e^{-\pi(x^2+y^2)}
\end{align*}
in Defintion \ref{dfn;BSfn}.  
  Then, by using the notation in (\ref{eq:FDef}), let  
          \begin{align*}
                 \omega(g,s; \phi_{1, \sigma}, \phi_{2, \sigma}, \Phi_\sigma) 
            = &   {\mathcal F}_{\sigma}(g,s; \phi_{1,\sigma}, \phi_{2,\sigma},\Phi_\sigma)   \varrho_{n,m}(g)(X - \sqrt{-1}Y)^{n_\sigma}  (\frac{1}{2} {\rm d}E - \frac{\sqrt{-1}}{2}{\rm d}H),
          \end{align*}
which is also introduced in \cite[page 80, (4.2.2)]{ha87}.
Then $\omega(g,0; \phi_{1, \sigma}, \phi_{2, \sigma}, \Phi_\sigma)$
gives an element in ${\rm Hom}_{ C_{\sigma,+}}( {\mathcal P}_{\sigma, {\mathbf C}}, {\mathscr B}_\sigma)$.
Put $\Phi_\infty=\prod_{\sigma \in I_F } \Phi_\sigma$ 
and define 
\begin{align*}
  \omega(g,s; \phi_{1, \infty}, \phi_{2, \infty}, \Phi_\infty) 
:= &    \wedge_{\sigma \in  I_F  }   
 \omega(g,s; \phi_{1, \sigma}, \phi_{2, \sigma}, \Phi_\sigma),  
\end{align*}
which gives an element in $H^{r_F}({\mathfrak g}_\infty, K_\infty, V_{\phi, \infty} \otimes L(n; {\mathbf C})  )$. 
By using elements $\omega(g,s; \phi_{1, \infty}, \phi_{2, \infty}, \Phi_\infty)$, we will explicitly construct 
elements in (\ref{eq:boundary}) in Section \ref{sec:eiscohclass}, where we call the elements Eisenstein cohomology classes.

\begin{rem}
In terms of the de Rham cohomology, we have 
\begin{align*}
  \omega (g,s; \phi_{1, \sigma}, \phi_{2, \sigma}, \Phi_\sigma )(X_g) 
=&         {\mathcal F}_{\sigma}(g,s; \phi_{1,\sigma}, \phi_{2,\sigma}, \Phi_\sigma) \varrho_{n,m}(g) (X -  \sqrt{-1}Y)^{n_\sigma} 
                 (\frac{1}{2} {\rm d}E - \frac{\sqrt{-1}}{2}{\rm d}H)(L^{-1}_g X_g).  
\end{align*}
where $\sigma \in I_F$ and $X_g\in T_g Y_{\mathcal K}$. 
By using a coordinate 
     $g=\begin{pmatrix} y^{\frac{1}{2}} & y^{-\frac{1}{2}} x  \\ 0 & y^{-\frac{1}{2}} \end{pmatrix} 
            \begin{pmatrix} \cos\theta & -\sin\theta \\ \sin\theta & \cos\theta  \end{pmatrix} 
               \in {\rm SL}_2({\mathbf R})$, 
we also recall the following formula:  
\begin{align}\label{eq:coord}
        (\frac{1}{2}{\rm d}E - \frac{\sqrt{-1}}{2}{\rm d}H)\circ L^{-1}_g 
   =  e^{2\sqrt{-1}\theta} \frac{{\rm d} z}{y}.   
\end{align}
\end{rem}

\subsection{Eisenstein cohomology classes}\label{sec:eiscohclass}

In this subsection, we introduce a cohomology classes which are associated with Eisenstein series. 
These are so called Eisenstein cohomology classes.
We prove the rationality of Eisenstein cohomology classes which are determined by the distinguished Bruhat-Schwartz function $\Phi^{(r)}$ 
according to Harder's description in \cite[IV]{ha87}. 
Furthermore, we discuss denominators of Eisenstein cohomology classes. 
This will produce us  cohomology classes in integral coefficients which are suitable for our construction of $p$-adic Asai $L$-functions.

We use the same notation as in Section \ref{sec:bdcoh}.
Hereafter, we always fix Hecke characters $\phi_1, \phi_2: F^\times \backslash F^\times_{\mathbf A} \to {\mathbf C}^\times$ satisfying the following conditions in Section \ref{sec:choice}:
\begin{align*}
   \phi_1 = |\cdot|^{-(\alpha+2m)+\frac{1}{2}}_{\mathbf A}, \quad 
   \phi_{2, \infty} =  |\cdot |^{-2(n-\alpha)  - (\alpha+2m) -\frac{1}{2}}_\infty, \quad 
   \phi_1\phi_2 =  |\cdot|^{ -n_\alpha -2(\alpha+2m) }_{\mathbf A} \omega^{-1}_\pi|_{F^\times_{\mathbf A}}.
\end{align*}
Note that ${\rm deg}(\phi) = r_F$ in the sense of Definition \ref{dfn:deg}.     
Recall that we put $k_\alpha$ to be $2(n-\alpha +t)$ in Section \ref{sec:choice}.  

We introduce a distinguished section ${\mathcal F}^0_\infty(-,s; \phi): {\rm GL}_2(F_\infty) \to {\mathbf C}$ to introduce a rational structure on the boundary cohomology groups.
For each $s\in {\mathbf C}$, define ${\mathcal F}^{0}_\sigma(-, s;\phi ): {\rm GL}_2(F_\sigma) \to {\mathbf C}$ by 
\begin{align*}
    {\mathcal F}^{0}_\sigma\left(  \begin{pmatrix} a & b \\ 0 & d  \end{pmatrix} \begin{pmatrix} \cos\theta & \sin \theta \\ -\sin \theta & \cos\theta \end{pmatrix}, s ;\phi   \right)  
    = \phi_1(a) \phi_2(d) | \frac{a}{d} |^{s+\frac{1}{2}}  e^{  \sqrt{-1}k_{\alpha, \sigma} \theta }.
\end{align*}
Note that 
  Lemma \ref{lem:FVal} shows the following identity:  
  \begin{align*}
        \omega_{\sigma}(g, s;\phi_\sigma, \Phi_\sigma) 
=&    2^{s-1} \sqrt{-1}^{ k_{\alpha, \sigma} }   \times \Gamma_{\mathbf C}(s+k_{\alpha, \sigma})
      \times     {\mathcal F}^{0}_{\sigma}(g, s;\phi ) \varrho_{n_\alpha, \alpha+2m}(g)(X -  \sqrt{-1}Y)^{n_\sigma} (\frac{1}{2} {\rm d}E - \frac{\sqrt{-1}}{2}{\rm d}H),  
  \end{align*} 
  where $\Phi_\sigma$ is as in Definition \ref{dfn;BSfn}.
For each ${\mathbf Q}$-subalgebra $R$ of ${\mathbf C}$, 
define 
\begin{align*}
 V_{\phi, s, {\rm fin}, R} 
    =\left\{  {\mathcal F}(-,s;\phi ) \in  V_{\phi, s, {\rm fin}} 
       \  \middle| \   {\mathcal F}(-,0;\phi ) \text{ takes values in } R \right\}. 
\end{align*}
We also define $V_{\phi, s, R}$ to be the space which is generated by the following elements over $R$:
\begin{align*}
    {\mathcal F}^0_\infty \otimes {\mathcal F}_{\phi, {\rm fin}}, 
  \quad ( {\mathcal F}^0_\infty :=\otimes_{\sigma\in \Sigma_F} {\mathcal F}^0_\sigma,\ 
              {\mathcal F}_{\phi, {\rm fin}} \in V_{\phi, s, {\rm fin}, R }).
\end{align*} 
For each ideal ${\mathfrak M}\subset \widehat{\mathcal O}_F$, 
     let $V^{ {\mathcal K}_1({\mathfrak M})}_{\phi, s, R}$ be the ${\mathcal K}_1({\mathfrak M})$-invariant subspace of $V_{\phi, s, R}$
and define 
$
{\mathscr B}^{\mathcal K_1({\mathfrak M})}_{\phi, R} 
= V^{ {\mathcal K}_1({\mathfrak M})}_{\phi, 0, R} \otimes L(n_\alpha;  R )$. 
Then ${\mathscr B}^{\mathcal K_1({\mathfrak M})}_{\phi, R}$ defines an $R$-structure on 
    $H^{ \ast }(\partial Y^{\rm BS}_{{\mathcal K}_0({\mathfrak M})}, {\mathcal L}(n_\alpha, \phi_0; {\mathbf C}  ))$ 
     (see \cite[Section 2.7]{ha87}).
More explicitly, the $R$-structure on  $H^{ r_F }(\partial Y^{\rm BS}_{{\mathcal K}_0({\mathfrak M})}, {\mathcal L}(n_\alpha, \phi_0; {\mathbf C}  ))$
is given by the space which is generated by the elements of the shape $\omega(g, 0; {\mathcal F}^0_\infty\otimes {\mathcal F}_{\phi, {\rm fin}}  )$ over $R$, where
\begin{align*}
 \omega(g, s; {\mathcal F}^0_\infty\otimes {\mathcal F}_{\phi, {\rm fin}} )  
 := \wedge_{\sigma \in I_F} {\mathcal F}^0_{\sigma}(g_\sigma, s;\phi_\sigma ) 
      \varrho_{n_\alpha, \alpha+2m}(g)(X_\sigma -  \sqrt{-1}Y_\sigma)^{n_\sigma} (\frac{1}{2} {\rm d}E_\sigma - \frac{\sqrt{-1}}{2}{\rm d}H_\sigma)
     \otimes {\mathcal F}_{\phi, {\rm fin}}(g_{\rm fin},s  ).
\end{align*}

To prove the rationality of Eisenstein cohomology classes, we will study its constant term according to the strategy in \cite[IV]{ha87}.
Define ${\rm r}: H^{ \ast }(Y_{\mathcal K_0({\mathfrak M})}, {\mathcal L}(n_\alpha, \phi_0; {\mathbf C}  )) 
                      \to H^{ \ast }(\partial Y^{\rm BS}_{\mathcal K_0({\mathfrak M})}, {\mathcal L}(n_\alpha, \phi_0; {\mathbf C}  ))$ by 
\begin{align*}
  {\rm r}(\omega)(g) = \int_{[{\rm N}_2]}  \omega( n(x) g)  {\rm d}x, \quad \left( n(x) = \begin{pmatrix} 1 & x \\ 0 & 1 \end{pmatrix}\right).
\end{align*}
Then the Bruhat decomposition of ${\rm GL}_2$ shows that 
\begin{align}\label{eq:mobius}
{\rm r}(\omega(-,-;\mathcal{F})  )(g,s) = \omega(g,s; {\mathcal F}) + \omega(g,-s; \mathcal{MF}).
\end{align}

For the distinguished section ${\mathcal F}(g, s; {\mathcal D}_r)$, which is introduced in Section \ref{sec:choice}, 
define $\delta(E_{r,s})$ to be 
\begin{align*}
   \delta(E_{r,s})(g) =& \sum_{\gamma \in {\rm B}_2(F)\backslash {\rm GL}_2(F)} \omega(\gamma g,s; {\mathcal F}(-,-; {\mathcal D}_r))    \\
   =& E_{r,s}(g) \wedge_{\sigma \in I_F} 
                  \varrho_{n_\alpha, \alpha+2m}(g)(X_\sigma -  \sqrt{-1}Y_\sigma)^{2(n-\alpha)}  
                   (\frac{1}{2} {\rm d}E_\sigma - \frac{\sqrt{-1}}{2}{\rm d}H_\sigma).   
\end{align*}
For ${\mathcal K} = {\mathcal K}_0 ({\mathfrak N}_F\varpi_{v_0}) \cap {\mathcal K}(p^r)  \subset {\rm GL}_2(\widehat{\mathcal O}_F)$,   
    $\delta(E_{r,0})$ gives an element in $H^{ r_F }(Y_{\mathcal K}(p^r), {\mathcal L}(n_\alpha, \phi_0; {\mathbf C}  ))$. 
We call this class an Eisenstein cohomology class. 
To discuss the rationality of this Eisenstein cohomology classes, 
   we introduce a normalized test functions as follows:
\begin{align*}
\widetilde{\mathcal F}_v(g, s; {\mathcal D}_r)
= \frac{1}{L(2s+1, \phi_{1,v} \phi^{-1}_{2,v} )} {\mathcal F}_v(g, s; {\mathcal D}_r), 
\end{align*}
for each $v\in \Sigma_F, v<\infty$.
Then Lemma \ref{lem:FVal} shows that 
$\widetilde{\mathcal F}_{\rm fin}(g, s; {\mathcal D}_r) := \otimes_{v<\infty}\widetilde{\mathcal F}_v(g_v, s; {\mathcal D}_r)$
gives an element in $ V^{ {\mathcal K}_1({\mathfrak N}_F\varpi_{v_0})\cap {\mathcal K}(p^r) }_{\phi, s, {\rm fin}, \overline{\mathbf Q}}$.
Let ${\mathcal F}^0(g,s; {\mathcal D}_r)={\mathcal F}^0_{\infty}(g,s; \phi_\infty) \otimes \widetilde{\mathcal F}_{\rm fin}(g, s; {\mathcal D}_r)$.

\begin{prop}\label{prop:RatEis}
{\itshape Let ${\mathcal F}^0(g,s; {\mathcal D}_r)={\mathcal F}^0_{\infty}(g,s) \otimes \widetilde{\mathcal F}_{\rm fin}(g, s; {\mathcal D}_r)$
                    and ${\mathcal K} = {\mathcal K}_0({\mathfrak N}_F\varpi_{v_0}) \cap {\mathcal K}(p^r)  \subset {\rm GL}_2(\widehat{\mathcal O}_F)$. 
  Denote by $\partial Y^{\rm BS}_{\mathcal K}(p^r)$ the boundary of the Borel-Serre compactification of $Y_{\mathcal K}(p^r)$.
\begin{enumerate}
\item We have 
         \begin{align}\label{eq:RatEis}
            \begin{aligned}
               & \int_{[{\rm N}_2]}  \delta(E_{0,s})( n(x) g)  {\rm d}x   \\
            =& 2^{(s-1)[F:{\mathbf Q}]} \sqrt{-1}^{ k_{\alpha} }   
                  \times \Gamma_{\mathbf C}(s+k_{\alpha}) 
                              L^{(\infty v_0{\mathfrak N}_F)}(2s+1, \phi_1\phi^{-1}_2)  
                  \times  \omega(g,s; {\mathcal F}^0(-,-;{\mathcal D}_0) )    \\
              &    + 2^{(s+k_\alpha-1)[F:{\mathbf Q}]} (-1)^{ [F:{\mathbf Q}] } 
                                 \times  \frac{\Gamma_{\mathbf C}(2s+k_\alpha-1)}{\Gamma_{\mathbf C}(s)^{[F:{\mathbf Q}]}}
                         L^{(\infty v_0)}(2s, \phi_1\phi^{-1}_2)  
                   \times  \omega(g,-s; {\mathcal M}{\mathcal F}^0(-,-;{\mathcal D}_0)  )
           \end{aligned}   
         \end{align}
         In particular, ${\rm r} (\delta(E_{0,0}))$ defines an element in 
         $H^{r_F}(\partial Y^{\rm BS}_{\mathcal K}(p^0), {\mathcal L}(n_\alpha, \phi_0 ; K_\pi  ))$.
\item The Eisenstein cohomology class $\delta(E_{0,0})$ is in $H^{ r_F }(Y_{\mathcal K}(p^0), {\mathcal L}(n_\alpha, \phi_0; K_\pi  ))$.
\end{enumerate}
}
\end{prop}
\begin{proof}
The identity in the first statement (\ref{eq:RatEis}) immediately follows from (\ref{eq:mobius}) and Lemma \ref{lem:FVal}. 
Recall that 
\begin{align*}
  \Gamma_{\mathbf C}(s+k_\alpha) L^{(\infty)}(2s+1, \phi_1\phi^{-1}_2)|_{s=0}
  \sim_{\overline{\mathbf Q}^\times} \pi^{-2(n-\alpha+t)} L^{(\infty)}( 2(n-\alpha+1), \omega^\prime_\phi ),  
\end{align*}
where the finite-order character $\omega^\prime_\phi$ is introduced in (\ref{eq:UnitChar1}).
Then assumption (\ref{eq:UnitChar2}) implies that $\omega^\prime_\phi$ is a totally even character. 
Hence the rationality of ${\rm r} (\delta(E_{0,0}))$ follows from (\ref{eq:RatEis}) and the fact that critical values of Hecke $L$-functions $L(s, \omega^\prime_\phi)$ in $K_\pi$
(Klingen-Siegel's theorem, see for instance \cite[Theorem 1]{sh76}).

\cite[Theorem 2]{ha87} yields that the map 
\begin{align*}
  {\rm Eis}: {\rm Im}({\rm r}) \to H^{ r_F }(Y_{\mathcal K}, {\mathcal L}(n_\alpha, \phi_0; {\mathbf C}  ))
  : {\omega}(g) \mapsto \sum_{\gamma \in {\rm B}_2(F)\backslash {\rm GL}_2(F)} \omega(\gamma g)
\end{align*}
is defined over $\overline{\mathbf Q}$ and it gives a section of ${\rm r}$.
Combining this and the first statement, the second statement follows.
\end{proof}

\begin{dfn}\label{dfn:EisDenom}
 Let $L\subset {\mathbf C}_p$ be an extension of ${\mathbf Q}_p$ and denote the ring of integers of $L$ by ${\mathcal O}_L$.   
 Define an ${\mathcal O}_L$-module $H^{ r_F }(Y_{\mathcal K}, {\mathcal L}(n_\alpha, \phi_0; {\mathcal O}_L  ))^\prime$ 
 to be the image of $H^{ r_F }(Y_{\mathcal K}, {\mathcal L}(n_\alpha, \phi_0; {\mathcal O}_L  ))$ 
 to $H^{ r_F }(Y_{\mathcal K}, {\mathcal L}(n_\alpha, \phi_0; {\mathbf C}_p ))$.
 We sometimes consider  $H^{ r_F }(Y_{\mathcal K}, {\mathcal L}(n_\alpha, \phi_0; {\mathcal O}_L  ))^\prime$
 as a sub ${\mathcal O}_L$-module of $H^{ r_F }(Y_{\mathcal K}, {\mathcal L}(n_\alpha, \phi_0; {\mathbf C}  ))$ 
 via fixed isomorphism ${\mathbf i}_p:{\mathbf C} \stackrel{\sim}{\to} {\mathbf C}_p$.
 
For each $\delta(E_{0,0}) \in H^{ r_F }(Y_{\mathcal K}(p^0), {\mathcal L}(n_\alpha, \phi_0; L  ))$, 
define $\delta_{{\mathcal K},L}(\Phi^{(0)})  \in {\mathcal O}_L$ to be a generator of the following ideal 
 \begin{align*}
  \left\{   a\in {\mathcal O}_L 
             \mid a\delta( E_{0,0} ) \in H^{ r_F }(Y_{\mathcal K}(p^0), {\mathcal L}(n_\alpha, \phi_0; {\mathcal O}_L  ))^\prime  \right\}.  
 \end{align*}
 This is well-defined, since $\delta( E_{0,0} )$ defines a rational cohomology class by Proposition \ref{prop:RatEis}. 
 We call $\delta_{{\mathcal K},L}(\Phi^{(0)})$ the denominator of $\delta( E_{0,0} )$.
\end{dfn}

\begin{rem}
We will discuss the rationality of $\delta(E_{r,0})$ and its relation with the critical values of Asai $L$-functions in Section \ref{sec:bdddenom}
for each $r>0$, 
where we reduce the rationality of $\delta(E_{r,0})$ to the rationality of $\delta(E_{0,0})$.
\end{rem}

\subsection{Cup products}\label{sec:cup} 

In this subsection, we introduce certain cup products of cohomology classes which are introduced in the previous subsections. 
This cup product is considered as a partial zeta integral, 
   and it will give an integral expression of the critical values of Asai $L$-functions twisted by Hecke characters $\varphi$.
This kind of description of Asai $L$-functions is found in \cite[Section 5.2]{gh99} in terms of the classical language if $\varphi$ is trivial.
We keep to introduce the adelic formulation.

The cup product in this subsection is given for the pair of
the Eisenstein cohomology classes
and the pull-backs to ${\rm GL}_2(F_{\mathbf A})$ of the image of cusp forms on ${\rm GL}_2(E_{\mathbf A})$ via the Eichler-Shimura-Harder map. 
Hence we firstly define the pull-back of cohomology classes on ${\rm GL}_2(E_{\mathbf A})$ to ${\rm GL}_2(F_{\mathbf A})$. 
For this purpose, define a differential operator $\nabla$ as follows:
 \begin{align*}
 \nabla = \frac{\partial^2}{\partial X \partial Y_c}  - \frac{\partial^2}{\partial X_c \partial Y }.
 \end{align*}
Assume that $A$ is a ${\mathbf Z}$-algebra such that $(n!)^{-1}\in A$.   
 Define $\Upsilon=\oplus^n_{\alpha=0} \Upsilon^\alpha: L(n;A) \to \oplus^n_{\alpha=0} L(2n-2\alpha;A)$  to be 
 \begin{align}\label{eq:upsdef}
   \Upsilon( P(X, Y, X_c, Y_c) ) = \oplus^n_{\alpha=0}  \frac{1}{\alpha!^2} \nabla^\alpha P(X,Y, X_c,Y_c)|_{X_c=X, Y_c=Y}.
 \end{align}
Note that we use the multi-index notation in the definition of $\nabla$.
 
 \begin{prop}\label{prop:hida}
(\cite[Lemma 2]{gh99})
 {\itshape 
 The map $\Upsilon=\oplus^n_{\alpha=0}\Upsilon^\alpha: L(n;A) \to \oplus^n_{\alpha=0} L(2n-2\alpha;A)$ is an isomorphism as  ${\rm SL}_2(A)$-modules. 
}
 \end{prop}

Let ${\mathcal K}^E$ be an open compact subgroup of ${\rm GL}_2(\widehat{\mathcal O}_E)$ and ${\mathcal K}^F= {\mathcal K}^E\cap {\rm GL}_2(\widehat{\mathcal O}_F)$.   
 Let us consider the local system ${\mathcal L}(n; A)$ (resp. ${\mathcal L}(n_\alpha; A)$) on 
       $Y^E_{\mathcal K}:=Y^E_{{\mathcal K}^E}$ (resp. $Y^F_{\mathcal K}:=Y^F_{{\mathcal K}^F}$)
which is determined by the ${\rm GL}_2(E)$-module $(\varrho_{n,m}, L(n; A))$ (resp. ${\rm GL}_2(F)$-module $(\varrho_{2n-2\alpha, \alpha+2m}, L(n_\alpha; A)$).
Define a map from ${\mathcal L}(n; A)$ to ${\mathcal L}(n_\alpha; A)$ by  
 \begin{align}\label{eq:res}
  {\mathcal L}(n; A)_{/Y^E_{{\mathcal K}}} \longrightarrow  {\mathcal L}(n_\alpha; A)_{/Y^F_{{\mathcal K} }} ; 
  (g, P_g) \longmapsto (g,  \Upsilon^\alpha P_{\iota(g)}),   
 \end{align} 
 where $\iota: {\rm GL}_2(F_{\mathbf A}) \to {\rm GL}_{2}(E_{\mathbf A})$ is the natural inclusion.
Then Proposition \ref{prop:hida} proves that $(\ref{eq:res})$ is a well-defined map between local systems.  
 Hence we have a map 
 \begin{align*}
     \begin{CD} H^{ r_F }_{\rm cusp}(Y^{E}_{\mathcal K} , {\mathcal L}(n; A))
                         @>>> H^{ r_F }_{\rm cusp}(Y^F_{\mathcal K}, {\mathcal L}(n_\alpha; A))      \end{CD}, 
                         \quad (0\leq \alpha \leq n)          
\end{align*}
and denote this map by $\Upsilon^\alpha$ again. 

Let ${\mathfrak N}$ be an ideal of $\widehat{\mathcal O}_E$ such that 
      ${\mathfrak N}$ is the conductor of $\pi$ and ${\mathfrak N}_F={\mathfrak N}\cap \widehat{\mathcal O}_F$. 
Put ${\mathcal K}^E={\mathcal K}_0({\mathfrak N} \varpi_{v_0}) \cap {\mathcal K}(p^r)\subset {\rm GL}_2( \widehat{\mathcal O}_E )$
       and ${\mathcal K}^F = {\mathcal K}_0({\mathfrak N}_F\varpi_{v_0}) \cap {\mathcal K}(p^r)  \subset {\rm GL}_2( \widehat{\mathcal O}_F )$. 
We write $  Y^\ast_{\mathcal K}(p^r) = Y_{  {\mathcal K}^\ast }$ for $\ast=E$ or $F$.  
For each Hecke character $\chi: E^\times\backslash E^\times_{\mathbf A} \to {\mathbf C}^\times$ of the conductor dividing ${\mathfrak N}$, 
we put $\chi_0=\chi|_{ {\mathcal O}^\times_E }, \chi_{00}=\chi|_{ {\mathcal O}^\times_F }$.
Then, in a similar way as in the definition of $\Upsilon^\alpha$ as above, we obtain a map
 \begin{align*}
     \begin{CD} \Upsilon^\alpha: H^{ r_F }_{\rm cusp}(Y^{E}_{\mathcal K}(p^r), {\mathcal L}(n, \chi_0; A))
                         @>>> H^{ r_F }_{\rm cusp}(Y^F_{\mathcal K}(p^r), {\mathcal L}(n_\alpha, \chi_{00}; A))      \end{CD}, 
                         \quad (0\leq \alpha \leq n).          
\end{align*} 
Applying $\chi$ for the central character $\omega_f$ of a cusp form $f$, 
we define $\delta^\alpha(f)$ (resp. $\eta^\alpha_f$) to be the image of $\delta(f)$ (resp. $\eta_f$) under the above map $\Upsilon^\alpha$.

Note that the natural embedding $\iota:{\rm GL}_2( F_{\mathbf A}) \to {\rm GL}_2( E_{\mathbf A})$ 
induces the map on the dual of Lie algebras:
\begin{align*}
   {\rm d}E_1 \longmapsto {\rm d}E, \quad 
   {\rm d}H_0 \longmapsto -{\rm d}H, \quad 
   {\rm d}E_2 \longmapsto {\rm d}E. 
\end{align*}
By using the above definitions and a description of $ \varrho(h^{(r)}_\xi) \eta_f$ in (\ref{eq:etaexp}), 
an explicit form of $\varrho(h^{(r)}_\xi) \eta^\alpha_f $ is given by 
\begin{align}\label{eq:ExpESH}
\begin{aligned}
   \varrho(h^{(r)}_\xi)  \eta^\alpha_f & \left( g   \right)(X_g)  \\
  =  p^{-mr}   
     & \sum^{n+1}_{i=-n-1}
      f^{i}(g h^{(r)}_\xi )  
         \left(  \varrho_{2n-2\alpha, \alpha+2m}(g) \cdot \Upsilon^\alpha ( v_i(-2))  {\rm d}E(L^{-1}_g X_g)   \right.   \\
     &     \left.  - \varrho_{2n-2\alpha, \alpha+2m}(g) \cdot \Upsilon^\alpha ( v_i(0))  \left( - {\rm d}H \right)(L^{-1}_g X_g)   
                                          -  \varrho_{2n-2\alpha, \alpha+2m}(g) \cdot  \Upsilon^\alpha ( v_i(2))    {\rm d}E(L^{-1}_g X_g)    \right).
\end{aligned}
\end{align}

Let $r \in {\mathbf Z} [\Sigma_{F,p}]$ with $ r_{v_1}=r_{v_2}>0$ for each $v_1, v_2\mid p$.
Write $Y_{\mathcal K}(p^r)=Y^F_{\mathcal K}(p^r)$, since only $Y^F_{\mathcal K}(p^r)$ appears hereafter.
We consider a cup product by dividing $Y_{\mathcal K}(p^r)$ to small pieces to make a local system trivial on each piece. 
We will see in the next section that the cup product on these small pieces will give us an integral expression of partial Asai $L$-functions. 
Consider
\begin{align*}
   {\rm det}_r: Y_{\mathcal K}(p^r) \to F^\times F^\times_{\infty,+} \backslash F^\times_{\mathbf A} 
                                                                                                            /  (\widehat{\mathcal O}^{(p)}_F )^\times (1+ p^r{\mathcal O}_{F,p})=:{\rm Cl}^+_F(p^r).  
\end{align*}
For each $x \in {\rm Cl}^+_F( p^r)$, 
define $Y_{\mathcal K}(p^r)_x$ to be the inverse image $\det^{-1}_r(x)$ of $x$.
We also denote by $\delta^\alpha(f)_x$ (resp. $\eta^\alpha_{f,x}, \delta(E_{r,0})_x$)
the pull-back of $\delta^\alpha(f)$ (resp. $\eta^\alpha_f, \delta(E_{r,0})$) via $Y_{\mathcal K}(p^r)_x \to Y_{\mathcal K}(p^r)$.

For the central character $\omega_f$ of a cusp form $f$, 
we put $\omega_{00}:=\omega_f|_{  \widehat{\mathcal O}^\times_F }$.  
Let $A$ be an ${\mathcal O}_{\pi}$-subalgebra of ${\mathbf C}_p$, 
and $\widetilde{A}$ the local system, which is determined by $(\varrho_{0, 2\kappa}, L(0; A))$.
If $g_1$ and $g_2\in {\rm GL}_2(F_{\mathbf A})$ give the same class in $Y_{\mathcal K}(p^r)_x$, then $u:=g_1g^{-1}_2$ satisfies 
$\det u \in {\mathcal O}^\times_{F,+}$ and $\det u \equiv 1$ mod $p^r$,  
and hence $\det(u)^{2\kappa}=1$. 
This shows that $\widetilde{A}$ is isomorphic to the trivial local system on $Y_{\mathcal K}(p^r)_x$. 
Note that $(n_\alpha!)^{-1} \in A$ by the assumption $p>n$.
For each $P, Q\in  L(n_\alpha; {\mathbf C}_p)$, we find that 
\begin{align}\label{eq:centweight}
\begin{aligned}\ 
  & [\varrho_{2n-2\alpha, \alpha+2m} (g)\cdot P, \varrho_{2n-2\alpha, \alpha+2m} (g)\cdot Q]_{2n-2\alpha}  \\
  =& [ (\det g)^{\alpha+2m} \varrho_{2n-2\alpha, 0} (g)\cdot P, (\det g)^{\alpha+2m} \varrho_{2n-2\alpha, 0} (g)\cdot Q]_{2n-2\alpha}  \\
  =& (\det g)^{2\alpha+4m} (\det g)^{2n-2\alpha} [P,Q]_{2n-2\alpha} \\
  =& (\det g)^{2n+4m}  [P,Q]_{2n-2\alpha}. 
\end{aligned}
\end{align}
Hence the non-degenerate pairing $[ \cdot, \cdot ]_{n_\alpha}$ on $L(n_\alpha; A)^{\otimes 2}$ 
   induces a cup product pairing:
\begin{align}\label{eq:cup}
 \begin{CD} 
   \cup: H^{ r_F }_{\rm cusp}(Y_{\mathcal K}(p^r)_x, {\mathcal L}(n_\alpha, \omega_{00}; A )) 
  \otimes H^{ r_F }(Y_{\mathcal K}(p^r)_x, {\mathcal L}(n_\alpha, \omega^{-1}_{00}; A ))
   @>>>  H^{ 2r_F }_{\rm c}(Y_{\mathcal K}(p^r)_x, \widetilde{A} )     \\
   @>>> A.
   \end{CD}
\end{align}
(See also \cite[(5.3)]{hi94} for this argument.)
The last map in (\ref{eq:cup}) is given by the upper horizontal map in the following commutative diagram:
\begin{align*}
   \begin{CD}
      H^{ 2r_F }_{\rm c}(Y_{\mathcal K}(p^r)_x, \widetilde{A} )   @>>>  A   \\
      @VVV   @VVV \\
      H^{ 2r_F }_{\rm c}(Y_{\mathcal K}(p^r)_x, {\mathbf C} )   @>\sim>>  {\mathbf C}   
   \end{CD}
\end{align*}
where the vertical maps are induced by the fixed isomorphism ${\mathbf i}_p:{\mathbf C} \to {\mathbf C}_p$
           and the lower horizontal map is given by the integration on $Y_{\mathcal K}(p^r)_x$. 
Then let us compute the image $\varrho(h^{(r)}_\xi) \delta^\alpha(f)_x \cup  \delta(E_{r,0})_x$ 
      of $\varrho(h^{(r)}_\xi) \delta^\alpha(f)_x \otimes  \delta(E_{r,0})_x$ via the cup product
      assuming that $A={\mathbf C}_p$.
For this purpose, we prepare some notations. 
Write $g_\infty \in {\rm GL}_2(F_\infty)$ as follows:
\begin{align}\label{eq:ginf}
  g_\infty= t y^{-\frac{1}{2}}\begin{pmatrix} y & x \\ 0 & 1  \end{pmatrix} k, \quad (t\in F^\times_\infty, x \in F_\infty, y \in F^\times_{\infty}, 
      k\in C_{\infty,+}).   
\end{align}
Then the identity (\ref{eq:coord}) implies that  
\begin{align*}
  \frac{1}{2^{r_F}} {\rm d}H\wedge {\rm d}E(L^{-1}_g X_g) = 2^{r_F} \frac{{\rm d} x\wedge {\rm d}y}{y^2}.   
\end{align*}
Hence, for each measurable function ${\mathcal F}$ on $Y_{\mathcal K}(p^r)$, 
we find that 
\begin{align*}
  \int_{Y_{\mathcal K}(p^r) }   {\mathcal F}(g)   {\rm d}H\wedge {\rm d}E(L^{-1}_g X_g)   
   =     \int_{ {\rm GL}_2(F) F^\times_{\mathbf A} \backslash {\rm GL}_2(F_{\mathbf A})  }  
          {\mathcal F}(g) 
          {\rm d}_rg, 
\end{align*}
where 
\begin{align*}
   {\rm d}_rg_\infty =  2^{2r_F} \frac{{\rm d} x\wedge {\rm d}y}{y^2} {\rm d} k, 
   \quad {\rm vol}({\rm d}k,  C_{\infty,+} ) = 1,  
   \quad {\rm vol}({\rm d}_rg_{\rm fin}, {\mathcal K}(p^r) ) = 1  . 
\end{align*}
Put
\begin{align}
& \label{eq:defC(a,i)}
\begin{aligned}
      C(\alpha,i) 
= &  [ \Upsilon^\alpha ( v_i(0)),  (X-\sqrt{-1}Y)^{2n-2\alpha}   ]_{2n-2\alpha}  
        + \sqrt{-1} [  \Upsilon^\alpha (v_i(-2)),  (X-\sqrt{-1}Y)^{2n-2\alpha}   ]_{2n-2\alpha}     \\ 
    & \quad \quad  - \sqrt{-1} [ \Upsilon^\alpha (v_i(2)),  (X-\sqrt{-1}Y)^{2n-2\alpha}   ]_{2n-2\alpha}, 
 \end{aligned}   
      \\  \notag
 &  f^\alpha(g) = \sum_{-n-1\leq i \leq n+1}  C(\alpha, i)f^i(g).   
\end{align}
Define a partial zeta integral $I^\alpha_{r, x, s}$ by 
\begin{align}\label{eq:imcup}
I^\alpha_{r, x, s} =
&   p^{-mr}   \int_{Y_{\mathcal K}(p^r)_x} 
                                 f^\alpha(g h^{(r)}_\xi)  E_{r,s} \left( g   \right)   (\det  g^{-1}_p g_\infty )^{2[\kappa]}  
                                  {\rm d}_rg, 
\end{align}
where we recall $\kappa=n+2m$.

\begin{prop}\label{prop:cup}
{\itshape 
For each $x \in {\rm Cl}^+_F( p^r)$, we have 
$\displaystyle 
\varrho(h^{(r)}_\xi) \delta^\alpha(f)_x \cup  \delta(E_{r,0})_x
= I^\alpha_{r, x, 0}$.
}
\end{prop}
\begin{proof}
Then (\ref{eq:centweight}) and the explicit description (\ref{eq:ExpESH}) of $\delta^\alpha(f)(g)$ prove that,   
for each $g\in Y_{\mathcal K}(p^r)$, we can write
\begin{align*}
    \varrho(h^{(r)}_\xi) \delta^\alpha(f) \cup  \delta(E_{r,0}) (g)
    = p^{-mr} f^\alpha(g h^{(r)}_\xi)  E_{r,s} \left( g   \right)   (\det  g^{-1}_p g_\infty )^{2[\kappa]}
\end{align*}
for an open neighborhood of $g$.
Hence the image of $\varrho(h^{(r)}_\xi) \delta^\alpha(f)_x \cup  \delta(E_{r,0})_x$ under the trace map (\ref{eq:cup})
   is given by the formula in the statement. 
See also \cite[Section 5, 6]{gh99} for a similar argument in terms of the coordinate of the Poincar$\acute{\rm e}$ upper 3-space.
\end{proof}

\begin{rem}
Let ${\mathcal K}$  be an open compact subgroup of ${\rm GL}_2( \widehat{\mathcal O}_F )$ 
   such that ${\mathcal K}(p^r) \subset {\mathcal K}$
   and that ${\mathcal K}^{(p)} = {\mathcal K}(p^r)^{(p)}$.
For each $\omega \in H^{2r_F}_{\rm c}(Y_{\mathcal K}(p^r), \widetilde{A} )$ and  $g\in Y_{\mathcal K}(p^r)$, we write 
$ \omega(g) = (g, P_g) $
for an open neighborhood of $g$, where $P_g\in A$ is the fiber of $\omega$ around $g$.   
Consider the trace map ${\rm Tr}_{{\mathcal K}/{\mathcal K}(p^r)}$,
     which is induced by the natural projection $Y_{\mathcal K}(p^r) \to Y_{\mathcal K}$. 

Since the action of $u\in {\mathcal K}$ on $(g, P_g)$ is given by $(gu, \varrho_{0, 2\kappa}(u^{-1}_p) P_g )$    
by the definition of the local system $\widetilde{A}$ on $Y_{\mathcal K}$,  
the fiber of $\omega\cdot u$ around $g$ is given by 
  $\varrho_{0, 2\kappa}(u^{-1}_p) P_{gu^{-1}}$.
Hence the image  of $\varrho(h^{(r)}_\xi) \delta^\alpha(f) \cup  \delta(E_{r,0})$ via the trace map ${\rm Tr}_{{\mathcal K}/{\mathcal K}(p^r)}$ is given by 
\begin{align*}
     {\rm Tr}_{{\mathcal K}/{\mathcal K}(p^r)} \left(  \varrho(h^{(r)}_\xi) \delta^\alpha(f) \cup  \delta(E_{r,0})   \right)(g)
   =&  \sum_{u\in {\mathcal K}/ {\mathcal K}(p^r)} 
           \varrho_{0, 2\kappa}(u^{-1}_p)
            \left(    p^{-mr} f^\alpha(g u^{-1} h^{(r)}_\xi)  E_{r,s} \left( g u^{-1}  \right)   (\det  (gu^{-1})^{-1}_p g_\infty )^{2[\kappa]}  \right)   \\
   =&  \sum_{u\in {\mathcal K}/ {\mathcal K}(p^r)} 
                p^{-mr} f^\alpha(g u h^{(r)}_\xi)  E_{r,s} \left( g u  \right)   (\det  g^{-1}_p g_\infty )^{2[\kappa]}  . 
\end{align*}
This deduces that 
\begin{align}\label{eq:trace}
 \begin{aligned} 
   &p^{-mr}   \int_{Y_{\mathcal K}(p^r)_x} 
                                 f^\alpha(g h^{(r)}_\xi)  E_{r,s} \left( g   \right)   (\det  g^{-1}_p g_\infty )^{2[\kappa]}  
                                  {\rm d}_rg   \\
 =& p^{-mr}   
      \sum_{u\in {\mathcal K}/ {\mathcal K}(p^r)}
      \int_{Y_{\mathcal K, x}} 
                                 f^\alpha(gu h^{(r)}_\xi)  E_{r,s} \left( g u  \right)   (\det  g^{-1}_p g_\infty )^{2[\kappa]}  
                                  {\rm d}_{\mathcal K}g,   
 \end{aligned}
\end{align}
where ${\rm vol}({\rm d}_{\mathcal K}g, {\mathcal K} )=1$ 
   and $Y_{\mathcal K, x}$ is the image of $Y_{\mathcal K}(p^r)_x$ under the natural projection $Y_{\mathcal K}(p^r) \to Y_{\mathcal K}$.
We will use the identity (\ref{eq:trace}) in Section \ref{sec:bdddenom} and \ref{sec:dist}.
\end{rem}

\section{Integral expression of Asai $L$-functions}\label{sec:AsaiInt}

In this section, we describe the integral expression of Asai $L$-functions following the strategy in \cite{gh99} and \cite{gh99b}. 
As in the previous sections, we emphasis to use the adelic language, 
that is, we express Asai $L$-functions in terms of a product of local zeta integrals. 
See Proposition \ref{prop:unfoldloc} for the result. 
We will prove the interpolation formula for Asai $L$-functions by using   
local zeta integrals in Section \ref{sec:urInt}, \ref{sec:pInt} and \ref{sec:InfInt}, 
which is a different method from a prior work in \cite{lw}.

Let $\varphi: F^\times\backslash F^\times_{\mathbf A} \to {\mathbf C}^\times$ 
      be a finite-order Hecke character of the conductor dividing $p^r$. 
For each $a \in {\mathbf Z}$, put $\varphi_a = |\cdot|^a_{\mathbf A}\varphi$  
and denote by $\widehat{\varphi_a}$ the $p$-adic avatar of $\varphi_a$.

The following proposition, combining with the interpolation formula in Section \ref{sec:interpolation},
 explains that the integral $I^\alpha_{r,x,s}$ in (\ref{eq:imcup}) gives an integral expression of the partial Asai $L$-functions:

\begin{prop}\label{prop:Birch}(Birch Lemma)
{\itshape 
Let ${\rm d}g$ be a Haar measure on ${\rm GL}_2(F_{\mathbf A})$ so that 
\begin{itemize}
  \item  $\displaystyle {\rm d}g_\infty = 2^{2r_F} {\rm d}^\times t \frac{ {\rm d}x {\rm d}y }{y^2} {\rm d}k$,
             where $g_\infty$ is written as in {\rm (\ref{eq:ginf})} and ${\rm vol}({\rm d}k, C_{\infty,+})=1$; 
  \item  ${\rm vol}( {\rm d}g_{\rm fin}, {\rm GL}_2(\widehat{\mathcal O}_F) ) = 1$.  
\end{itemize}

Then we have 
\begin{align}\label{eq:birch}
\begin{aligned}
  & \sum_{x\in {\rm Cl}^+_F(p^r) } 
   \widehat{\varphi_{2[\kappa]} } (x) I^\alpha_{r, x, s}  \\
   =&   \frac{  p^{-mr}  }{ [{\rm GL}_2(\widehat{\mathcal O}_F) :  {\mathcal K}_0({\mathfrak N}_F\varpi_{v_0}) \cap {\mathcal K}(p^r)  ]  } 
       \int_{ {\rm GL}_2(F)\backslash  {\rm GL}_2(F_{\mathbf A})  / F^\times_\infty } 
           f^\alpha ( g h^{(r)}_\xi  )  
           E_{r,s} (g)  
        |\cdot|^{2[\kappa]}_{\mathbf A}\varphi(\det g)
       {\rm d}g.   
\end{aligned}
\end{align}
}
\end{prop}
\begin{proof}
We firstly remark that the left-hand side of the statement is well-defined. 
Since the region $Y_{\mathcal K}(p^r)_x$ of the integral $I^\alpha_{r, x, s}$ can be written as 
${\rm det}^{-1}_r(x)$, 
 we have to check that $\widehat{ \varphi_{2[\kappa]} }(\det g_1)=\widehat{ \varphi_{2[\kappa]} }( \det g_2)$ for each $g_1, g_2 \in {\rm det}^{-1}_r(x)$. 
However this is obvious, since $\widehat{ \varphi_{2[\kappa]} }$ is trivial on $F^\times F^\times_\infty \widehat{\mathcal O}_F (p^r)$ by the definition of the $p$-adic avatar.

By the definition of the $p$-adic avatar of Hecke characters,  
we find that 
\begin{align*}
    \sum_{x\in {\rm Cl}^+_F(p^r) } 
        \widehat{ \varphi_{2[\kappa]} }(x)  I^\alpha_{r, x, s}
=& \sum_{x\in {\rm Cl}^+_F(p^r) } 
      p^{-mr}   \int_{Y_{\mathcal K}(p^r)_x} 
                                 f^\alpha(g h^{(r)}_\xi)  E_{r,s} \left( g   \right)  (\det g^{-1}_p g_\infty)^{ 2[\kappa] }  
                                  {\rm d}_rg \widehat{\varphi_{2[\kappa]}} (x)     \\
=& p^{-mr}
     \sum_{x\in {\rm Cl}^+_F(p^r) } 
          \int_{Y_{\mathcal K}(p^r)_x} 
                                 f^\alpha(g h^{(r)}_\xi)  E_{r,s} \left( g   \right)  (\det g^{-1}_p g_\infty)^{ 2[\kappa] }
                                 \widehat{\varphi_{2[\kappa]}} (\det g)
                                  {\rm d}_rg      \\
=& p^{-mr}
      \int_{Y_{\mathcal K}(p^r)} 
                                 f^\alpha(g h^{(r)}_\xi)  E_{r,s} \left( g   \right)   
                                 \varphi_{2[\kappa]} (\det g)
                                  {\rm d}_rg. 
\end{align*}
Then the definition of $\varphi_{2[\kappa]}$ and the normalization of measure prove the statement.
\end{proof}

Let $d_{E/F} = \begin{pmatrix} (2\xi)^{-1} & 0 \\ 0 & 1 \end{pmatrix}$. 
Recall that $W_\pi$ is the Whittaker function which is introduced in (\ref{eq:WhittDef}).  
Define 
\begin{align*}
  W^\alpha_\pi(g) =  \sum_{-n-1\leq i \leq n+1} C(\alpha, i) W^i_{\pi}\left(  g    \right), \quad 
 W^{\alpha}_{\pi, v}(g) 
    = \begin{cases}  \sum_{-n_v-1\leq i \leq n_v+1} C(\alpha, i) W^i_{\pi,v}\left(  g    \right),  & (v\mid \infty),  \\
                              W_{\pi,v}(g), &  (v<\infty). \end{cases}  
\end{align*}
The following proposition gives the unfolding of the right-hand side of (\ref{eq:birch}):

\begin{prop}\label{prop:Unfold1}
{\itshape
We have 
\begin{align*}
 &  \int_{ {\rm GL}_2(F)\backslash  {\rm GL}_2(F_{\mathbf A})  / F^\times_\infty}  
     f^\alpha(g) E_{r,s} (g) |\cdot|^{2[\kappa]}_{\mathbf A} \varphi(\det g )  {\rm d}g    \\
 =& \prod_v \int_{{\rm N}_2(F_v) \backslash {\rm GL}_2(F_v)} 
                     W^\alpha_{\pi,v} \left(  d_{E/F}  g  \right)
                    \Phi^{(r)}_v( (0,1) g )  | \cdot |^{s+n-\alpha+1}_{v } \varphi_v(\det g) {\rm d}g.
\end{align*}
}
\end{prop}
\begin{proof}
Denote by $f^\alpha_u(g) =|\det g|^{ \frac{[\kappa]}{2}  }_{E_{\mathbf A} } f^\alpha(g) =|\det g|^{[\kappa]}_{\mathbf A} f^\alpha(g)$ the unitarization of $f^\alpha$.  
Let $P_{\rm mir}$ be the mirabolic subgroup of  ${\rm GL}_2$, that is,  
\begin{align*}
    P_{\rm mir}=\left\{ \begin{pmatrix} \ast & \ast \\ 0 & 1 \end{pmatrix}  \right\} \subset {\rm GL}_2.
\end{align*}
Then the left-hand side of the statement is given by 
\begin{align*}
  & \int_{   {\rm B}_2 (F)\backslash  {\rm GL}_2(F_{\mathbf A})  / F^\times_\infty}  
     f^\alpha_u(g) {\mathcal F}_{\Phi^{(r)}}(g) |\cdot |^{[\kappa]}_{\mathbf A} \varphi(\det g)  {\rm d}g   \\
=  & \int_{   P_{\rm mir} (F)\backslash  {\rm GL}_2(F_{\mathbf A})  / F^\times_\infty}  
        f^\alpha_u(g) 
     \Phi^{(r)}_{\rm fin}( (0,1) g )  \phi_{1,{\rm fin}}(\det g) |\det g |^{s+\frac{1}{2}}_{\rm fin}       \\
   & \quad \times  \left\{ \int_{F^\times_\infty}  \Phi_\infty((0,t) g) \phi_{1,\infty}|\cdot|^{s+\frac{1}{2}}_\infty(\det tg)   
                                  \phi^{-1}_{1,\infty}\phi^{-1}_{2,\infty}(t)   {\rm d}^\times t   \right\}    |\cdot|^{ [\kappa] }_{\mathbf A}  \varphi (\det g)   {\rm d}g.
\end{align*}
Note that $\phi_{1,\infty}\phi_{2,\infty} = \omega^{-1}_{\pi, \infty} |\cdot|^{-2[\kappa]}_{\infty}$ by the assumption (\ref{eq:UnitChar2})     
and that $\varphi(\det tg) = \varphi(\det g)$ for each $t \in F^\times_\infty$, 
    since $\varphi$ is of finite-order. 
Hence the above integral becomes  
\begin{align*}
   & \int_{   P_{\rm mir} (F)\backslash  {\rm GL}_2(F_{\mathbf A})  / F^\times_\infty}  {\rm d}g
        \int_{F^\times_\infty}{\rm d}^\times t   \\
     & \quad \quad    f^\alpha_u(tg) 
                                \Phi^{(r)}_{\rm fin}( (0,1) g )  |\det g |^{s-(\alpha+2m) + 1}_{\rm fin}  
                               \Phi_\infty((0,1) tg) |\cdot|^{s-(\alpha+2m)+ 1}_\infty(\det tg)      |\det t g|^{ [\kappa] }_{\mathbf A}   \varphi (\det tg)     \\
=    & \int_{   P_{\rm mir} (F)\backslash  {\rm GL}_2(F_{\mathbf A})  }  {\rm d}g
        f^\alpha_u(g) 
     \Phi^{(r)}( (0,1) g )  |\det g |^{s-(\alpha+2m) + 1}_{\mathbf A}     |\cdot |^{ [\kappa] }_{\mathbf A} \varphi(\det g)  \\
=    & \int_{   P_{\rm mir} (F)\backslash  {\rm GL}_2(F_{\mathbf A})  }  {\rm d}g
        f^\alpha_u(g) 
     \Phi^{(r)}( (0,1) g )     |\cdot |^{ s+n-\alpha+1 }_{\mathbf A} \varphi(\det g).  
\end{align*}

The Fourier-Whittaker expansion of $f_u$ shows that 
\begin{align*}
& \int_{   P_{\rm mir} (F)\backslash  {\rm GL}_2(F_{\mathbf A})  }  {\rm d}g
        f^\alpha_u(g) 
     \Phi^{(r)}( (0,1) g )     |\cdot |^{ s }_{\mathbf A} \varphi(\det g)  \\
=  & \int_{ P_{\rm mir}(F) \backslash  {\rm GL}_2(F_{\mathbf A})  }  
         \sum_{a\in E^\times } W^\alpha_\pi\left(  \begin{pmatrix} a & 0 \\ 0 & 1 \end{pmatrix} g  \right) 
           \Phi^{(r)} ( (0,1) g )  | \cdot |^{s}_{\mathbf A} \varphi(\det g)  {\rm d}g    \\
  = & \int_{ P_{\rm mir}(F) \backslash  {\rm GL}_2(F_{\mathbf A})  }  
            \Phi^{(r)}( (0,1) g ) | \cdot |^{s}_{\mathbf A} \varphi(\det g)  {\rm d}g    \\
     & \quad \times 
      \left\{ 
       \sum_{a \in F^\times}  W^\alpha_\pi \left(  \begin{pmatrix} a & 0 \\ 0 & 1 \end{pmatrix} g  \right)   
     + \sum_{a, a^\prime \in F^\times}  W^\alpha_\pi \left(  \begin{pmatrix} (1+a^\prime \xi) a & 0 \\ 0 & 1 \end{pmatrix} g  \right)  
     + \sum_{a \in F^\times}  W^\alpha_\pi \left(  \begin{pmatrix} a & 0 \\ 0 & 1 \end{pmatrix} d_{E/F} g  \right)  
        \right\}  \\
  = & \int_{ {\rm N}_2(F) \backslash  {\rm GL}_2(F_{\mathbf A})  }  
            \Phi^{(r)}( (0,1) g ) | \cdot |^{s}_{\mathbf A} \varphi(\det g) {\rm d}g    \\
     & \quad \times   \left\{ 
     W^\alpha_\pi \left(   g  \right)   
     +  \sum_{a\in F^\times}W_\pi \left(  \begin{pmatrix} (1+a \xi)  & 0 \\ 0 & 1 \end{pmatrix} g  \right)  
     +W^\alpha_\pi \left(  d_{E/F} g  \right)  
        \right\}  \\
  = & \int_{ {\rm N}_2(F_{\mathbf A}) \backslash  {\rm GL}_2(F_{\mathbf A})  }  
        \int_{F\backslash F_{\mathbf A} } {\rm d}u
            \Phi^{(r)}( (0,1) g ) | \cdot |^{s}_{\mathbf A} \varphi(\det g)  {\rm d}g    \\
     & \quad \times   \left\{ 
     W^\alpha_\pi \left(   \begin{pmatrix} 1 & u \\ 0 & 1 \end{pmatrix}  g  \right)   
     +  \sum_{a\in F^\times}W^\alpha_\pi \left(  \begin{pmatrix} (1+a \xi)  & 0 \\ 0 & 1 \end{pmatrix} \begin{pmatrix} 1 & u \\ 0 & 1 \end{pmatrix}  g  \right)  
     +W^\alpha_\pi \left(  d_{E/F} \begin{pmatrix} 1 & u \\ 0 & 1 \end{pmatrix}  g  \right)  
        \right\}  \\
  = & \int_{ {\rm N}_2(F_{\mathbf A}) \backslash  {\rm GL}_2(F_{\mathbf A})  }  
        \int_{F\backslash F_{\mathbf A} } {\rm d}u
            \Phi^{(r)}( (0,1) g ) | \cdot |^{s}_{\mathbf A} \varphi(\det g)  {\rm d}g    \\
     & \quad \times   \left\{ 
          \psi_E(u) W^\alpha_\pi \left(    g  \right)   
     +  \sum_{a\in F^\times}  \psi_E( (1+a\xi) u ) W^\alpha_\pi \left(  \begin{pmatrix} (1+a \xi)  & 0 \\ 0 & 1 \end{pmatrix}   g  \right)  
     +  \psi_E(u(2\xi)^{-1}) W^\alpha_\pi\left(  d_{E/F}  g  \right)  
        \right\}.
\end{align*}
By using 
\begin{align*}
          \int_{F\backslash F_{\mathbf A} }   \psi_E(u) {\rm d}u
     =   \int_{F\backslash F_{\mathbf A} }   \psi_E( (1+a\xi) u ) {\rm d}u
     = 0, \quad 
          \int_{F\backslash F_{\mathbf A} }   \psi_E(u(2\xi)^{-1}) {\rm d}u
     = 1,      
\end{align*}
the above integral is equal to 
\begin{align*}
   & \int_{ {\rm N}_2(F_{\mathbf A}) \backslash  {\rm GL}_2(F_{\mathbf A})  }  
        W^\alpha_\pi \left(  d_{E/F}  g  \right)
            \Phi^{(r)}( (0,1) g ) | \cdot |^{s}_{\mathbf A} \varphi(\det g) {\rm d}g    \\
  =& \prod_v \int_{{\rm N}_2(F_v) \backslash {\rm GL}_2(F_v)} 
                     W^\alpha_{\pi,v} \left(  d_{E/F}  g  \right)
                    \Phi^{(r)}_v( (0,1) g )  | \cdot |^{s}_{v } \varphi_v(\det g) {\rm d}g.
\end{align*}
This proves the proposition.
\end{proof}

Define local zeta integrals $ I_v \left(s; W_{\pi, v}, \Phi^{(r)}_{v}, \varphi_v \right)$ to be 
\begin{align}
   \begin{cases}   
           \displaystyle        
                    \int_{ {\rm N}_2(F_v )   \backslash {\rm GL}_2(F_v ) }  {\rm d}g
                     W^{\alpha_v}_{\pi, v}\left( d_{E/F} g    \right)  
                                  \varphi_v\left( \det  g  \right)  
                                    |\det g|^{s}_v \Phi_v( (0,1) g),            &   (v\mid \infty),    \\
           \displaystyle            \int_{ {\rm N}_2(F_v )  \backslash {\rm GL}_2(F_v ) }  {\rm d}g
                          W_{\pi, v}\left( d_{E/F} g h^{(r)}_{\xi,v}   \right)  
                                  \varphi_v\left( \det  g  \right)  
                                    |\det g|^{s}_v \Phi^{(r)}_{v}( (0,1) g),          &    (v\mid p),  \\          
            \displaystyle     \int_{ {\rm N}_2(F_v )  \backslash {\rm GL}_2(F_v ) }  {\rm d}g
                          W_{\pi, v}\left( d_{E/F} g     \right)  
                                  \varphi_v\left( \det  g  \right)  
                                    |\det g|^{s}_v \Phi_v( (0,1) g),                    &   (v\nmid p\infty).       
        \end{cases} 
\end{align}
Summarizing the arguments in Proposition \ref{prop:Birch},  \ref{prop:Unfold1}, 
we obtain the following proposition:

\begin{prop}\label{prop:unfoldloc}
{\itshape 
We have 
\begin{align*}
 \sum_{x\in C_F(p^r) } 
        \widehat{ \varphi_{2[\kappa]} } (x)I^\alpha_{r, x, s}   
   = & [{\rm GL}_2(\widehat{\mathcal O}_F) :  {\mathcal K}_0({\mathfrak N}_F\varpi_{v_0}) \cap {\mathcal K}(p^r)  ] 
         \times p^{-mr} q^{ [\kappa] r}_p
          \prod_{v \in \Sigma_F} I_v\left( s + n -\alpha +1 ; W_{\pi, v}, \Phi^{(r)}_v, \varphi_v  \right). 
\end{align*}
}
\end{prop}

The explicit formula for the local integrals $I_v\left( s + n -\alpha +1 ; W_{\pi, v}, \Phi^{(r)}_v, \varphi_v  \right)$
will be given in Section \ref{sec:urInt}, \ref{sec:pInt} and \ref{sec:InfInt} 
(see Proposition \ref{prop:UnramInt}, \ref{prop:TameInt}, \ref{prop:AuxInt}, \ref{prop:pInt} and \ref{prop:InfInt} for the result), 
which is described by the values of Asai $L$-functions.

\section{$p$-adic Asai $L$-functions}\label{sec:pAsai}

In this section, we give a definition of $p$-adic Asai $L$-functions ${\mathscr L}^\alpha_p({\rm As}^+_{\mathcal M}(\pi))$.  
See Theorem \ref{thm:Main} for the construction, which is the main theorem in this paper. 

In the construction, we use Eisenstein series $E_{r,s}$, which are introduced in Section \ref{sec:choice}. 
For the integral property of $p$-adic Asai $L$-functions, 
we have to study the behavior of the denominators of Eisenstein cohomology classes $\delta(E_{r,s})$. 
This is done in Section \ref{sec:bdddenom}. 
Besides this integrality, 
we have to construct a certain projective system by using partial zeta integrals $I^\alpha_{r,x,s}$, 
 since the $p$-adic $L$-functions are elements in an Iwasawa algebra. 
In our setting, this is done in Section \ref{sec:dist}.
In Section \ref{sec:interpolation}, we construct $p$-adic Asai $L$-functions (see Theorem \ref{thm:Main}).

\subsection{Bounding denominators}\label{sec:bdddenom}

In this subsection, we study behavior of the denominator of the partial zeta-integral $I^\alpha_{r, x,s}$.
See Corollary \ref{cor:denombdd} for the result.

Define a subgroup ${\mathcal K}_Z(p^r)$ of ${\rm GL}_2(\widehat{\mathcal O}_F)$  
to be ${\mathcal K}(p^r)\widehat{\mathcal O}^\times_F$.
Note that a complete set of representatives of ${\mathcal K}_Z(p^r)/ {\mathcal K}(p^r)$ is given by 
$\left\{  z(u): = \left( \begin{smallmatrix}  u & 0 \\ 0 & u  \end{smallmatrix} \right)  \middle|  u \in ({\mathcal O}_{F,p} / p^r {\mathcal O}_{F,p} )^\times    \right\}$
and that ${\mathcal K}_Z(p^0) = {\mathcal K}(p^0)$.
Write 
\begin{align*}
 Y_{{\mathcal K}_Z}(p^r)
={\rm GL}_2(F) \backslash {\rm GL}_2(F_\mathbf A)/C_{\infty,+}F^\times_\infty  \left( {\mathcal K}_0( {\mathfrak N}_F\varpi_{v_0} )  \cap {\mathcal K}_Z(p^r) \right).  
\end{align*}
For each $x\in {\rm Cl}^+_F(p^r)$, define 
\begin{align*}
Y_{{\mathcal K}_Z}(p^r)_x =
{\rm image\ of\ } Y_{\mathcal K}(p^r)_x 
{\rm \ via\ }
Y_{\mathcal K}(p^r) \to Y_{{\mathcal K}_Z}(p^r).
\end{align*}

We deduce the behavior of $\{ \delta(E_{r,s}) \}_{r\geq 1}$ by studying its relation to $\delta(E_{0,s})$. 
For this purpose, we need the following elementary lemmas: 

\begin{lem}\label{lem:intYZ}
{\itshape    
Let $I^\alpha_{r, x, s}$ be as in (\ref{eq:imcup}).   
Then we have 
\begin{align*}
   I^\alpha_{r, x, s} 
   =  p^{-mr} \int_{Y_{{\mathcal K}_Z}(p^r)_x} 
      f^\alpha ( g h^{(r)}_\xi  )  
      \left( \sum_{u\in ({\mathcal O}_{F,p} / p^r {\mathcal O}_{F,p} )^\times}E_{r,s}\left( g z(u)  \right)   \right)
      (\det g^{-1}_p\det g_\infty)^{2[\kappa]}
 {\rm d}_{Z, r}g, 
\end{align*}
where we define a Haar measure ${\rm d}_{Z, r}g$ so that ${\rm d}_{Z, r, \infty}g = {\rm d}_{r, \infty}g$ and ${\rm vol}( {\rm d}_{Z, r, {\rm fin}}g, {\mathcal K}_Z(p^r))=1$. 
}
\end{lem}
\begin{proof}
In general, for each integrable function ${\mathcal F}\in L^1(Y_{\mathcal K}(p^r))$ on $Y_{\mathcal K}(p^r)$, we have 
\begin{align*}
     \int_{Y_{\mathcal K}(p^r)_x} {\mathcal F}(g)  {\rm d}_rg 
     =\int_{Y_{{\mathcal K}_Z}(p^r)_x}  \sum_{u\in ({\mathcal O}_{F,p} / p^r {\mathcal O}_{F,p} )^\times}   {\mathcal F}(gz(u))    {\rm d}_{Z, r}g.  
\end{align*}
Since $\varrho(h^{(r)}_\xi)f$ is invariant by the right-translation of $z(u)$ for each $u\in ({\mathcal O}_{F,p} / p^r {\mathcal O}_{F,p} )^\times$, 
the statement follows from (\ref{eq:trace}) for ${\mathcal K}={\mathcal K}_Z(p^r)$.
\end{proof}

\begin{lem}\label{lem:schav}
{\itshape 
For each $r \geq 1$,  
we have 
\begin{align*}
    \sum_{u\in ({\mathcal O}_{F,p} / p^r {\mathcal O}_{F,p} )^\times }  
      \Phi^{ (r) }_{p}( u m,  u n)
  = q^{r}_p \Phi^{(0) }_{p}\left( (m,n) \begin{pmatrix} p^r & 0 \\ 0 & 1 \end{pmatrix} \right)
      - q^{r-1}_p \Phi^{ (0) }_{p}\left( (m,n) \begin{pmatrix} p^{r-1} & 0 \\ 0 & 1 \end{pmatrix} \right).   
\end{align*}
}
\end{lem}
\begin{proof}
The elementary calculation shows that 
\begin{align*}
    & \sum_{u\in ({\mathcal O}_{F,v} /\varpi^{r}_v {\mathcal O}_{F,v} )^\times }  
      \psi_{F,v}\left( \frac{u m}{  \varpi^{r}_v } \right) {\mathbb I}_{ {\mathcal O}_{F,v} }(m)   \\
  =& \sum_{u\in {\mathcal O}_{F,v} /\varpi^{r}_v {\mathcal O}_{F,v}   }  
      \psi_{F,v}\left( \frac{um}{ \varpi^{r}_v } \right) {\mathbb I}_{ {\mathcal O}_{F,v} }(m)     
     -  \sum_{u\in {\mathcal O}_{F,v} /\varpi^{r-1}_v {\mathcal O}_{F,v}   }  
      \psi_{F,v}\left( \frac{ \varpi_v um}{ \varpi^r_v } \right) {\mathbb I}_{ {\mathcal O}_{F,v}  }(m)   \\
   =& q^{r}_v {\mathbb I}_{ \varpi^{r}_v {\mathcal O}_{F,v} }(m)
      - q^{r-1}_v  {\mathbb I}_{\varpi^{r-1}_v {\mathcal O}_{F,v} }(m)   
\end{align*}
for each $v\mid p$. 
This shows the statement.
\end{proof}

Let 
$A$ be an ${\mathcal O}_\pi$-subalgebra of ${\mathbf C}_p$ and 
$a(p^{r^\prime}) =  \begin{pmatrix} p^{r^\prime} & 0 \\ 0 & 1 \end{pmatrix} \in {\rm GL}_2(F_p)$ for each $r^\prime \in {\mathbf Z}$ with $0\leq r^\prime \leq r$.     
Recall that $\omega_{00}$ is defined to be $\omega_\pi|_{\widehat{\mathcal O}^\times_F}$. 
Define a homomorphism $\varrho( a(p^{r^\prime}) )$ between local systems ${\mathcal L}(n_\alpha, \omega_{00}; A  ))$ on $Y_{{\mathcal K}}(p^0)$ and $Y_{{\mathcal K}_Z}(p^r)$ 
    to be  
\begin{align*}
  \varrho( a(p^{r^\prime}) ): {\mathcal L}(n_\alpha, \omega_{00}; A  )_{ / Y_{{\mathcal K}}(p^0) } 
                                                   \to {\mathcal L}(n_\alpha, \omega_{00}; A  )_{  /  Y_{{\mathcal K}_Z}(p^r) }; 
   (g, P_g) \mapsto  (g, p^{-mr^\prime} \varrho^{(p)}_{n,m}( a (p^{r^\prime})) P_{ga(p^{r^\prime})}).
\end{align*}
Since we have 
\begin{align*}
   a(p^{r^\prime})^{-1} u   a(p^{r^\prime})  \in {\mathcal K}(p^0)  
\end{align*}
for each $u\in {\mathcal K}_Z(p^r)$, $\varrho( a(p^{r^\prime}) )$ is well-defined. 
See Section \ref{sec:ESHmap}, where we did a similar argument for the definition of $\rho(h^{(r)}_\xi )$ in (\ref{eq:etaexp}).

\begin{lem}\label{lem:int0}
{\itshape 
Recall that $\delta_{{\mathcal K}(p^0), K_\pi }(\Phi^{(0)} )$ is the denominator of $\delta(E_{0,0})$, which is defined in Definition \ref{dfn:EisDenom}. 
Then, for each $r\geq r^\prime \geq 0$,   
\begin{align*}
\delta^{\rm int}_{r^\prime} \left( E_{0,0}    \right) 
    :=  \delta_{{\mathcal K}(p^0), K_\pi }(\Phi^{(0)} ) 
          \varrho(   a( p^{r^\prime} )  )
           \delta\left(  E_{0,0} \right)  
\end{align*}
is an element in $H^{ r_F }(Y_{\mathcal K_Z}(p^r), {\mathcal L}(n_\alpha, \omega^{-1}_{00}; {\mathcal O}_\pi  ))^\prime$.
}
\end{lem}
\begin{proof}
By the definition, $\delta^{\rm int}_0( E_{0,0} )$ defines an element in 
$H^{ r_F }(Y_{\mathcal K}(p^0), {\mathcal L}(n_\alpha, \omega^{-1}_{00}; {\mathcal O}_\pi  ))^\prime$.  
Since $\delta^{\rm int}_{r^\prime} (E_{0,0})$ is the image of $\delta^{\rm int}_0(E_{0,0})$ under $\varrho(   a( p^{r^\prime} )  )$, 
the statement follows.
\end{proof}

Let $a(p^r) = \begin{pmatrix} p^r & 0 \\ 0 & 1 \end{pmatrix}\in {\rm GL}_2(F_p)$ and 
\begin{align}\label{dfn:eav}
     E^{\rm av}_{r,s}(g)  
  :=& \frac{1}{q^{r-1}_p} \sum_{u\in ({\mathcal O}_{F,p} / p^{r} {\mathcal O}_{F,p} )^\times }
           E_{r,s}(gz(u)) 
   = q_p E_{0,s} \left( g a(p^r)  \right)
      -  E_{0,s} \left( g a(p^{r-1})  \right).
\end{align}
Note that the middle term of (\ref{dfn:eav}) is appeared in Lemma \ref{lem:intYZ}  
and that Lemma \ref{lem:schav} shows the second identity in (\ref{dfn:eav}).
We write the image of $E^{\rm av}_{r,0}$ via the Eichler-Shimura map as 
$\delta(E^{\rm av}_{r,0}) \in H^{ r_F } (Y_{{\mathcal K}_Z}(p^r), {\mathcal L}(n_\alpha, \omega^{-1}_{00}; {\mathbf C}_p ) )$. 
Let   
$ \iota_{Z,r,x}:Y_{{\mathcal K}_Z}(p^r)_x \to Y_{{\mathcal K}_Z}(p^r)   $  
 be the natural inclusion.

\begin{prop}\label{prop:intEav}
{\itshape
For each $r\geq 1$, 
we find that 
\begin{align*}
   \delta_{{\mathcal K}(p^0), K_\pi }(\Phi^{(0)})
     \iota^\ast_{Z,r,x} \left( \delta(E^{\rm av}_{r,0}) \right) 
      \in H^{r_F}(Y_{{\mathcal K}_Z} (r)_x, {\mathcal L}(n_\alpha, \omega^{-1}_{00}; {\mathcal O}_\pi )).    
\end{align*}
In particular, the denominators of $\iota^\ast_{Z,r,x}\delta(E^{\rm av}_{r,0})$ are bounded by a constant $\delta_{{\mathcal K}(p^0), K_\pi }(\Phi^{(0)})$ which is independent of $r$ and $x$.
}
\end{prop}

\begin{proof}
Lemma \ref{lem:schav} shows that 
\begin{align*}
   \delta(E^{\rm av}_{r,0}) 
   =  q_p \varrho(a(p^r))   \delta( E_{0,0})  
        - \varrho(a(p^{r-1})) \delta\left( E_{0,0} \right),  
\end{align*}
and hence $\delta_{{\mathcal K}(p^0), K_\pi }(\Phi^{(0)}) \delta(E^{\rm av}_{r,0})$ is an element in 
$H^{ r_F } (Y_{{\mathcal K}_Z}(p^r), {\mathcal L}(n_\alpha, \omega^{-1}_{00}; {\mathcal O}_\pi ) )$ by Lemma \ref{lem:int0}.
This shows the proposition.
\end{proof}

\begin{cor}\label{cor:denombdd}
{\itshape 
For each $r>0$, 
     the value $ \delta_{ {\mathcal K}(p^0), K_\pi }(\Phi^{(0)}) \Omega^{-1}_{\pi, p} I^\alpha_{r, x, 0}$
          is an element in ${\mathcal O}_\pi$ 
}
\end{cor}
\begin{proof}
Proposition \ref{prop:cup}, Lemma \ref{lem:schav} and (\ref{dfn:eav}) show that 
\begin{align*}
\Omega^{-1}_{\pi, p} I^\alpha_{r, x, 0}
=&  q^{r-1}_p 
     \times \varrho(h^{(r)}_\xi)\eta^\alpha_{f,x} \cup  \iota^\ast_{Z,r,x} \delta(  E^{\rm av}_{r, 0}(g)  ). 
\end{align*}
Proposition \ref{prop:intEav} shows that 
the denominator of $\varrho( h^{(r)}_\xi )\eta_{f,x} \cup  \iota^\ast_{Z,r, x} \delta(  E^{\rm av}_{r,0}(g)  )$
is bounded by $\delta_{ {\mathcal K}(p^0), K_\pi }(\Phi^{(0)})$, which is independent of $r$.
This shows the statement.
\end{proof}

\subsection{Distribution property}\label{sec:dist}

In this subsection, we consider a distribution property of partial zeta integrals.
This enables us to construct a projective system in an Iwasawa algebra which gives the $p$-adic Asai $L$-functions. 
We firstly show a distribution relation for $I^\alpha_{r,x,s}$ in Proposition \ref{prop:dist}. 
The main result of this subsection is Corollary \ref{cor:dist}, which is an immediate consequence of Proposition \ref{prop:dist}, 
and which will be used for a construction an element in an Iwasawa algebra in the next subsection.

\begin{prop}\label{prop:dist}
{\itshape 
Let  $\lambda_{p,0}$ be the Hecke eigenvalue in (\ref{eq:lam0}), 
       $c_{r,\alpha, s}$ the constant in (\ref{eq:cars}) 
       and $I^\alpha_{r, x, s}$ the partial zeta integral in (\ref{eq:imcup}).
Then, for each $0< r$,  
we have  
\begin{align*}
  c_{r+1, \alpha, s}  \lambda^{-(r+1)}_{p,0}
  \sum_{y \in {\rm pr}^{-1}_{r} (x) } 
      I^\alpha_{r+1, y, s} 
  = c_{r, \alpha, s}  \lambda^{-r}_{p,0}
       I^\alpha_{r, x, s},
\end{align*}
where ${\rm pr}_{r}:{\rm Cl}^+_F(p^{r+1}) \to {\rm Cl}^+_F(p^r)$ is the natural projection.
}
\end{prop}
\begin{proof}
Note that (\ref{eq:trace}) shows that 
\begin{align}\label{eq:DistSum}
 \begin{aligned}
 & c_{r+1, \alpha, s} 
  \sum_{y \in {\rm pr}^{-1}_{r} (x)} 
      I^\alpha_{r+1, y, s}    \\
=&  c_{r+1, \alpha, s} p^{ -m(r+1) }
    \sum_{ g^\prime \in {\mathcal K}(p^r)/{\mathcal K}(p^{r+1}) }
    \int_{Y_{{\mathcal K} }(p^{r})_x}
               f^\alpha \left( g  g^\prime h^{(r+1)}_\xi  \right)  
                     E_{r,s} \left(g g^\prime \right)   
                     (\det g^{-1}_pg _\infty)^{ 2[\kappa] }
                          {\rm d}_rg.
\end{aligned}
\end{align}
We also recall that $\lambda_{p,0} p^{m} = \lambda_{p}$.
Hence it suffices to show that the right-hand side of (\ref{eq:DistSum}) is equal to 
\begin{align*}
  c_{r, \alpha, s} \lambda_p p^{-r}\cdot p^{-mr}
    \int_{Y_{{\mathcal K} }(p^{r})_x}
               f^\alpha \left( g  h^{(r)}_\xi  \right)  
                     E_{r,s} \left(g  \right)  
                     \det (g^{-1}_pg_\infty)^{ 2[\kappa] }
                          {\rm d}_rg.
\end{align*}
Write representatives $g^\prime  \in {\mathcal K}(p^r)/{\mathcal K}(p^{r+1})$ as follows:
\begin{align}\label{eq:gprime}
    g^\prime = g^{(1)}_{a,c} g^{(2)}_{b,d},  \quad 
    g^{(1)}_{a,c}
    = \begin{pmatrix} 1+a p^r   & 0 \\  c p^r & 1 \end{pmatrix}, \quad 
    g^{(2)}_{b,d}
    =  \begin{pmatrix} 1   & b p^r \\  0 & 1 \end{pmatrix}
                      \begin{pmatrix} 1   & 0 \\  0 & 1+d p^r \end{pmatrix}^{-1}     
\end{align}
for some $a, b, c, d \in {\mathcal O}_{F,p}/p{\mathcal O}_{F,p}$.

Firstly we compute the summation for $g^{(2)}_{b,d}$ in (\ref{eq:gprime}). 
By the definition of $\Phi^{(r)}_p$ in Definition \ref{dfn;BSfn}, we find that 
\begin{align*}
  E_{r,s}\left( g g^{(1)}_{a,c} g^{(2)}_{b,d}   \right)  
       =    E_{r,s}(g g^{(1)}_{a,c} ).            
\end{align*}
Note that 
\begin{align*}
    g^{(2)}_{b,d}  h^{(r+1)}_\xi                 
  = h^{(r)}_\xi   
    \begin{pmatrix} p & b+d\xi  \\ 0 & 1 \end{pmatrix}
    \begin{pmatrix} 1   & 0 \\  0 & 1+d p^r \end{pmatrix}^{-1}. 
\end{align*}
Since $\{1, \xi\}$ gives a pair of basis of ${\mathcal O}_{E,p}$ over ${\mathcal O}_{F,p}$, 
the action of $U(p)$-operator on $f$ is given as follows:
\begin{align*}
   U(p)(f) = \sum_{b,d \in {\mathcal O}_{F,p}/ p{\mathcal O}_{F,p} } \varrho\left( \begin{pmatrix} p & b+d\xi  \\ 0 & 1 \end{pmatrix} \right) f = \lambda_p f. 
\end{align*}
Hence we have 
\begin{align*}
  & \sum_{b, d \in {\mathcal O}_{F,p}/ p{\mathcal O}_{F,p}}
    \int_{Y_{{\mathcal K} }(p^{r})_x}
              f^\alpha \left( g g^{(1)}_{a, c} g^{(2)}_{b,d} h^{(r+1)}_\xi  \right)  
                    E_{r,s} \left(g g^{(1)}_{a,c} g^{(2)}_{b,d} \right)   
                        \det (g^{-1}_pg_\infty)^{ 2[\kappa] }  {\rm d}_rg \\ 
 =&  \sum_{b, d \in {\mathcal O}_{F,p}/ p{\mathcal O}_{F,p}}
    \int_{Y_{{\mathcal K} }(p^{r})_x}
              f^\alpha \left( g g^{(1)}_{a,c} h^{(r)}_\xi  
    \begin{pmatrix} p & b+d\xi  \\ 0 & 1 \end{pmatrix}  \right)  
                         E_{r,s} \left(g g^{(1)}_{a,c} \right)   
                        \det (g^{-1}_pg_\infty)^{ 2[\kappa] }  {\rm d}_rg    \\                      
=& \lambda_p \int_{Y_{{\mathcal K} }(p^{r})_x}
              f^\alpha \left( g g^{(1)}_{a,c} h^{(r)}_\xi \right)  
                  E_{r,s} \left(g g^{(1)}_{a,c} \right)   
                        \det (g^{-1}_pg_\infty)^{ 2[\kappa] }  {\rm d}_rg. 
\end{align*}

Secondary we compute the summation for $g^{(1)}_{a,c}$ in (\ref{eq:gprime}).
Note that 
\begin{align*}
   \left( h^{(r)}_\xi  \right)^{-1} g^{(1)}_{a,c}  h^{(r)}_\xi  
 =&  \begin{pmatrix} 1+a p^r - c p^r \xi &   (a-c\xi)\xi \\  cp^{2r} & 1+cp^r \xi \end{pmatrix} \in {\mathcal K}_1(p^r)_p.
\end{align*}
Hence we find that 
\begin{align*}
    f^\alpha \left( g g^{(1)}_{a,c}  h^{(r)}_\xi  \right)     
    =  f^\alpha \left( g   h^{(r)}_\xi  \right).       
\end{align*}
 Proposition \ref{prop:distKE2} shows that 
 \begin{align*}
 &  c_{r+1, \alpha, s} 
    \sum_{a, c \in {\mathcal O}_{F,p}/ p{\mathcal O}_{F,p}}
    \int_{Y_{{\mathcal K} }(p^{r})_x}
              f^\alpha \left( g g^{(1)}_{a, c}  h^{(r)}_\xi  \right)  
                    E_{r,s} \left(g g^{(1)}_{a,c}  \right)   
                         \det (g^{-1}_pg_\infty)^{ 2[\kappa] } {\rm d}_rg   \\
 =&   c_{r, \alpha, s} 
        \int_{Y_{{\mathcal K} }(p^{r})_x}
              f^\alpha \left( g  h^{(r)}_\xi  \right)  
                    E_{r,s} \left(g   \right)   
                        \det (g^{-1}_pg_\infty)^{ 2[\kappa] }  {\rm d}_rg.
\end{align*}
This proves the proposition.
\end{proof}

Let 
\begin{align}\label{eq:Itilde}
\widetilde{I}^\alpha_{r,x, s}
 = \frac{1}{\Omega_{\pi, p}} 
    \times c_{r, \alpha, s}  \lambda^{-r}_{p,0}
       I^\alpha_{r,x, s}.
\end{align}
Then the following statement is an immediate corollary of Proposition \ref{prop:dist}:

\begin{cor}\label{cor:dist}(Distribution property)
{\itshape 
We have 
$\displaystyle 
  \sum_{y \in {\rm pr}^{-1}_r(x) }       
        \widetilde{I}^\alpha_{r+1, y, s} 
  =   \widetilde{I}^\alpha_{r,x, s}$.  
}
\end{cor}

\subsection{Interpolation formulas}\label{sec:interpolation}

In this subsection, we construct $p$-adic Asai $L$-functions ${\mathscr L}^\alpha_{p}({\rm As}^+_{\mathcal M}(\pi))$. 
Let ${\rm rec}:{\rm Cl}^+_F(p^r) \to {\rm Gal}(F ( p^r )/ F)$ be the geometrically normalized reciprocity map.
Put $\sigma_x={\rm rec}(x)$ for $x\in {\rm Cl}^+_F(p^r)$.
Write $2\xi= \sqrt{-D_{E/F}} \in E$.
For each $v\in \Sigma_F$, define $\lambda_{E_v/F_v}(\psi_{F,v})$ to be 1 if $v$ is split in $E/F$
and to be  the Langlands constant 
associated with $E_v/F_v$ and $\psi_{F,v}$ if $v$ is non-split,  
and put $\lambda_{E/F, p}(\psi_F)=\prod_{v\mid p} \lambda_{E_v/ F_v}(\psi_{F,v})$. 
See \cite[(30.4.1)]{bh06} for the definition of the Langlands constants, and 
   \cite[(30.4.3)]{bh06} yields that $\lambda_{E/F, p}(\psi_F)$ is a 4-th root of unity.
Note that $\lambda_{p,0}$ is a $p$-adic unit, since $\pi$ is nearly $p$-ordinary. 
Hence Corollary \ref{cor:denombdd} shows that $\delta_{{\mathcal K}(p^0), K_\pi } (\Phi^{(0)})  \widetilde{I}^\alpha_{r, x, 0} \in {\mathcal O}_\pi$ for each $r>0$.
Define 
\begin{align}\label{eq:cnax}
\begin{aligned}
c_\infty(n,\alpha, \xi) =& \prod_{v\mid \infty} (-1)^{n_v} \sqrt{-1}^{\alpha_v}  2^{2r_F}D^{ -\frac{1}{2}( n_v-\alpha_v ) }_{E/F}   \binom{n_v}{\alpha_v}^2
                                         \in {\mathcal O}^\times_\pi,  \\ 
   {\mathscr L}^\alpha_{p, v_0, r}({\rm As}^+_{\mathcal M}(\pi)) 
     =& \delta_{{\mathcal K}(p^0), K_\pi } (\Phi^{(0)}) 
           c_\infty(n,\alpha, \xi)^{-1}
          \sigma_{\xi^2}  \lambda^{-1}_{E/F, p}(\psi_F)
           \sum_{x\in {\rm Cl}^+_F(p^r) } 
           \widetilde{I}^\alpha_{r, x, 0} 
           \sigma_x  
              \in {\mathcal O}_\pi  [ {\rm Gal}( F( p^r )/ F )   ].        
\end{aligned}
\end{align}
Recall that $v_0$ is the fixed auxiliary prime. 
Corollary \ref{cor:dist} shows that the system  $\left\{ {\mathscr L}^\alpha_{p, v_0, r}({\rm As}^+_{\mathcal M}(\pi)) \right\}_{r \geq 1}$ is a projective system. 
We define 
\begin{align*}
     {\mathscr L}^\alpha_{p, v_0}({\rm As}^+_{\mathcal M}(\pi))
     =& {\rm Tw}^{[\kappa]+(\alpha+2m)}_p
     \lim_{\substack{ \leftarrow \\ r  }} {\mathscr L}^\alpha_{p, v_0, r}({\rm As}^+_{\mathcal M}(\pi))
     \in  {\mathcal O}_\pi  [[ {\rm Gal}( F (\mu_{p^\infty})/ F )   ]],    
\end{align*}
where 
${\rm Tw}_p$ is defined in Remark \ref{rem:mainthm} \ref{rem:mainthmTw} 
and $v_0$ is a fixed auxiliary prime in Definition \ref{def:aux}.     
The following statement is the main theorem in this paper: 

\begin{thm}\label{thm:Main}(Interpolation property)
{\itshape 
Assume that 
\begin{itemize}
 \item $\pi$ is nearly $p$-ordinary;
 \item $\omega_{\pi,p}$ is unramified;  
 \item  if $\pi$ is conjugate self-dual, then $\alpha \neq n$; 
 \item the conductor ${\mathfrak N}$ of $\pi$ is square-free;
\item for each $v\mid {\mathfrak N}_F$ with $v\nmid p$, suppose either one of the following conditions: 
\begin{itemize}
\item $\omega_{\pi, v}$ is ramified;  
\item if $\omega_{\pi, v}$ is unramified, 
         then $v=ww_c\mid {\mathfrak N}_F$ splits in $E/F$
        and one of $\pi_w$ and $\pi_{w_c}$ is an unramified principal representation 
                and the other is a special representation;
\end{itemize} 
 \end{itemize}
Then, 
there exists ${\mathscr L}^\alpha_{p}({\rm As}^+_{\mathcal M}(\pi))\in  K_\pi \otimes_{\mathcal O_\pi} {\mathcal O}_\pi  [[ {\rm Gal}( F(\mu_{p^\infty})/ F )   ]]$
    for each finite-order Hecke character $\varphi$ of a $p$-power conductor 
         satisfying that $(-1)^{n - \alpha }\varphi(-1_\sigma)=1$ for each $\sigma \in \Sigma_{F,\infty}$,
we have
\begin{align*}
      \widehat{\phi }  ( {\mathscr L}^\alpha_{p}({\rm As}^+_{\mathcal M}(\pi))  )   
    =&      {\mathcal E}_\infty({\rm As}^+_{\mathcal M} (\pi) ( \phi ))
        {\mathcal E}_p({\rm As}^+_{\mathcal M} (\pi) ( \phi )) 
        \frac{  L( 0, {\rm As}^+_{\mathcal M} (\pi)(\phi)) }{\Omega( {\rm As}^+_{\mathcal M}(\pi)  )}, 
\end{align*}
where $\phi=|\cdot|^{n-\alpha}_{\mathbf A}\varphi$.
}
\end{thm}
\begin{proof}
Proposition \ref{prop:unfoldloc} yields that 
\begin{align*}
      \widehat{\phi} \left( \delta_{{\mathcal K}(p^0), K_\pi } (\Phi^{(0)})^{-1}  {\mathscr L}^\alpha_{p, v_0}({\rm As}^+_{\mathcal M}(\pi))  \right)   
    =&  \widehat{\varphi_{2[\kappa]}} \left(  \delta_{{\mathcal K}(p^0), K_\pi } (\Phi^{(0)})^{-1}   \lim_{\substack{ \leftarrow \\ r  }} {\mathscr L}^\alpha_{p, v_0, r}({\rm As}^+_{\mathcal M}(\pi))   \right)   \\
    =&  c_\infty(n,\alpha, \xi)^{-1} 
           \times   [{\rm GL}_2(\widehat{\mathcal O}_F) :  {\mathcal K}_0({\mathfrak N}_F\varpi_{v_0}) \cap {\mathcal K}(p^r)  ]  \\
      & \quad      \times p^{-mr} q^{ [\kappa] r}_p
                        \widehat{\phi}(\sigma_{\xi^2})  \lambda^{-1}_{E/F, p}(\psi_F) 
          \prod_{v \in \Sigma_F}       I_v\left( s + n -\alpha +1 ; W_{\pi, v}, \Phi^{(r)}_v, \varphi_v  \right).  
\end{align*}
The explicit calculation of the local integrals $ I_v\left(s; W_{\pi, v}, \Phi^{(r)}_v, \varphi_v\right) $ in Section \ref{sec:urInt}, \ref{sec:pInt} and \ref{sec:InfInt} 
(Proposition \ref{prop:UnramInt}, \ref{prop:TameInt}, \ref{prop:AuxInt}, \ref{prop:pInt} and \ref{prop:InfInt}) shows that the above product of local integrals is equal to 
\begin{align*}
          {\mathcal E}_\infty({\rm As}^+_{\mathcal M} (\pi) ( \phi ))
        {\mathcal E}_p({\rm As}^+_{\mathcal M} (\pi) ( \phi )) 
        \frac{  L( 0, {\rm As}^+_{\mathcal M} (\pi)(\phi)) }{\Omega( {\rm As}^+_{\mathcal M}(\pi)  )}  
       \times q_{v_0}\left(1 - \omega_{\pi,v_0}(\varpi_{v_0}) \varphi^2_{v_0}(\varpi_{v_0})  q^{-2(n-\alpha+1)}_{v_0} \right).
\end{align*}
Hence we have to remove a term relating with $v_0$ from the above identity to obtain the interpolation formula in the statement. 
We may assume that $\omega_{\pi}$ is trivial, since an auxiliary prime $v_0$ appears only in this case. 
Let $\sigma_{v_0}$ be an element in ${\rm Gal}(F(\mu_{p^\infty}) / F)$ so that $\widehat{\phi}(\sigma_{v_0}) = \phi(\varpi_{v_0})$. 
Define 
\begin{align*}
   P_{v_0} = q_{v_0}\left(1 - \omega_{\pi,v_0}(\varpi_{v_0}) q^{-2}_{v_0}  \sigma^2_{v_0}     \right) 
                 = q_{v_0}\left(1 -  q^{-2}_{v_0}  \sigma^2_{v_0}     \right)  \in {\mathcal O}_\pi[[ {\rm Gal}( F(\mu_{p^\infty})/ F ) ]]. 
\end{align*}
Then conditions (Aux1) and (Aux2) in Definition \ref{def:aux} show that 
   $P_{v_0}$ is an element in ${\mathcal O}_\pi[[ {\rm Gal}( F(\mu_{p^\infty})/ F ) ]]^\times$.   
Hence 
\begin{align*}
   {\mathscr L}^\alpha_{p}({\rm As}^+_{\mathcal M}(\pi)) 
       :=  \delta_{{\mathcal K}(p^0), K_\pi } (\Phi^{(0)})^{-1}
            {\mathscr L}^\alpha_{p, v_0}({\rm As}^+_{\mathcal M}(\pi)) \times P^{-1}_{v_0}
       \in K_\pi \otimes_{ {\mathcal O}_\pi }  {\mathcal O}_\pi[[ {\rm Gal}( F(\mu_{p^\infty})/ F ) ]]
\end{align*}
satisfies the interpolation formula in the statement.
\end{proof}

\begin{rem}\label{rem:denom}
    We find that $ \delta_{{\mathcal K}(p^0), K_\pi } (\Phi^{(0)}) {\mathscr L}^\alpha_{p} ({\rm As}^+_{\mathcal M}(\pi))$
     is an element in ${\mathcal O}_\pi[[ {\rm Gal}( F(\mu_{p^\infty}) / F ) ]]$ 
     by the proof of Corollary \ref{cor:denombdd}. 
\end{rem}

\section{Local integral at unramified and  tame places}\label{sec:urInt}

In this section, we compute the local integral $I_v\left(s; W_{\pi, v}, \Phi^{(r)}_v, \varphi_v \right)$ for the place $v\nmid p\infty$. 
See Proposition \ref{prop:UnramInt}, \ref{prop:TameInt} and \ref{prop:AuxInt}
 for the result.
We write $\Phi^{(r)}_v$ as $\Phi_v$ in this section.

\subsection{At places not dividing $\infty p {\mathfrak N}_F v_0$}

Suppose that $v\in \Sigma_F$ does not divide $\infty p {\mathfrak N}_F v_0$ 
and 
write $\Phi_v = \Phi^{ (r) }_v$ for the simplicity.
Recall that
\begin{align}\label{eq:urlocint}
\begin{aligned}
 I_v\left(s; W_{\pi, v}, \Phi_v, \varphi_v \right) 
 =&  \int_{ F^\times_v  }  \frac{{\rm d}^\times a}{|a|_v} 
        \int_{ F^\times_v }  {\rm d}^\times t
                          W_{\pi, v}\left( d_{E/F}  t \begin{pmatrix} a & 0 \\ 0 & 1 \end{pmatrix}     \right)  
                                  \varphi_v\left( a t^2  \right)  
                                  | a t^2 |^{ s }_v  \Phi_v\left(  (0,1) t \begin{pmatrix} a & 0 \\ 0 & 1 \end{pmatrix} ,s\right)  \\
 =&  \int_{ F^\times_v  }  \frac{{\rm d}^\times a}{|a|_v} 
                          W_{\pi, v}\left(   \begin{pmatrix} a & 0 \\ 0 & 1 \end{pmatrix}     \right)  
                          \varphi_v \left( a \right)
                          |a|^{ s }_v  
     \times   \int_{ F^\times_v } {\rm d}^\times t
                    \omega_{\pi, v}(t) \times   \varphi^2_v\left( t  \right)                 
                                  |  t |^{2s}_v   \\
 =&  \int_{ F^\times_v  }  \frac{{\rm d}^\times a}{|a|_v} 
                          W_{\pi, v}\left(   \begin{pmatrix} a & 0 \\ 0 & 1 \end{pmatrix}     \right)  
                          \varphi_v \left( a \right)
                          |a|^{s}_v
      \times L(2s, \omega_{\pi, v} \varphi^2_v   ).
\end{aligned}
\end{align}

\begin{prop}\label{prop:UnramInt}
{\itshape Assuming $v\nmid \infty p {\mathfrak N}_F v_0$, we have 
\begin{align*}
 I_v\left(s; W_{\pi, v}, \Phi_v, \varphi_v \right) 
 = L(s, {\rm As}^+(\pi)_v\otimes \varphi_v).
\end{align*}
}
\end{prop}
\begin{proof}
By using (\ref{eq:urlocint}), the proof is done in a straightforward way, so we omit it.    
\end{proof}

\subsection{At places $v\mid {\mathfrak N}_F, v\nmid p$}

Assuming that ${\mathfrak N}$ is square-free, we give an explicit formula of the local integral.
We firstly recall the explicit formula of $L(s, {\rm As}^+(\pi)_v)$ as follows:
\begin{itemize}
  \item If $v=ww_c$ is split in $E/F$, then we have $L(s, {\rm As}^+(\pi)_v) = L(s, \pi_w \times \pi_{w_c})$.
          (See \cite[Proposition (1.4)]{gj78} for the explicit formula for $L(s, \pi_w \times \pi_{w_c})$.)
  \item Suppose that $v$ is inert in $E/F$. 
          If $\pi_v$ is a principal series representation $\pi(\mu_v, \nu_v)$ 
                             for some unramified character $\mu_v:E^\times_v\to {\mathbf C}^\times$ 
                                    and ramified character $\nu_v:E^\times_v\to {\mathbf C}^\times$, then
              \begin{align*}
                L(s, {\rm As}^+(\pi)_v) = L(s, \mu_v|_{F^\times_v}). 
              \end{align*}      
           If  $\pi_v={\rm Sp}(\eta_v) \subset \pi(\eta_v|\cdot|^{\frac{1}{2}}_v, \eta_v|\cdot|^{-\frac{1}{2}}_v)$ is a special representation
                 for some quadratic character $\eta_v:E^\times_v\to {\mathbf C}^\times$,  
           then 
           \begin{align*}
               L(s, {\rm As}^+(\pi)_v)= L(s+1, \eta_v|_{F^\times_v}) L(s, \tau_{E_v/F_v}\eta_v|_{F^\times_v}),   
           \end{align*}              
           where $\tau_{E_v / F_v}$ is the quadratic character associated with $E_v/ F_v$.     
\end{itemize}

Note that 
\begin{align*}
   &  I_v\left(s; W_{\pi, v}, \Phi_v, \varphi_v \right)  \\
 =& \frac{1}{[{\rm GL}_2({\mathcal O}_{F,v}) : {\mathcal K}_0( {\mathfrak N}_{F,v} ) ]}
   \int_{ F^\times_v  }  \frac{{\rm d}^\times a}{|a|_v} 
        \int_{ F^\times_v }  {\rm d}^\times t
                          W_{\pi, v}\left( d_{E/F}  t \begin{pmatrix} a & 0 \\ 0 & 1 \end{pmatrix}     \right)  
                                  \varphi_v \left( a t^2  \right)  
                                  | a t^2 |^{ s }_v  \Phi_v\left(  (0,1) t \begin{pmatrix} a & 0 \\ 0 & 1 \end{pmatrix} ,s\right)   \\
 =& \frac{1}{[{\rm GL}_2({\mathcal O}_{F,v}) : {\mathcal K}_0( {\mathfrak N}_{F,v} ) ]} 
      \int_{ F^\times_v  }  \frac{{\rm d}^\times a}{|a|_v} 
                          W_{\pi, v}\left(   \begin{pmatrix} a & 0 \\ 0 & 1 \end{pmatrix}     \right)  
                          \varphi_v\left( a \right)
                          |a|^{ s  }_v  \\  
=&: \frac{1}{[{\rm GL}_2({\mathcal O}_{F,v}) : {\mathcal K}_0( {\mathfrak N}_{F,v} ) ]} 
        \widetilde{I}_v\left(s; W_{\pi, v}, \Phi_v, \varphi_v \right).  
\end{align*}
Then the computation of $\widetilde{I}_v\left(s; W_{\pi, v}, \Phi_v, \varphi_v \right)$ is done in a straightforward way. 
We summarize the result of the computation as follows:

\begin{prop}\label{prop:TameInt}
{\itshape
\begin{enumerate}
 \item Suppose that $v=ww_c$ splits in $E/F$. 
          \begin{enumerate} 
            \item Suppose that both $\pi_w$ and $\pi_{w_c}$ are principal series representations. 
                      Write $\pi_{w} = \pi(\mu_w, \nu_w)$ and $\pi_{w_c} = \pi(\mu_{w_c}, \nu_{w_c})$, 
                      where $\mu_w$ and $\mu_{w_c}$ are unramified characters.  
                      Then we have 
                   \begin{align*}
                      \widetilde{I}_v\left(s; W_{\pi, v}, \Phi_v, \varphi_v \right)  
                    = L(s, {\rm As}^+(\pi)_v \otimes\varphi_v ) 
                       \times \begin{cases}  1,     &   ( \nu_w \nu_{w_c} \text { is ramified }),  \\
                                              L(s, \nu_w \nu_{w_c} \varphi_v)^{-1} ,    &   ( \nu_w \nu_{w_c} \text { is unramified }).
                                  \end{cases}  
                   \end{align*}
            \item Suppose that either one of $\pi_w$ and $\pi_{w_c}$ is principal series representation and the other is a special representation.
                     Then we have 
                   \begin{align*}
                      \widetilde{I}_v\left(s; W_{\pi, v}, \Phi_v, \varphi_v \right)  
                    = L(s, {\rm As}^+(\pi)_v \otimes\varphi_v ). 
                  \end{align*}  
            \item Suppose that both $\pi_w={\rm Sp}(\eta_w)$ and $\pi_{w_c}={\rm Sp}(\eta_{w_c})$ are special representations, 
                     where both $\eta_w$ and $\eta_{w_c}$ are unramified characters by the square-freeness of ${\mathfrak N}$. 
                     Then we have 
                               \begin{align*}
                   \widetilde{I}_v\left(s; W_{\pi, v}, \Phi_v, \varphi_v \right)  
                =  L(s, {\rm As}^+(\pi)_v\otimes\varphi_v ) 
                    \times   L(s, \eta^{-1}_w \eta_{w_c})^{-1}.
          \end{align*}
          \end{enumerate}  
 \item Suppose that $v$ is inert in $E/F$. 
          Then we have 
          \begin{align*}
           \widetilde{I}_v\left(s; W_{\pi, v}, \Phi_v, \varphi_v \right)
                = L(s, {\rm As}^+(\pi)_v\otimes\varphi_v )
                  \times \begin{cases}  
                         1,  &   (\pi_v:\text{principal series}), \\
                         L(s, \tau_{E_v/F_v}\eta_v\varphi_v|_{F^\times_v})^{-1}  ,  & (\pi_v={\rm Sp}(\eta_v):\text{special}).
                   \end{cases}      
          \end{align*}  
\end{enumerate}
}
\end{prop}

\begin{rem}
For each $v\mid {\mathfrak N}_F$, suppose either one of the following conditions: 
\begin{itemize}
\item $\omega_{\pi, v}$ is ramified;  
\item  if $\omega_{\pi, v}$ is unramified, 
        then $v=ww_c\mid {\mathfrak N}_F$ splits in $E/F$
        and one of $\pi_w$ and $\pi_{w_c}$ is an unramified principal representation 
                and the other is a special representation.
\end{itemize}
Then Proposition \ref{prop:TameInt} yields that 
\begin{align*}
I_v\left(s; W_{\pi, v}, \Phi_v, \varphi_v \right) 
   = \frac{1}{[{\rm GL}_2({\mathcal O}_{F,v}) : {\mathcal K}_0( {\mathfrak N}_{F,v} ) ]}  
          L(s, {\rm As}^+(\pi_v)\otimes\varphi_v ).
\end{align*}
\end{rem}

\subsection{At an auxiliary place}

Let $v_0$ be the fixed auxiliary place of $F$, which is defined in Definition \ref{def:aux}. 
\begin{prop}\label{prop:AuxInt}
{\itshape
We have 
\begin{align*}
 I_{v_0}\left(s; W_{\pi, v_0}, \Phi_{v_0}, \varphi_{v_0} \right) 
    =   \frac{ q_{v_0} }{ [{\rm GL}_2({\mathcal O}_{F, v_0}) : {\mathcal K}_0(\varpi_{v_0}) ]  } 
         \frac{ L(s, {\rm As}^+(\pi)_{v_0} \otimes \varphi_{v_0})  }{L( 2s, \omega_{\pi,v_0} \varphi^2_{v_0} )}.  
\end{align*}
}
\end{prop}
\begin{proof}
We omit the subscript $v_0$ in this proof. 
Recall that for each integralable function ${\mathcal F}\in L^1({\rm N}_2(F) \backslash {\rm GL}_2(F))$ on ${\rm N}_2(F) \backslash {\rm GL}_2(F)$, 
we have 
\begin{align}\label{eq:IntFor}
  \int_{{\rm N}_2(F )  \backslash {\rm GL}_2(F ) } {\mathcal F}(g){\rm d}g  
  = \frac{ \zeta_{F}(2) }{ \zeta_{F}(1)}   
      \int_{F^\times}  {\rm d}^\times t
      \int_{ F^\times  }  \frac{{\rm d}^\times a}{|a|}
      \int_{ F  }  {\rm d}x
      {\mathcal F}\left( t  \begin{pmatrix} a & 0 \\ 0 & 1 \end{pmatrix} \begin{pmatrix} 0 & 1 \\ -1 & 0 \end{pmatrix} \begin{pmatrix} 1 & x \\ 0 & 1 \end{pmatrix}   \right) 
\end{align} 
(see \cite[Section 3.1.6]{mv10}).
We note that 
\begin{align*}
\Phi\left( (0,1) t  \begin{pmatrix} a & 0 \\ 0 & 1 \end{pmatrix} \begin{pmatrix} 0 & 1 \\ -1 & 0 \end{pmatrix} \begin{pmatrix} 1 & x \\ 0 & 1 \end{pmatrix} \right)
  =& \Phi( -t, -tx ).
\end{align*}
Hence the formula (\ref{eq:IntFor}) shows that $I\left(s; W_{\pi}, \Phi, \varphi \right)$ is given as follows:
\begin{align*}
 I\left(s; W_{\pi}, \Phi, \varphi \right) 
 =&  \int_{ {\rm N}_2(F )  \backslash {\rm GL}_2(F ) }  {\rm d}g
                          W_{\pi}\left( d_{E/F} g     \right)  
                                  \varphi\left( \det  g  \right)  
                                   |\det g|^{   s  } \Phi( (0,1) g,s)  \\
 =& \frac{ \zeta_{F}(2) }{ \zeta_{F}(1)}   
      \int_{ F^\times  }  \frac{{\rm d}^\times a}{|a|} 
                          W_{\pi}\left(   \begin{pmatrix} a & 0 \\ 0 & 1 \end{pmatrix}     \right)  
                          \varphi\left( a \right)
                          |a|^{ s }   \\
 =& \frac{ q }{ [{\rm GL}_2({\mathcal O}_{F}) : {\mathcal K}_0(\varpi) ]  } 
       \frac{ L(s, {\rm As}^+(\pi) \otimes \varphi)  }{L( 2s, \omega_{\pi} \varphi^2 )}.  
\end{align*}
This proves the statement.
\end{proof}

\section{Local integral at $p$-adic places}\label{sec:pInt}

By Proposition \ref{prop:unfoldloc} and (\ref{eq:Itilde}), 
it suffices to compute the following integral for the interpolation formula in Theorem \ref{thm:Main}: 
\begin{align*}
 &  \lambda^{-r}_p c_{r, \alpha, s}  
        \times q^{ [\kappa] r }_p I_p\left( s+n-\alpha+1; W_{\pi, p}, \Phi^{(r)}_p, \varphi_p\right) \\
&=   \lambda^{-r}_p c_{r, \alpha, s}  
     \times  q^{ [\kappa] r}_p     
      \int_{  {\rm N}_2( F_p )  \backslash {\rm GL}_2( F_p )   } {\rm d}g
     W_{\pi,p}\left(  d_{E/F}   g h^{(r)}_\xi \right)  
       \varphi_p(\det g)  |\det g|^{  s+n-\alpha+1  }_p \Phi^{ (r) }_p( (0,1) g)    \\    
&=: I_{r,\varphi}(s+n-\alpha+1).
\end{align*}
Recall that the constants $\alpha_{\pi_p}$ and $\beta_{\pi_p}$ are defined in (Sat1), (Sat2) in Section \ref{sec:CPconj}. 
Then we find that 
\begin{align*}
  \lambda^{-r}_p c_{r, \alpha, s}  
     \times q^{ [\kappa] r}_p    
  =& (\beta_{\pi_p} q^{[\kappa]+1} )^{-r} \times \omega_{\pi, p}( p )^{-r} q^{(n_\alpha+2s)r}_p \times  q^{-r}_p \times q^{ [\kappa] r  }_p    \\
  =& \left(  \beta_{\pi_p} \omega_{\pi,p}( p )  q^{-n_\alpha +(1-2s)}_p   \right)^{ -r }  \\
  =& \left(  \beta_{\pi_p} \omega_{\pi,p}( p )  q^{  3 - 2(s+n-\alpha+1)}_p    \right)^{ -r }.
\end{align*}
Hereafter we drop the subscript $p$ in this section.

\begin{prop}\label{prop:pInt}
{\itshape 
Recall that $\chi_{\alpha_\pi} : F^\times_p \to {\mathbf C}^\times$ is the unramified character such that $\chi_{\alpha_\pi}(\varpi_v) = \alpha_{\pi_v}$ for each $v \in \Sigma_{F, p}$.
Then we have 
\begin{align*}
I_{r, \varphi}(s)
     =  \frac{1}{ [{\rm GL}_2({\mathcal O}_F): {\mathcal K}(p^r)]  }  
      \times  \frac{ \gamma(s, \chi_{\alpha_\pi} \varphi, \psi_F )} { \varphi(\xi^2) \lambda_{E/F}(\psi)^{-1} \gamma(s, {\rm As}^+(\pi)\otimes \varphi, \psi_F)}. 
\end{align*}
In particular, we have 
\begin{align*}
    I_{r, \varphi}(n-\alpha+1)   
  =   \frac{   \widehat{\phi}(\sigma_{\xi^2})^{-1}  \lambda_{E/F}(\psi_F)   }{ [{\rm GL}_2({\mathcal O}_F): {\mathcal K}(p^r)]  }  
      \times {\mathcal E}( {\rm As}^+_{\mathcal M}(\pi) (\phi)_p  )  L(0, {\rm As}^+_{\mathcal M}(\pi) (\phi)_p  ).
\end{align*}
}
\end{prop}
\begin{proof}
Let 
\begin{align*}
  I\left(s;  \left( \varrho( h^{(r)}_\xi )W_\pi\right) \otimes \varphi, \Phi^{(r)} \right)
  = \int_{  {\rm N}_2( F )  \backslash {\rm GL}_2( F )   } {\rm d}g
     W_{\pi}\left(  d_{E/F}   g h^{(r)}_\xi \right)  
       \varphi(\det g)  |\det g|^{  s  } \Phi^{ (r) }( (0,1) g). 
\end{align*}
The functional equation of zeta integrals for Asai $L$-functions (\cite[Appendix, Theorem]{fl93}, \cite[(2.1.2)]{cci}) shows that 
\begin{align*}
    I\left(s;  \left( \varrho( h^{(r)}_\xi )W_\pi\right) \otimes \varphi, \Phi^{(r)} \right) 
    = \gamma_{\rm RS}(s; {\rm As}^+(\pi)\otimes \varphi, \psi_F, \xi  )^{-1}  
         I\left(1-s;  \left(\varrho( h^{(r)}_\xi ) W_\pi\right) \otimes  \varphi^{-1}\omega^{-1}_\pi,  \widehat{\Phi^{(r)}} \right), 
\end{align*}    
where $\gamma_{\rm RS}(s; {\rm As}^+(\pi)\otimes \varphi, \psi_F, \xi  )$ is the $\gamma$-factor of Rankin-Selberg zeta integral. 
Note that \cite[Corollary 2.2]{cci} yields that
\begin{align}\label{eq:gamma}
\begin{aligned}
    \gamma_{\rm RS}(s; {\rm As}^+(\pi)\otimes \varphi, \psi_F, \xi  )
    =& \omega_\pi(\xi)\varphi(\xi^2) | \xi^2 |^{s-\frac{1}{2}} 
       \times \lambda^{-1}_{E/F}(\psi_F) \gamma(s, {\rm As}^+(\pi)\otimes \varphi, \psi_F) \\   
    =& \varphi(\xi^2) \times \lambda^{-1}_{E/F}(\psi_F) \gamma(s, {\rm As}^+(\pi)\otimes \varphi, \psi_F). 
\end{aligned}
\end{align}

Compute the following integral:
\begin{align}\label{eq:p-Int1}
\begin{aligned}
     & I\left( 1-s;  \left(\varrho( h^{(r)}_\xi ) W_\pi\right) \otimes  \varphi^{-1}\omega^{-1}_\pi,  \widehat{\Phi^{(r)}} \right)   \\
  = &         \int_{  {\rm N}_2( F )  \backslash {\rm GL}_2( F )   } {\rm d}g
     W_{\pi }\left(  d_{E/F}   g h^{(r)}_\xi \right)  
       \varphi^{-1} \omega^{-1}_\pi(\det g)  |\det g|^{  1 - s  } \widehat{\Phi^{(r)}} ( (0,1) g).     
\end{aligned}
\end{align}
Let us use the following coordinate:
\begin{align*}
   g = t \begin{pmatrix} a & 0 \\ 0 & 1 \end{pmatrix}
           \begin{pmatrix} 0 & -1 \\ 1 & 0 \end{pmatrix}
           \begin{pmatrix} 1 & -x \\ 0 & 1 \end{pmatrix}
           \begin{pmatrix} 0 & -1 \\ 1 & 0 \end{pmatrix}  
      = t \begin{pmatrix} a & 0 \\ 0 & 1 \end{pmatrix}
           \begin{pmatrix} 1 & 0 \\ x & 1 \end{pmatrix}
       = t \begin{pmatrix} a & 0 \\ x & 1 \end{pmatrix}.      
\end{align*}
Note that $(0,1) g = (tx , t)$.
Then the formula (\ref{eq:IntFor}) shows that 
the integral (\ref{eq:p-Int1}) is given by
\begin{align}
\label{eq:p-Int2}
\begin{aligned}
& \frac{\zeta_F(2)}{\zeta_F(1)}
       \int_{F^\times} \frac{{\rm d}^\times a}{|a|}  \int_{F} {\rm d}x        
       W_{\pi}\left(   d_{E/F}    \begin{pmatrix} a & 0 \\ x & 1 \end{pmatrix}
           \begin{pmatrix} p^r & \xi \\ 0 & 1  \end{pmatrix}          \right) 
        \omega^{-1}_\pi\varphi^{-1}( a ) | a |^{1-s}    \\    
   & \times  \int_{F^\times} 
      \widehat{ \Phi^{(r)} }( tx,t ) \omega^{-1}_\pi\varphi^{-2} |\cdot|^{2-2s}(t)  {\rm d}^\times t   \\   
= &\frac{\zeta_F(2)}{\zeta_F(1)} q^{-r}  \int_{F^\times} \frac{{\rm d}^\times a}{|a|}  \int_{F} {\rm d}x        
       W_{\pi}\left(   d_{E/F}    \begin{pmatrix} a & 0 \\ x p^r & 1 \end{pmatrix}
           \begin{pmatrix} p^r & \xi \\ 0 & 1  \end{pmatrix}          \right) 
        \omega^{-1}_\pi\varphi^{-1}( a ) | a |^{1-s}    \\    
   & \times  \int_{F^\times} 
      \widehat{ \Phi^{(r)} }( t x p^r, t ) \omega^{-1}_\pi\varphi^{-2} |\cdot|^{2-2s}(t)  {\rm d}^\times t.  
\end{aligned}
\end{align}
Note that 
\begin{align} \label{eq:GamTrans}
     \begin{pmatrix} a & 0 \\ x p^r & 1 \end{pmatrix}
     \begin{pmatrix} p^r & \xi \\ 0 & 1  \end{pmatrix}     
     = \begin{pmatrix} a p^r & a \xi \\ 0 & 1  \end{pmatrix}
         \begin{pmatrix} 1 - p^r x \xi   & -x\xi^2  \\  x p^{2r} & 1 + x \xi  p^r \end{pmatrix}.
\end{align}
Since $\widehat{ \Phi^{(r)} }( t x p^r, t )\neq 0$ if and only if $\widehat{ \Phi^{(r)} }( t x p^r, t )= 1$, which is also equivalent to the following condition: 
\begin{align}\label{eq:PhiHatCond}
   tx p^r \in {\mathcal O}, \quad 
   t\in p^{-r} + {\mathcal O},  \quad 
   ({\mathcal O}:={\mathcal O}_{F,p}),
\end{align}
we find that the second matrix in the right-hand side of (\ref{eq:GamTrans}) is an element in ${\mathcal K}_1(p^r)$. 
Hence the integral (\ref{eq:p-Int2}) becomes
\begin{align}
\label{eq:p-Int3}
\begin{aligned}
   & \frac{\zeta_F(2)}{\zeta_F(1)} q^{-r}  \int_{F^\times} \frac{{\rm d}^\times a}{|a|}  \int_{F} {\rm d}x        
       \psi_E(a\xi (2\xi)^{-1})
       W_{\pi}\left(  \begin{pmatrix} a p^r & 0 \\ 0 & 1  \end{pmatrix}     \right) 
        \omega^{-1}_\pi\varphi^{-1}( a ) | a |^{1-s}    \\    
     & \times  \int_{F^\times} 
         \widehat{ \Phi^{(r)} }( t x p^r, t ) \omega^{-1}_\pi\varphi^{-2} |\cdot|^{2-2s}(t)  {\rm d}^\times t.   
\end{aligned}
\end{align}
Note also that  the condition (\ref{eq:PhiHatCond}) is equivalent to 
$t\in p^{-r} + {\mathcal O}$ and $x\in {\mathcal O}$
and that 
 $ \psi_E(a\xi (2\xi)^{-1}) = \psi_F(a)$. 
Recall that the explicit formula of the Whittaker function $W_\pi$ in (\ref{eq:WhittStab}): 
\begin{align*}
    W_{\pi}\left(  \begin{pmatrix} a  & 0 \\ 0 & 1  \end{pmatrix}     \right)
    =\chi_{\beta_\pi}(a)  |a|^\frac{1}{2}_E {\mathbb I}_{\mathcal O}(a).
\end{align*}
Hence the integral (\ref{eq:p-Int3}) becomes 
\begin{align*}
  & \frac{\zeta_F(2)}{\zeta_F(1)} q^{-r}  \int_{F^\times} \frac{{\rm d}^\times a}{|a|} 
       \psi_F(a)
       \chi_{\beta_\pi}(a p^r) | a p^r |^{\frac{1}{2}}_E  {\mathbb I}_{\mathcal O}( a p^r )
        \omega^{-1}_\pi\varphi^{-1}( a ) | a |^{1-s}        
      \times  \int_{ p^{-r} (1+p^r {\mathcal O})}  \omega^{-1}_\pi\varphi^{-2} |\cdot|^{2-2s}(t)  {\rm d}^\times t \\
 = & \frac{\zeta_F(2)}{\zeta_F(1)} q^{-r}  \int_{F^\times} {\rm d}^\times a 
       \psi_F(a p^{-r})
       \chi_{\beta_\pi}(a ) | a  |  {\mathbb I}_{\mathcal O}( a  )
        \omega^{-1}_\pi \varphi^{-1}( a p^{-r} ) | a p^{-r} |^{-s}        
      \times \omega_\pi \varphi^2 ( p^r ) q^{2r(1-s)} 
                    {\rm vol}( 1 + p^r {\mathcal O} , {\rm d^\times}t)  \\ 
 = & \frac{\zeta_F(2)}{\zeta_F(1)} q^{-r}  \int_{F^\times} {\rm d}^\times a 
        \psi_F(a p^{-r})
       \chi_{\beta_\pi}(a )   {\mathbb I}_{\mathcal O}( a )
        \omega^{-1}_\pi\varphi^{-1}( a ) | a |^{1-s}  
          {\rm d}^\times a   
                  \times  \omega_{\pi}\varphi( p^r ) q^{-rs}       
      \times \omega_\pi \varphi^2 ( p^r ) q^{2r(1-s)} 
               q^{-r} \zeta_F(1)  \\ 
=    &   \int_{F^\times} \widehat{{\mathbb I}_{ p^{-r}+\mathcal O}}( a )
        \chi^{-1}_{\alpha_\pi} \varphi^{-1}( a ) | a |^{1-s}   
          {\rm d}^\times a       
        \times \omega^2_\pi \varphi^3 ( p^r ) q^{ -3sr}  \zeta_F(2),        
\end{align*}
where we use $\widehat{{\mathbb I}_{ p^{-r}+\mathcal O}}( a ) 
                          = \psi_F(a p^{-r})  {\mathbb I}_{\mathcal O}( a )$ in the last identity.
Then \cite[Proposition 3.1.5]{bu98} shows that the above integral is given by 
\begin{align*}
\gamma(s, \chi_{\alpha_\pi} \varphi,  \psi_F)
    \cdot \int_{F^\times} 
       {\mathbb I}_{ p^{-r}+\mathcal O}( a )
        \chi_{\alpha_\pi}  \varphi( a ) | a |^{s}   
          {\rm d}^\times a   
= \gamma(s, \chi_{\alpha_\pi} \varphi, \psi_F) 
     \cdot \chi_{\alpha_\pi} \varphi( p^{-r} ) q^{rs} \times \zeta_F(1) q^{-r}.       
\end{align*}

Summarizing the above computation, we have 
\begin{align*}
   I_{r,\varphi}(s)
=& \left(  \beta_\pi \omega_{\pi} ( p )  q^{ 3-2s }    \right)^{-r}   
   \times \gamma_{\rm RS}(s, {\rm As}^+(\pi)\otimes \varphi, \psi_F, \xi)^{-1}  \\
  & \quad \quad  \times \omega^2_\pi \varphi^3 ( p^r ) q^{-3sr } \zeta_F(2) 
   \times \gamma(s, \chi_{\alpha_\pi} \varphi, \psi_F )  \chi_{\alpha_\pi} \varphi( p^{-r} ) q^{rs}  \zeta_F(1) q^{-r}  \\
=& \varphi^2( p^r )  
      \times q^{-4r}\zeta_F(1)\zeta_F(2)  
      \times  \frac{ \gamma(s, \chi_{\alpha_\pi} \varphi, \psi_F )} { \varphi(\xi^2) \lambda_{E/F}(\psi_F)^{-1}  \gamma(s, {\rm As}^+(\pi)\otimes \varphi, \psi_F)}, 
\end{align*}
where we used (\ref{eq:gamma}) in the last identity.  
Recall that $q^{-4r}\zeta_F(1)\zeta_F(2)=[{\rm GL}_2({\mathcal O}_F): {\mathcal K}(p^r)]^{-1}$.
Since $\varphi$ is a Hecke character of finite-order and of a $p$-power conductor, we find that $\varphi(p)=1$.
This shows the proposition.
\end{proof}

\section{Local integral at infinite places}\label{sec:InfInt}

The aim of this section is to compute the following local integral at $\infty$: 
\begin{align*}
   & I_\infty(s; W_{\pi,\infty},  \Phi_\infty, \varphi_\infty)   \\
=&  \sum_{-n-1 \leq i \leq n+1}   
    C(\alpha, i)
           \int_{{\rm N}_2(F_\infty)  \backslash {\rm GL}_2(F_\infty)} {\rm d} g
        W^i_{\pi,\infty}\left( d_{E/F}   g \right)  
        \varphi_\infty(\det g)  
        |\det g|^{s}_\infty
        \Phi_\infty( (0,1) g ).  
\end{align*}
The explicit formula of $I_\infty(s; W_{\pi,\infty},  \Phi_\infty, \varphi_\infty)$ is conjectured by Ghate (\cite[page 629, Conjecture 1]{gh99})
and it is proved in \cite[Theorem 4.1]{ls14}. 
However their description is given by the classical language, 
although we use the adelic formulation. 
It is straightforward to check the compatibility between their calculation and the one in this paper, but also not clear for the reader. 
In fact, although we show that the local integral $I_\infty(s; W_{\pi,\infty},  \Phi_\infty, \varphi_\infty)$ is written by the modified Euler factor at $\infty$ in Proposition \ref{prop:InfInt}, 
Ghate and Lanphier-Skogman do not mention it in their paper. 
Besides Ghate and Lanphier-Skogman compute the integral if $\varphi_\infty$ is trivial. 
Hence it should be better to explain this compatibility for the reader's convenience.  

In Section \ref{sec:ghconj}, we briefly compare the terminology in \cite{gh99} and in this paper. 
It will clarify the compatibility between the global method in \cite{gh99} and the local method in this paper. 
In Section \ref{sec:LocIntInf}, we deduce the explicit formula $I_\infty(s; W_{\pi,\infty},  \Phi_\infty, \varphi_\infty)$ 
by using \cite[Theorem 4.1]{ls14} (see Theorem \ref{Ghconj}) and by computing local integrals. 
The result (see Proposition \ref{prop:InfInt}) is written by the modified Euler factor at $\infty$, which we defined in Section \ref{sec:CPconj}.

We should make a remark on \cite[Appendix]{lw}. 
Loeffler-Williams also mention the connection between Ghate's conjecture and the modified Euler factor at $\infty$ by using the classical language,  
and  they describe their interpolation formula by the critical values in the left-half of the critical range in (\ref{eq:critrange}) (see \cite[Theorem 1.1]{lw}). 
However, comparing with the local calculation at the $p$-adic places (Proposition \ref{prop:pInt}), 
   it seems to be better to describe the interpolation formula in terms of the right-half of the critical range in this paper.          
Together with the adelic method, 
 the description in this section will give a more concise and direct explanation of the interpolation formula for $p$-adic Asai $L$-functions.

\subsection{Ghate's conjecture}\label{sec:ghconj}

We recall Ghate's conjecture (\cite[Conjecture 1]{gh99}), which is proved in \cite[Theorem 4.1, Theorem A.1]{ls14}. 
We briefly explain the relation between Ghate's computation in \cite{gh99} and ours for the reader's convenience. 

Recall the constant in (\ref{eq:defC(a,i)}):
\begin{align*}
       C(\alpha,i) 
= &  [ \Upsilon_\alpha ( v_i(0)),  (X-\sqrt{-1}Y)^{2n-2\alpha}   ]_{2n-2\alpha}  
        + \sqrt{-1} [  \Upsilon_\alpha (v_i(-2)),  (X-\sqrt{-1}Y)^{2n-2\alpha}   ]_{2n-2\alpha}     \\  
    & \quad \quad  - \sqrt{-1} [ \Upsilon_\alpha (v_i(2)),  (X-\sqrt{-1}Y)^{2n-2\alpha}   ]_{2n-2\alpha},  
\end{align*}
which appears in the local integral $I_\infty(s; W_{\pi,\infty},  \Phi_\infty, \varphi_\infty)$.  
We firstly prove that this constant $C(\alpha,i)$ coincides with the constant $c(m,\alpha)$ in \cite[page 628]{gh99} up to a simple factor as follows. 

Recall that $v(j) (j=0, \pm 2)$ is a certain polynomial which is appeared in the explicit formula in Eichler-Shimura map (see identity (\ref{eq:v(j)def}) for the definition)
and that $\Upsilon_\alpha$ is defined in (\ref{eq:upsdef}). 
Hence the values $[ \Upsilon_\alpha(v_i(j)), (X-\sqrt{-1}Y)^{2n-2\alpha} ]_{2n-2\alpha}$ $(j=0, \pm 2, -n-1 \leq i \leq n+1)$
   are directly calculated by using the pairing $[\cdot, \cdot]_{2n-2\alpha}$ on $L(2n-2\alpha; {\mathbf C})$ as follows:

\begin{lem}\label{lem:upv}
{\itshape
For $-n-1 \leq i \leq n+1$, we have
\begin{align*}
    & [ \Upsilon_\alpha(v_i(-2)), (X-\sqrt{-1}Y)^{2n-2\alpha} ]_{2n-2\alpha} \\
 = & (-1)^{n-\alpha} \sqrt{-1}^{\alpha+i-1}  \binom{2n+2}{n+1-i}^{-1} \binom{n}{\alpha}^2
           \sum^\alpha_{t=0}  (-1)^{t} \binom{\alpha}{t}  \binom{2n-2\alpha}{ n+1-i-2t },   \\
   &   [ \Upsilon_\alpha(v_i(0)), (X-\sqrt{-1}Y)^{2n-2\alpha} ]_{2n-2\alpha}  \\
 = &      2(-1)^{n-\alpha} \sqrt{-1}^{\alpha+i}  \binom{2n+2}{n+1-i}^{-1} \binom{n}{\alpha}^2
           \sum^\alpha_{t=0}  (-1)^{t} \binom{\alpha}{t} \binom{2n-2\alpha}{ n-i-2t },   \\
    & [ \Upsilon_\alpha(v_i(2)), (X-\sqrt{-1}Y)^{2n-2\alpha} ]_{2n-2\alpha} \\
 = &    (-1)^{n-\alpha} \sqrt{-1}^{\alpha+i+1}  \binom{2n+2}{n+1-i}^{-1} \binom{n}{\alpha}^2
           \sum^\alpha_{t=0}  (-1)^{t} \binom{\alpha}{t} \binom{2n-2\alpha}{ n-1-i-2t}.
\end{align*}
}
\end{lem}

By using Lemma \ref{lem:upv} and the definition of $C(\alpha, i)$, we deduce that 
\begin{align}\label{eq:C(a,i)}
\begin{aligned}
      C(\alpha,i) 
=& (-1)^{n} (-1)^{ \frac{i-\alpha}{2}} \binom{n}{\alpha}^2 \binom{2n+2}{n+1-i}^{-1}
       \times  \sum_t  (-1)^{t} \binom{\alpha}{t} \binom{2n-2\alpha+2}{n-2t+i+1}.
\end{aligned}
\end{align}

Next, we compare $C(\alpha, i)$ with constants $c(m,\alpha)$ in the notation in \cite[page 628]{gh99}.  
This constant $c(m,\alpha)$ is described in the following form: 

\begin{lem}(\cite[page 1103, Lemma A.3]{ls14})
{\itshape
We have
\begin{align*}
 c(m,\alpha) =  \frac{(-1)^{\frac{3n+\alpha-m+3}{2}}  }{2} \binom{n}{m}^2   \sum^m_{t=0} (-1)^t \binom{m}{t}  \binom{2n-2m+2}{ \alpha-2t }.
\end{align*}
}
\end{lem}

We compare the notation in \cite{gh99} with ours. 
Replacing $m$ (resp. $\alpha, s+n-\alpha+1$) in \cite{gh99} by $\alpha$ (resp. $n+i+1, s$) in this paper, we find that
\begin{align*}
 \bullet \ &   c(m,\alpha)
 =  \frac{(-1)^{\frac{3n+\alpha-m+3}{2}}  }{2} \binom{n}{m}^2   
    \sum^m_{t=0} (-1)^t \binom{m}{t}  \binom{2n-2m+2}{ \alpha-2t }  \\
&\longmapsto
  \frac{(-1)^{\frac{3n+(n+i+1)-\alpha+3}{2}}  }{2} \binom{n}{\alpha}^2   
    \sum^\alpha_{t=0} (-1)^{t} \binom{\alpha}{\alpha-t}  \binom{2n-2\alpha+2}{ n+i+1 -2t }  \\
& \quad \quad = \frac{(-1)^{\frac{i-\alpha}{2}}  }{2} \binom{n}{\alpha}^2   
            \sum^\alpha_{t=0} (-1)^t \binom{\alpha}{t}  \binom{2n-2\alpha+2}{ n-2t+i+1 }
         = \frac{(-1)^n}{2} C(\alpha, i) \binom{2n+2}{n+1-i}, \\
  \bullet \  &  \Gamma\left( \frac{s+n+1-m+\alpha}{2} \right)
   \Gamma\left( \frac{s+3n+3-m-\alpha}{2} \right)   \\
 & \longmapsto 
 \Gamma\left( \frac{ (s-n+\alpha-1)+n+1-\alpha+(n+i+1)}{2} \right)
   \Gamma\left( \frac{(s-n+\alpha-1)+3n+3-\alpha-(n+i+1)}{2} \right)   \\
& \quad \quad  =  \Gamma\left( \frac{s+n+1+i}{2} \right)
     \Gamma\left( \frac{s+n+1-i}{2} \right),  \\
 \bullet \ &  \sum^{2n+2}_{\substack{\alpha=0 \\ \alpha\equiv n+1+m \text{ mod }2 } } 
 \longmapsto 
 \sum^{2n+2}_{\substack{n+i+1=0 \\ n+i+1 \equiv n+1+\alpha \text{ mod }2 } } 
  = \sum^{ n+1}_{\substack{i=-n-1 \\ i \equiv \alpha \text{ mod }2 } } .
\end{align*}
This clarifies the relation between the calculation in \cite{gh99} and the one in this paper. 
Hence we can find that the following formula, which is essentially used in the computation of $I_\infty(s; W_{\pi,\infty},  \Phi_\infty, \varphi_\infty)$:

\begin{thm}\label{Ghconj}
(\cite[page 629, Conjecture 1]{gh99}, \cite[Theorem 4.1]{ls14})
{\itshape
We have
\begin{align}\label{eq:Ghconj} 
\begin{aligned}
 &  \frac{(-1)^n}{2}
      \sum^{ n+1}_{\substack{i=-n-1 \\ i \equiv \alpha \ {\rm mod}\ 2 } }
      C(\alpha, i)\binom{2n+2}{n+1-i}
     \Gamma\left( \frac{s+n+1+i}{2} \right)
     \Gamma\left( \frac{s+n+1-i}{2} \right)  \\
= &     \frac{(-1)^\alpha\sqrt{\pi}\binom{n}{\alpha}^2}{2^{s-n+\alpha-1}} 
          \frac{\Gamma(\frac{s+n-\alpha+1}{2})}{\Gamma(\frac{s-n+\alpha}{2})}
          \frac{\Gamma(s)\Gamma(s+n+1)}{\Gamma(s+n-\alpha+1)}.                     
\end{aligned}
\end{align}
}
\end{thm}

\subsection{Local integral at $\infty$}\label{sec:LocIntInf}

In this subsection, we compute the following summation of local integrals:
\begin{align}\label{eq:IInf}
    I_\infty(s; W_{\pi,\infty}, \Phi_\infty, \varphi_\infty)  
   =&\sum_{-n-1 \leq i \leq n+1}   
    C(\alpha, i)  
         \int_{ {\rm N}_2( F_\infty )  \backslash {\rm GL}_2( F_\infty )} {\rm d} g
        W^i_{\pi, \infty}\left( d_{E/F}   g \right)  
        \varphi_\infty(\det g)
        |\det g|^{s}_\infty
        \Phi( (0,1) g ).
\end{align}
The argument in Section \ref{sec:cup} and \ref{sec:AsaiInt} shows that 
 the integrand in (\ref{eq:IInf}) is invariant under the right translation by $C_{\infty, +}={\rm SO}_2( F_\infty )$.
Hence we find that  $I(s; W_{\pi,\infty}, \Phi_\infty, \varphi_\infty)$ is given by  
\begin{align*}
   &  \sum_{-n-1 \leq i \leq n+1}   
    C(\alpha, i) \cdot 2^{2r_F}
       \int_{ F^\times_\infty } \frac{ {\rm d}^\times a }{ |a| }
         \int_{ F^\times_{\infty, +} } {\rm d}^\times t 
        W^i_{\pi, \infty}\left( d_{E/F}   t \begin{pmatrix} a & 0 \\ 0 & 1 \end{pmatrix}  \right)  
        \varphi  | \cdot |^{ s} (at^2) 
        \Phi_\infty \left( (0,1) t \begin{pmatrix} a & 0 \\ 0 & 1 \end{pmatrix} \right)    \\
   =& \sum_{-n-1 \leq i \leq n+1}   
    C(\alpha, i) \cdot 2^{2r_F}
         \int_{ F^\times_\infty } \frac{ {\rm d}^\times a }{ |a| }
        W^i_{\pi, \infty}\left( d_{E/F}    \begin{pmatrix} a & 0 \\ 0 & 1 \end{pmatrix}  \right)  
        \varphi ( a ) | a |^{ s }  
      \times   \int_{ F^\times_{\infty, +} }  {\rm d}^\times t 
            |t|^{2 s}
        \Phi_\infty \left( 0,t  \right). 
\end{align*}
We compute this integral for each $v\mid \infty$. 
Hereafter we always write $F_v$ to be ${\mathbf R}$ and we omit the subscript $v$.
We prepare the following two easy lemmas:

\begin{lem}\label{lem:WInft}
{\itshape 
We have 
\begin{align*}
    \int_{ {\mathbf R}^\times_{+} }  {\rm d}^\times t 
            |t|^{2 s} 
        \Phi \left( 0,t  \right)
   =& (-1)^{ n - \alpha + 1}   2^{s - n + \alpha-3} \times   \Gamma_{\mathbf C}(s+n-\alpha+1).
\end{align*}
}
\end{lem}
\begin{proof}
The left-hand side of the statement is given by  
\begin{align*}
      \int_{ {\mathbf R}^\times_{ + } }  {\rm d}^\times t 
         |t|^{2s}  
        \Phi \left( 0,t \right)    
 = &   \int_{ {\mathbf R}^\times_{ + } }  {\rm d}^\times t 
         |t|^{2s}    
          2^{-k_\alpha} (\sqrt{-1} t)^{k_\alpha} e^{-\pi t^2}    \\
 = & 2^{ -k_\alpha  } \sqrt{-1}^{k_\alpha}
        \times  \frac{1}{2} \pi^{- (s+n-\alpha + 1 ) } \Gamma\left( s + n-\alpha+1   \right)  \\
 = &  ( -1 )^{n-\alpha+1}
        \times 2^{s - n + \alpha-3} \times 2(2\pi)^{- (s+n-\alpha+1 ) }  \Gamma ( s + n-\alpha+1  ).
\end{align*}
This proves the statement.   
\end{proof}

\begin{lem}\label{lem:WInfa}
{\itshape 
Write $2\xi= \sqrt{-D_{E/F}} \in E$.
Suppose that $\varphi(-1) = (-1)^{n-\alpha}$.
Then we have 
\begin{align*}
   \int_{{\mathbf R}^\times } \frac{ {\rm d}^\times a }{ |a| }
        W^i_{\pi} \left( d_{E/F}    \begin{pmatrix} a & 0 \\ 0 & 1 \end{pmatrix}  \right)  
        \varphi ( a ) | a |^{ s   }    
=& (-1)^{k-1} D^{ -\frac{1}{2}(s - 1 )  }_{E/F} \binom{2n+2}{n-i+1} \times \frac{  1 + (-1)^{\alpha - i}  }{2}  \\
   & \quad   \times 2^2 (2\pi)^{-(s+ n + 1 )}
      \Gamma\left(   \frac{  s+ n + 1 + i}{2} \right) 
      \Gamma\left( \frac{ s + n + 1 -i }{2} \right).
\end{align*}
}
\end{lem}
\begin{proof}
We recall that
\begin{align*}
  d_{E/F} = & \begin{pmatrix} \sqrt{-D_{E/F}} & 0 \\ 0 & 1   \end{pmatrix},  \\ 
        W^i_\pi\left(  d_{E/F} \begin{pmatrix}  a & 0 \\ 0 & 1 \end{pmatrix}    \right) 
 = &  2^3 
         (-1)^{k-i-1}  
         \binom{2n+2}{n-i+1} e^{\sqrt{-1}  (\frac{\pi}{2}) i }  
         \sqrt{-1}^i  \left( \sqrt{D_{E/F}} a \right)^{ n+2 }K_i \left(4\pi\sqrt{D_{E/F}} a \right)  \\ 
 = &  2^3 (-1)^{k-1} D^{\frac{n}{2}+1}_{E/F} 
       \times  \binom{2n+2}{n-i+1}  
           a^{n+2} K_i \left(4\pi D^\frac{1}{2}_{E/F}  a \right)   
\end{align*}
for $0<a$.  
Hence the formula
\begin{align*}
 W^i_{\pi}\left( \begin{pmatrix} - a & 0 \\ 0 & 1 \end{pmatrix} \right)  
   = W^i_{\pi}\left( \begin{pmatrix} - a & 0 \\ 0 & 1 \end{pmatrix} \right) = (-1)^{n+i} W^i_{\pi}\left( \begin{pmatrix}  a & 0 \\ 0 & 1 \end{pmatrix} \right),   
\end{align*}
which follows from Definition \ref{adeliccusp} \ref{adeliccusp(ii)} \ref{adeliccusp(iii)},
shows that 
\begin{align*}
 & \int_{{\mathbf R}^\times } \frac{ {\rm d}^\times a }{ |a| }
        W^i_{\pi}\left( d_{E/F}    \begin{pmatrix} a & 0 \\ 0 & 1 \end{pmatrix}  \right)  
        \varphi ( a ) | a |^{ s  }   \\
=&2^3 (-1)^{k-1} D^{ \frac{n}{2}+1 }_{E/F}  \binom{2n+2}{n-i+1} 
    \times  (1+(-1)^{n+i} \varphi(-1))
                 \int^\infty_0 \frac{ {\rm d}^\times a }{ |a| }
         a^{n+2} K_i \left(4\pi D^\frac{1}{2}_{E/F}  a \right)  
        a^{ s  }    \\ 
=&2^3 (-1)^{k-1} D^{\frac{n}{2}+1}_{E/F}  \binom{2n+2}{n-i+1} 
    \times \left( 1 + (-1)^{ n+i} \varphi(-1)  \right) \times 
    \int^\infty_0  
          K_i \left(4\pi D^\frac{1}{2}_{E/F}  a \right)  
         a^{ s  }
        {\rm d}^\times a. 
\end{align*}
By the assumption, 
we have $\varphi(-1) = (-1)^{n-\alpha}$.
Recall a well-known formula (\cite[page 91]{mos66}):
\begin{align*}
  \int^\infty_0 K_\nu( \mu a ) a^s \frac{{\rm d} a}{ a} 
  = 2^{s-2} \mu^{-s} 
      \Gamma\left(   \frac{s+ \nu}{2} \right) 
      \Gamma\left( \frac{s- \nu}{2} \right), 
      \quad ({\rm Re}(s\pm \nu)>0).      
\end{align*}
Hence the above integral is equal to 
\begin{align*}
 &2^3 (-1)^{k-1} D^{\frac{n}{2}+1}_{E/F}  \binom{2n+2}{n-i+1}   \left( 1 + (-1)^{\alpha - i} \right)  \\
  & \times 
     2^{ (s + n + 1)-2}  
     \left( 4\pi D^{\frac{1}{2}}_{E/F} \right)^{-(s + n + 1)}
      \Gamma\left(   \frac{  s + n + 1 + i}{2} \right) 
      \Gamma\left( \frac{ s + n + 1 -i }{2} \right)  \\
 =& (-1)^{k-1} D^{ \frac{n}{2}+1 -\frac{1}{2}(s + n + 1)  }_{E/F} \binom{2n+2}{n-i+1} \times \frac{  1 + (-1)^{ \alpha - i }  }{2} \\
   &   \times 2^2 (2\pi)^{-(s + n + 1)}
      \Gamma\left(   \frac{  s+ n + 1 + i}{2} \right) 
      \Gamma\left( \frac{ s + n +1 -i }{2} \right).
\end{align*}
This proves the statement.
\end{proof}

By using Theorem \ref{Ghconj} and Lemma \ref{lem:WInft}, \ref{lem:WInfa}, we deduce the explicit formula for $I(n-\alpha+1; W_{\pi,\infty}, \Phi_\infty, \varphi_\infty )$ as follows:

\begin{prop}\label{prop:InfInt}
{\itshape 
Let $c_\infty(n, \alpha, \xi)$ be the constant in (\ref{eq:cnax}).
Suppose that $\varphi(-1_\sigma) = (-1)^{n_\sigma- \alpha_\sigma}$ for each $\sigma \in \Sigma_{F, \infty}$.
Then we have 
\begin{align*}
   I_\infty(n-\alpha+1; W_{\pi,\infty}, \Phi_\infty, \varphi_\infty )
   =  c_\infty(n, \alpha, \xi) \times   {\mathcal E}_\infty( {\rm As}^+_{\mathcal M}(\pi)( \phi ) )     L_\infty( 0, {\rm As}^+_{\mathcal M}(\pi)( \phi ) ).
\end{align*}
}
\end{prop}
\begin{proof}
We drop the subscript $\infty$ in this proof.
Lemma \ref{lem:WInft} and \ref{lem:WInfa} show that 
\begin{align*}
  &   I (s; W_\pi, \Phi, \varphi)   
          \times \left(  2^{2r_F} \times (-1)^{ n - \alpha + 1}   2^{s - n + \alpha -3}  \times   \Gamma_{\mathbf C}(s+n-\alpha+1)      
    \times   (-1)^{k-1} D^{ -\frac{1}{2}(s - 1 )  }_{E/F}  \right)^{-1} \\
=& \sum_{ \substack{ -n-1\leq i \leq n+ 1   \\ i  \equiv \alpha \ {\rm mod} \ 2  } }   
           C(\alpha, i)
    \times \binom{2n+2}{n-i+1} 
2^2 (2\pi)^{-(s+ n + 1)}
      \Gamma\left(   \frac{  s+ n + 1 + i}{2} \right) 
      \Gamma\left( \frac{ s + n  + 1 -i }{2} \right)  \\
=&2(-1)^{n}
    \times  2^2 (2\pi)^{-(s+ n +1)}  \\
  & \times   \sum_{ \substack{ -n-1\leq i \leq n+ 1   \\  i  \equiv \alpha \ {\rm mod} \ 2 } }   
       \frac{(-1)^n}{2} C(\alpha, i) \binom{2n+2}{n+1-i}
      \Gamma\left(   \frac{  s+ n + 1 + i}{2} \right) 
      \Gamma\left( \frac{ s + n + 1 -i }{2} \right).  
\end{align*}
Hence Theorem \ref{Ghconj} shows that 
\begin{align*}
      I (s; W_\pi, \Phi, \varphi)  
  = & 2^{2r_F} \times  (-1)^{ n - \alpha + 1}   2^{s-n+\alpha-3} 
            \times   \Gamma_{\mathbf C}(s+ n-\alpha+1  )       
            \times   (-1)^{k-1} D^{ -\frac{1}{2}(s - 1 )  }_{E/F}    \\
    & \times 2(-1)^{n}
       \times  2^2 (2\pi)^{-(s+n + 1 ) } \\  
    & \times      \frac{(-1)^\alpha\sqrt{\pi}\binom{n}{\alpha}^2}{2^{s-n+\alpha-1}} 
          \frac{\Gamma(\frac{s+n-\alpha+1}{2})}{\Gamma(\frac{s-n+\alpha}{2})}
          \frac{\Gamma(s)\Gamma(s+n+1)}{\Gamma(s+n-\alpha+1)} \\    
  = &   (-1)^{n} 
        \times  2 (2\pi)^{- (2s+2n-\alpha + 2) }    
         \times 2^{2r_F}D^{ -\frac{1}{2}(s - 1 )  }_{E/F}   
        \times       \binom{n}{\alpha}^2 
          \frac{\Gamma(\frac{s+n-\alpha+1}{2})}{ \pi^{-\frac{1}{2}} \Gamma(\frac{s-n+\alpha}{2})}
          \Gamma( s )\Gamma(s+n+1) \\   
  = &  (-1)^{n} 
         \times 2^{2r_F}D^{ -\frac{1}{2}( s - 1 )  }_{E/F}   
        \times       \binom{n}{\alpha}^2 
          \frac{ 2 (2\pi)^{ -(n-\alpha+1)} \Gamma(\frac{s+n-\alpha+1}{2})}{ \pi^{-\frac{1}{2}} \Gamma(\frac{s-n+\alpha}{2})}
        \times  \Gamma_{\mathbf C}( s ) 
                   \Gamma_{\mathbf C}( s + n + 1 ).   
\end{align*}
By using $\Gamma_{\mathbf R}(s) \Gamma_{\mathbf R}(s+1) = \Gamma_{\mathbf C}(s)$, 
 we find that  
\begin{align*}
       I (n-\alpha+1; W_\pi, \Phi, \varphi) 
   = &(-1)^{n} 
         \times 2^{2r_F}D^{ -\frac{1}{2}( n-\alpha )  }_{E/F} \times       \binom{n}{\alpha}^2  
       \times \Gamma_{\mathbf C}(n-\alpha+1)  \Gamma_{\mathbf C}(n-\alpha+1) 
                   \Gamma_{\mathbf C}(2n-\alpha+2)  \\
   = &(-1)^{n} 
         \times 2^{2r_F}D^{ -\frac{1}{2}( n-\alpha )  }_{E/F}
          \times       \binom{n}{\alpha}^2  
         \times \Gamma_{\mathbf R}(n-\alpha+1)^2  \Gamma_{\mathbf R}(n-\alpha+2)^2 
                   \Gamma_{\mathbf C}(2n-\alpha+2).  
\end{align*}
The well-known formula:
\begin{align*}
  \Gamma_{\mathbf R}(1+s)\Gamma_{\mathbf R}(1-s) = \frac{1}{ \sin \left( \pi\frac{s+1}{2}  \right) }    
\end{align*}
shows that 
\begin{align*}
\Gamma_{\mathbf R}(n-\alpha+1)^2 
   =& \left(     \sqrt{-1}^{n-\alpha}  \Gamma_{\mathbf R}(1-(n-\alpha))^{-1}  \right)^2,
       \quad \text{if $n-\alpha$ is even}, \\
\Gamma_{\mathbf R}(n-\alpha+2)^2 
   =& \left(  \sqrt{-1}^{n-\alpha+1}  \Gamma_{\mathbf R}(1-(n-\alpha+1))^{-1}  \right)^2,
       \quad \text{if $n-\alpha$ is odd}.
\end{align*}
Hence $I (n-\alpha+1; W_\pi, \Phi, \varphi)$ is equal to 
\begin{align*}
  2^{2r_F}D^{ -\frac{1}{2}( n-\alpha ) }_{E/F}       \binom{n}{\alpha}^2
\times \sqrt{-1}^{-\alpha} (-1)^{n-\alpha}
\times \begin{cases} 
             \displaystyle  \sqrt{-1}^{-(2n-\alpha)} \frac{ \Gamma_{\mathbf C}(s+2n-\alpha+2) \Gamma_{\mathbf R}(n-\alpha+2)^2
                             }{   \Gamma_{\mathbf R}( 1-( n-\alpha)  )^2 },   &  (n-\alpha:\text{even}),  \\
             \displaystyle     \sqrt{-1}^{-(2n-\alpha+2)} \frac{ \Gamma_{\mathbf C}(s+2n-\alpha+2) \Gamma_{\mathbf R}(n-\alpha+1)^2
                             }{   \Gamma_{\mathbf R}( -( n-\alpha)  )^2 },  &  (n-\alpha:\text{odd}).                           
           \end{cases}
\end{align*}
Therefore, by the definition of the modified Euler factor ${\mathcal E}_\infty( {\rm As}^+_{\mathcal M}(\pi)( \phi ) )$ in Section \ref{sec:CPconj},  
we deduce that 
\begin{align*}
  I (n-\alpha+1; W_\pi, \Phi, \varphi)
  = (-1)^n \sqrt{-1}^{\alpha}  2^{2r_F}D^{ -\frac{1}{2}( n-\alpha ) }_{E/F}       \binom{n}{\alpha}^2
       {\mathcal E}_\infty( {\rm As}^+_{\mathcal M}(\pi)( \phi ) )     L_\infty( 0, {\rm As}^+_{\mathcal M}(\pi)( \phi ) ). 
\end{align*}
This proves the proposition.
\end{proof}

\section*{Acknowledgements} 

The author is sincerely grateful for Ming-Lun Hsieh.  
The discussion with him on this subject was very helpful for preparing this paper.
The author also thanks to Tadashi Ochiai for valuable comments. 
The author was supported by JSPS Grant-in-Aid for Young Scientists (B) Grand Number JP17K14174.

\end{document}